\documentclass[11pt]{amsart}

\usepackage[margin=1in]{geometry}  
\usepackage{graphicx}              
\usepackage{amsmath}               
\usepackage{amsfonts}              
\usepackage{amsthm}                
\usepackage{color}
\usepackage{tabulary}
\usepackage{caption}
\usepackage{subcaption}
\usepackage{amssymb}
\usepackage{todonotes}
\theoremstyle{plain}
\usepackage{mathptmx}
\usepackage{tikz-cd} 
\usepackage{mathrsfs}

\newtheorem{thm}{Theorem}[section]

\newtheorem{lem}[thm]{Lemma}

\newtheorem{prop}[thm]{Proposition}
\newtheorem{cor}[thm]{Corollary}
\newtheorem{conj}[thm]{Conjecture}
\newtheorem{defn}[thm]{Definition}

\theoremstyle{definition}

\DeclareMathOperator{\PR}{PSL_2(\RR)}

\DeclareMathOperator{\LL}{\mathcal{L}}
\DeclareMathOperator{\GG}{\mathscr{G}}

\DeclareMathOperator{\Homeo}{Homeo}
\DeclareMathOperator{\Homeop}{Homeo^+}

\DeclareMathOperator{\Fix}{Fix}

\newcommand{\NN}{\mathbb{N}}      %
\newcommand{\ZZ}{\mathbb{Z}}      
\newcommand{\RR}{\mathbb{R}}      
\newcommand{\CC}{\mathbb{C}}      
\newcommand{\HH}{\mathbb{H}}      

\begin{document}

\title[Laminar groups]{Laminar groups and 3-manifolds}

\author{Hyungryul Baik \& KyeongRo Kim}
\address{Department of Mathematical Sciences, KAIST,  
291 Daehak-ro, Yuseong-gu, Daejeon 34141, South Korea }
\email{hrbaik@kaist.ac.kr \& cantor14@kaist.ac.kr}

\begin{abstract}
Thurston showed that the fundamental group of a close atoroidal
3-manifold admitting a co-oriented taut foliation acts faithfully on
the circle by orientation-preserving homeomorphisms. This action on the circle is
called a universal circle action due to its rich information. In this
article, we first review Thurston's theory of universal circles and
follow-up work of other authors. We note that the universal
circle action of a 3-manifold group always admits an invariant
lamination. A group acting on the circle with an invariant lamination
is called a laminar group. In the second half of the paper, we discuss
the theory of laminar groups and prove some interesting properties of
laminar groups under various conditions. 


\smallskip
\noindent \textbf{MSC classes:} 20F65, 20H10, 37C85, 37E10, 57M60. 
\end{abstract}

\maketitle

\section{Introduction}\label{intro}

A few years before Perelman came up with his proof of Poincar\'{e}
conjecture using the theory of Ricci flow \cite{perelman2002entropy}, \cite{perelman2003ricci} (built upon the work of Hamilton \cite{hamilton1995formation}), 
Thurston showed his vision to finish
the geometrization program using foliations in 3-manifolds in
\cite{Thurston97}. Although Thurston left the manuscript unfinished
after Perelman's resolution of geometrization conjecture,
\cite{Thurston97} contains abundant beautiful ideas which are closely related to
many interesting results by a number of authors including Ghys \cite{ghys1984flots}, 
Calegari-Dunfield \cite{CalDun03}, Calegari
\cite{calegari2006universal}, Fenley \cite{fenley2012ideal,
  fenley2016quasigeodesic}, Barbot-Fenley \cite{barbot2015free}, Gabai-Kazes \cite{gabai1997order, gabai1998group},
Mosher \cite{mosher1996laminations}, and Frankel
\cite{frankel2018coarse}. 

One of the main theme of \cite{Thurston97} is to combine a few approaches to 3-manifolds which are proven to be successful and fruitful. In particular, a deep connection between codimension-1 objects in 3-manifold and 3-manifold group actions on the low-dimensional spaces has been investigated. One of the main theorem in the paper is the following. 

\begin{thm}[Thurston's universal circle for co-orientable taut
  foliations \cite{Thurston97}]
\label{thm:univcirctautfol}
Let $M$ be a closed atoroidal 3-manifold admitting a co-orientable taut foliation $\mathcal{F}$. Then there exists a faithful homomorphism $\rho_{univ}: \pi_1(M) \to \Homeop(S^1)$. 
\end{thm} 

In fact, it is not just any group action on the circle. Thurston
called the circle obtained in the above theorem a \textsf{universal
  circle} for the taut foliation $\mathcal{F}$. Let's denote it by
$S^1_{univ}$. The name suggests that
$\rho_{univ}$ is not just an action but it ``sees'' the structure of
the foliation. In fact, a universal circle consists of following data:
\begin{itemize}
\item[(1)] Let $\widetilde{\mathcal{F}}$ be a covering foliation of
  $\mathcal{F}$ in the universal cover $\widetilde{M}$ of $M$. For each
  leaf $\lambda$ of $\widetilde{\mathcal{F}}$, there exists a circle
  $S^1_\infty(\lambda)$ so that the action of $\pi_1(M)$ on the leaves
  extend continuously on the set of such circle.
\item[(2)] For each leaf $\lambda$ of $\widetilde{\mathcal{F}}$, there
  exists a monotone map $\phi_\lambda: S^1_{univ} \to
  S^1_\infty(\lambda)$, i.e., a continuous surjection so that the
  preimage of each point in the range is connected.
 \item[(3)] For each $\alpha \in \pi_1(M)$ and a leaf $\lambda$, the
   following diagram commutes:
   \[ \begin{tikzcd}   
 S^1_{univ} \arrow{r}{ \rho_{univ}(\alpha) } \arrow[swap]{d}{ \phi_\lambda } &
 S^1_{univ} \arrow{d}{ \phi_{\alpha(\lambda)} } \\%
S^1_\infty(\lambda) \arrow{r}{\alpha }& S^1_\infty(\alpha(\lambda))
\end{tikzcd}
\]
\item[(4)] (comparability condition) For each leaf $\lambda$ of
  $\widetilde{\mathcal{F}}$, the maximal connected intervals in
  $S^1_{univ}$ which are mapped to points by $\phi_\lambda$ are called 
  the gaps associated to $\lambda$ and the complement of the gaps is
  called the core associated to $\lambda$. For any two incomparable
  leaves $\mu, \lambda$, the core associated to $\mu$ is contained in
  a single gap associated to $\lambda$ and vice versa.  
\end{itemize}

For the construction of the universal circle, we borrow materials
largely from \cite{CalDun03}, so for the interested readers, please
consult \cite{CalDun03} for details. Here we recall main ingredients
and rough ideas to see the big picture. As we will see in the
construction, there are some choices involved and as a result, a
universal circle is not unique. Perhaps coming up with a canonical way
of obtaining a universal circle via some universal property would be
desirable. 

Many results analogous to Theorem \ref{thm:univcirctautfol} have been
obtain in the literature under the presence of other codimension-1
objects or flows in the 3-manifold. For instance, 
Calegari obtained the result
for 3-manifold with quasi-geodesic flow \cite{calegari2006universal},
and Calegari and Dunfield showed this result in the case for essential laminations with solid
torus guts \cite{CalDun03}. Later Hamm generalized Calegari-Dunfield's
work to more general class of essential laminations in his PhD thesis
\cite{hamm2009filling}. 

In Sections 2-4,  we briefly review these works. In
Section \ref{sec:invlamination}, we observe that in all
those cases, the action on the circle comes with an invariant
lamination. This motives the study on groups acting on the circle with
invariant laminations (and such groups are called laminar groups). In
Sections 6-11, we discuss some recent and on-going
work in the theory of laminar groups. 
We emphasize that by no means the review of the material in the literature in Sections 2-5 can serve as a thorough survey for all related work.

\subsection*{Acknowledgments} 
We thank Michele Triestino, Steven Boyer, Sanghyun Kim, Thierry Barbot and Michel Boileau for fruitful conversations. We would also like to thank Hongtaek Jung for helpful discussions especially regarding Section $9$.
The first half of the paper were written based on the mini-course the first author gave in the workshop `Low dimensional actions of 3-manifold groups' at
Universit\'{e} de Bourgogne in November, 2019. 
The first author was partially supported by Samsung Science \& Technology Foundation grant No. SSTF-BA1702-01, and 
the second author was partially supported by the Mid-Career Researcher Program (2018R1A2B6004003) through the National Research Foundation funded by the government of Korea.

\section{$S^1$-bundle over the leaf
  space}
\label{sec:circlebundle}

Let $\mathcal{F}$ be a co-oriented
taut foliation in a 3-manifold $M$. Let $\widetilde{\mathcal{F}}$ be
the foliation in the universal cover $\widetilde{M}$ of $M$ which
covers $\mathcal{F}$, and let $L = L(\widetilde{\mathcal{F}})$ be the
leaf space of this covering foliation. As a set, each point of $L$ corresponds to a leaf of $\widetilde{\mathcal{F}}$. 
To give a topology, we say a sequence of leaves $\mu_i$ converges to a leaf $\mu_\infty$ if for every compact subset $K$ of $\widetilde{M}$, 
$\mu_i \cap K$ converges to $\mu_\infty \cap K$ in the Hausdorff topology. 

The leaf space $L$ is an 1-dimensional manifold in the sense that each point has a neighborhood 
homeomorphic to $\mathbb{R}$ but $L$ does not have to be Hausdorff.
In fact, the leaf space $L$ is Hausdorff if and only if it is homeomorphic to $\mathbb{R}$. In that case, we say $\mathcal{F}$ is an $\mathbb{R}$-covered foliation. 

In all other cases, $L$ is not Hausdorff. The co-orientation of $\mathcal{F}$ gives an orientation on each embedded line segment of the leaf space. 
Therefore, it induces a partial order on the leaf space $L$. For two leaves $\alpha, \beta$ of $\widetilde{\mathcal{F}}$, 
we say $\alpha < \beta$ if there exist an embedded closed interval in
$L$ whose end points are $\alpha, \beta$, and it is an oriented path from $\alpha$ to $\beta$ 
with respect to the induced orientation. One caveat is that we need to know there exists no closed transversal to $\widetilde{\mathcal{F}}$ which 
will be shown later in this section. 
$\mathcal{F}$ is $\mathbb{R}$-covered if and only if the induced partial order on $L$ is a total order. 
In general, if $\mathcal{F}$ is not $\mathbb{R}$-covered, there are non-comparable leaves.  

We say $\mathcal{F}$ is branched in the forward direction if there are
three leaves $\alpha, \beta, \gamma$ of $\widetilde{\mathcal{F}}$ such
that $\alpha, \beta$ are non-comparable but $\alpha > \gamma$ and
$\beta > \gamma$.
Similarly, one can define a branching in the backward direction. In
the paper, when we think of a non $\mathbb{R}$-covered foliation, we
only consider the case $\mathcal{F}$ has two-sided branching i.e., it
is branched in both forward and backward direction for simplicity.
For what we discuss in this section, this assumption is not so
relevant. 

We would like to construct what can be called a $S^1$-bundle over
$L(\widetilde{\mathcal{F}})$ in some sense. In other words, we would like to
assign one copy of the circle to each leaf, but where does it come
from? To begin with, we recall the result of Candel. 

In general, for a manifold $M$ with dimension $n \geq 3$, a 2-dimensional lamination is called a Riemann surface lamination if each leaf is a Riemann surface. 
More precisely, suppose $M$ admits
an atlas with product charts $\phi_p
: U_p \to B_p \times K_p$ where $B_p$ is a domain in $\mathbb{C}$, $K_p$ is 
a closed subspace of $\mathbb{R}^{n-2}$, $U_p$ is an open subset of $M$, and $\phi_p$ is a homeomorphism. We further assume that 
each change of the coordinates have
the form $\phi_p \circ
\phi_q^{-1}(b, k) = (\psi(b, k),  \rho(k))$ where 
$\psi, \rho$ are continuous functions and $\psi$ is holomorphic in $b$. Such an atlas $\Lambda$ is called a Riemann surface lamination. 
We will focus in the case $M$ is a
3-manifold, and $\Lambda$ is a
surface lamination in $M$. In fact,
we assume $M$ to be a closed
hyperbolic 3-manifold throughout of
the rest of the paper. Also,
$B_p$ is always taken to be the unit
disk $D$. Hence we consider the
product charts $U_i = D \times K_i$.

Candel obtained a significant generalization of the classical uniformization theorem for Riemann surfaces in the setting of Riemann surface laminations. 
In particular, this provides a sufficient condition for $(M, \Lambda)$ to admits a Riemannian metric so that its restriction to $\Lambda$ is 
a leaf-wise hyperbolic metric. We only recall main ideas and for
detail of Candel's work, we refer the readers to
\cite{candel1993uniformization} or \cite{Calebook}. 

The classical uniformization theorem says that if a closed Riemann surface has a negative Euler characteristic, then it admits a hyperbolic metric. 
To state a similar result for laminations, we need to develop a notion
which plays a role similar to the Euler characteristic. To do this, we
first need to discuss invariant transverse measures on laminations. An
invariant transverse measure $\mu$ for a lamination $\Lambda$ is a
collection of nonnegative Borel measure on the leaf space of $\Lambda$
in each product chart which is compatible on the overlap of distinct
charts. 

Now when $\Lambda$ is a Riemann
surface lamination, the leafwise
metric determines a leafwise closed
2-form, say $\Omega$. The product
measure $\mu \times \Omega$ is a
signed Borel measure on the total
space $\Lambda$. We call the total
mass of this measure the Euler
characteristic $\chi(\mu)$ of
$\mu$. As in the case of the
classical uniformization theorem,
the sign of the Euler characteristic
is important.

Note that if $U = D \times K$ is a product chart
for $\Lambda$, then
$(\mu \times \Omega) (U) = \int_K
(\int_{D \times k} \Omega)
d\mu(k)$. When $\Lambda$ admits
a leafwise hyperbolic metric, then 
$\int_{D \times k} \Omega$ is
negative and $\mu$ is a nonnegative
measure by definition, hence 
$(\mu \times \Omega) (U)$ is
negative for each product chart
$U$. As a consequence, we have
$\chi(\mu) < 0$. What Candel proved
is that the converse is also true. 

\begin{thm}[Candel's uniformization theorem \cite{candel1993uniformization}] 
Let $\Lambda$ be a Riemann surface
lamination. Then $\Lambda$ admits a
leafwise hyperbolic metric if and
only if the Euler characteristic
$\chi(\mu)$ is negative for all
nontrivial invariant transverse
measure $\mu$. 
\end{thm} 

Let's go back to our case: $M$ is a
closed hyperbolic 3-manifold and
$\mathcal{F}$ is a co-orientable
taut foliation.
First, we observe that no leaf of $\mathcal{F}$
is the 2-sphere $S^2$. This follows from the Reeb stability theorem.

\begin{thm}[Reeb stability theorem]
 Let $\mathcal{F}$ be a cooriented taut foliation in a closed 3-manifold
 $M$.  Suppose $\mathcal{F}$ has a leaf homeomorphic to $S^2$. Then
 $M$ is homeomorphic to $S^2 \times S^1$ and $\mathcal{F}$ is the
 product foliation by the spheres. 
\end{thm}
\begin{proof}[Sketch of the proof]
Since $\pi_1(S^2)$ is trivial, holonomy along any path on the
spherical leaf is trivial. Therefore, the spherical leaf has a
neighborhood which is foliated as a product. This shows that the set
of spherical leaves form an open subset of $M$.

Since $M$ is compact,
We know that if we have a sequence of closed leaves $\lambda_i$ which
converge to a leaf $\lambda$, then $\lambda$ is also closed. If all
$\lambda_i$ are spheres, then in a small neighborhood of $\lambda$,
the projection along the vertical direction in each product chart
defines a covering map from $\lambda_i$ to $\lambda$ for large enough
$i$. Since $\mathcal{F}$ is co-oriented, $\lambda$ is also necessarily
a sphere. Therefore, the set of spherical leaves form a closed subset
of $M$. Since the set is both open and closed, it should be $M$ itself. 
\end{proof} 

Since $M$ is assumed to be hyperbolic in our case, we do not have any
spherical leaf. 

Also no leaf is a torus. Since
$M$ is atoroidal, if any leaf is a
torus, then it would bound a solid
torus. One can foliate the solid
torus where the boundary is also a
leaf, and it is called a Reeb
component. First, one can foliate 
$H = \{ (x, y, z) \in \mathbb{R}^3 :
z \geq 0\}$ by the horizontal
planes $\{(x, y, z): z = c\}$.  
Quotient $H \setminus \{(0,0,0)\}$ by the equivalence
relation $(x, y, z) \sim (2x, 2y,
2x)$. In this case, one can easily
see that if a transversal in $M$
travels from the complement of the
Reeb component into the Reeb component
by passing through the boundary leaf
of the Reeb component (the torus
leaf), it cannot escape the Reeb
component again. Hence, if
$\mathcal{F}$ is a taut foliation,
it cannot have a Reeb component. 

From this, one can conclude that each leaf of
$\mathcal{F}$ is of hyperbolic
type. Therefore, the condition of
Candel's theorem is satisfied, and
$M$ admits a leafwise hyperbolic
metric. 

By a work of Rosenberg \cite{rosenberg1968foliations} which is an important improvement of the
seminal work of Novikov \cite{novikov1965topology}, we know the followings about $M$ and
$\mathcal{F}$:
\begin{itemize}
\item[(i)] $M$ is irreducible.
\item[(ii)] For each leaf $\alpha$ of $\mathcal{F}$, the inclusion map
  $\alpha \mapsto M$ induces an injective homomorphism $\pi_1(\alpha)
  \to \pi_1(M)$. 
  \item[(iii)] Every closed transversal to $\mathcal{F}$ is nontrivial
    in $\pi_1(M)$. 
\end{itemize} 

Here we can immediately see that the leaf space $L$ is a tree in the sense that 
there is no cycle embedded in $L$. If there exists such a cycle, 
it corresponds to a closed transversal to $\widetilde{\mathcal{F}}$, so it projects down to 
a closed transversal to $\mathcal{F}$. Then it must be nontrivial in $\pi_1(M)$ while it lifts to a loop 
in the universal cover of $M$, a contradiction.

From this result, one can deduce the following theorem immediately. 
\begin{thm}
\label{thm:properlyembedded}
  Let $M, \mathcal{F}$ be as above. Then every leaf of $\widetilde{\mathcal{F}}$ is a
  properly embedded plane in $\widetilde{M}$. 
\end{thm}

\begin{proof}[Sketch of a proof] 
First of all, every leaf of $\widetilde{\mathcal{F}}$ is simply connected. Let $\widetilde{\lambda}$ be a leaf of $\widetilde{\mathcal{F}}$,
and $\gamma$ be a loop on $\widetilde{\lambda}$. Note that $\lambda$ is a covering of some leaf $\lambda$ of $\mathcal{F}$. 

Since $\widetilde{M}$ is simply connected, it is homotopically trivial in $\widetilde{M}$,
so it must be homotopically trivial in $M$. On the other hand, by the theorem of Rosenberg-Novikov above, $\lambda$ is $\pi_1$-injectively embedded in $M$. 
Thus, it must be trivial in $\lambda$. By the homotopy lifting property, this implies that the original loop $\gamma$ is homotopically trivial in $\widetilde{\lambda}$. 
Since $\gamma$ is arbitrary, this implies that $\widetilde{\lambda}$ is simply connected. 

Now by the Reeb stability theorem, no leaf is a sphere. Hence all leaves of $\widetilde{\mathcal{F}}$ must be planes. 
For a leaf $\lambda$ of $\widetilde{\mathcal{F}}$, if it is covered by product charts so
that in each chart, the intersection with $\lambda$ is connected (each
connected component is called a plaque), then it must be properly
embedded. Therefore, if $\lambda$ is not properly embedded, there
exists a product chart where $\lambda$ intersects in at least two
plaques. In that case, one can make a closed loop in $\widetilde{M}$
such that first use the transversal in that product chart to connect
two points in different plaques of $\lambda$, and close it up by a
path contained in $\lambda$. Now this path in $\lambda$ is covered by
finitely many product charts, so one can tilt it to get a transversal
which is very close to the original path (see Figure
\ref{fig:tilting}. In our case, the charts $U_1$ and $U_n$ could
coincide). Using this technique, one
gets a closed transversal $\widetilde{\gamma}$ to $\widetilde{\mathcal{F}}$ which
intersects $\lambda$. It gets mapped to a closed transversal $\gamma$ in
$\mathcal{M}$ and by Part $(iii)$ of Novikow-Rosenburg theorem above,
$\gamma$ must be homotopically nontrivial. On the other hand, since
$\widetilde{M}$ is simply connected, $\widetilde{\gamma}$ is
homotopically trivial, a contradiction. We conclude that every leaf is
properly embedded.
\end{proof}

\begin{figure}
  \includegraphics[width=15cm]{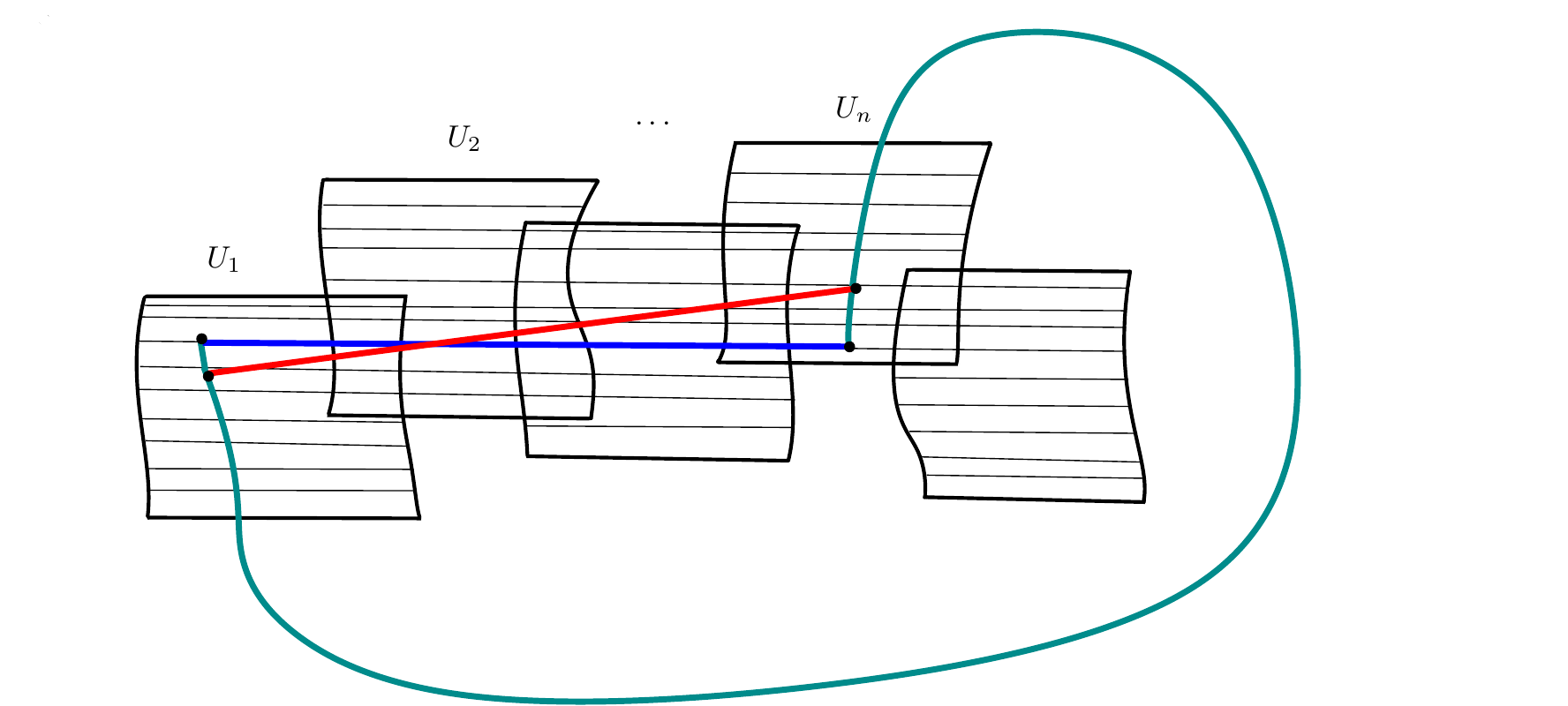}
  \caption{Consider the loop obtained by
    concatenating the blue arc which is contained in
    a leaf of the foliation with the green arc which is assumed to be
    transverse to the foliation.
    The blue arc is tilted to the red arc to make the whole loop transverse
  to the foliation. Note that one cannot draw the green arc so that
it intersects the chart $U_n$ from below, since it would contradict to
the fact that the foliation is co-oriented.}
  \label{fig:tilting}
\end{figure}

Combining this result with Candel's theorem, we find a metric on $M$
so that each leaf of $\widetilde{\mathcal{F}}$ equipped with the
induced path metric is isometric to the hyperbolic plane
$\mathbb{H}^2$. For each leaf $\lambda$ of $\widetilde{\mathcal{F}}$, since
$\lambda$ can be identified with $\mathbb{H}^2$ and the ideal boundary
of $\mathbb{H}^2$ is homeomorphic to the circle (called the circle at
infinity), we get the circle at infinity $S^1_{\infty}(\lambda)$ for
$\lambda$. Now we define the circle bundle at infinity$E_\infty$ as
the set of all circles at infinity for the leaves of
$\widetilde{\mathcal{F}}$. In other words, 
$E_\infty  = \cup_{\lambda \in L} S^1_\infty(\lambda)$. $E_\infty$ can be obtained from the ``cylinders'' over each transverse arc to
$\widetilde{\mathcal{F}}$ by patching them together
appropriately. We explain what this means in the next section.

\section{Leaf pocket theorem and the special sections}
\label{sec:specialsections}

Now we have circles, one for each leaf of
$\widetilde{\mathcal{F}}$. We need to combine them to make one big
mother circle which we will call a universal circle. This is done as
follows: in the last section, we defined $E_\infty$ as a set, so we
first give a description of its topology. Second, we note that there are some special sections for the
bundle $E_\infty$ which are preserved under the deck group action on
$\widetilde{M}$. Third, we observe that they can be circularly
ordered so that the deck group action is order-preserving. Finally,
taking a order completion of the set of special sections, we get a
circle.

To do this, we need to understand both ``tangential geometry'' and
``transverse geometry'' of $\mathcal{F}$. For the tangential geometry,
here is one useful lemma.
\begin{lem}
  There exists $\epsilon > 0$ such that every leaf of
  $\widetilde{\mathcal{F}}$ is quasi-isometrically embedded in its
  $\epsilon$-neighborhood. 
 \end{lem} 
 \begin{proof} For each point $p$ in $M$, consider a product chart
   $U_p$ which is evenly covered by the universal covering map so that
   one connected component of the preimage of $U_p$ is a product chart
   where each leaf of $\widetilde{\mathcal{F}}$ intersects at most
   once. The last condition can be satisfied by the reason in the
   proof of Theorem \ref{thm:properlyembedded}. 

   By compactness of $M$, we cover $M$ with finitely many those
   product charts $U_1, \ldots, U_n$. Again since $M$ is compact, we
   can apply the Lebesgue number lemma to conclude that there exists
   $\epsilon > 0$ such that every ball of radius $2\epsilon$ is
   contained in one of the product charts $U_i$.

   Now let $\lambda$ be any leaf of $\widetilde{\mathcal{F}}$, and let
   $N$ be the $\epsilon$-neighborhood of $\lambda$. By our choice of
   $\epsilon$, lifts of the product charts $U_i$ cover entire
   $N$. Since these are lifts of finitely many product charts, they
   have uniformly bounded geometry. This shows that $\lambda$ is
   quasi-isometrically embedded in $N$. 
 \end{proof}

A positive number $\epsilon$ as in the above lemma is called a
separation constant of $\mathcal{F}$. 
 
The transverse geometry of $\mathcal{F}$ is described in so-called the leaf
pocket theorem. To state the theorem, we first need to define the
endpoint map. Let $\lambda$ be a leaf of $\widetilde{\mathcal{F}}$,
and $p$ be a point on it. Then for any vector $u$ in the unit tangent space $UT_p \lambda$
at $p$, let $e(u)$ be the endpoint in $S^1_\infty (\lambda)$ of the
geodesic ray in $\lambda$ determined by $u$. This defines a map, again
we call it $e$, from
the unit tangent bundle $UT \widetilde{\mathcal{F}}$ of
$\widetilde{\mathcal{F}}$ to $E_\infty$. Now we give $E_\infty$ the
finest topology so that the map $e: UT \widetilde{\mathcal{F}} \to
E_\infty$ is continuous.

Now we explain what we meant by ``patching cylinders'' in the last
section. Let $\tau$ be any transverse arc to
$\widetilde{\mathcal{F}}$. Then $UT \widetilde{\mathcal{F}}_\tau$ is
literally a cylinder. If $e(v_1) = e(v_2)$ for $v_1 \in UT
\widetilde{\mathcal{F}}_{\tau_1}$ and $v_2 \in UT
\widetilde{\mathcal{F}}_{\tau_2}$ for two transverse arcs $\tau_1,
\tau_2$, then we identify $v_1$ and $v_2$. Hence $E_\infty$ is
obtained from the disjoint union of cylinders of the form  $UT
\widetilde{\mathcal{F}}_\tau$ under these 
identifications. 

Going back to the transverse geometry of the foliation, we call a map $m: I \times \mathbb{R}_{\geq 0} \to \widetilde{M}$ a
\textsf{marker} if $m(\{k\} \times \mathbb{R}_{\geq 0})$ is a geodesic ray in
a single leaf of $\widetilde{\mathcal{F}}$ for each $k \in I$ and
$m(I \times \{t\})$ is a
transverse arc with length no greater than $\epsilon/3$ for all $t \in
\mathbb{R}_{\geq 0}$ where $\epsilon$ is a separation constant for
$\mathcal{F}$. 

Let $p \in \widetilde{M}$ and $\lambda$ be leaf
of  $\widetilde{\mathcal{F}}$ containing $p$. Suppose there exists a marker
$m$ such that $p = m((k, 0))$ for some $k \in I$.
This means that there exists a transversal $m(I \times \{0\})$ at $p$,
the holonomy along the geodesic ray $m(\{k\} \times
\mathbb{R}_{\geq 0} )$ emanating from $p$ is defined for the whole
time. Said differently, along this ray, nearby leaves are not pulled
away from the leaf $\lambda$ too fast. The following theorem of
Thurston shows that for arbitrary $p \in \widetilde{M}$, there exists
abundant of directions with this property. This describes the
transverse geometry of $\mathcal{F}$. 

Original proof of the leaf pocket theorem given by Thurston
in \cite{Thurston97} uses the existence of harmonic measures for
foliations. An alternative, purely topological proof is given by
Calegari-Dunfield \cite{CalDun03}. We omit the discussion of the proof
here and only briefly explain how this theorem is applied to get a set
of cyclically ordered set of sections. 

\begin{thm}[Leaf pocket theorem \cite{Thurston97}, \cite{CalDun03}]
For every leaf $\lambda$ of $\widetilde{\mathcal{F}}$, the set of
endpoints of markers is dense in $S^1_{\infty}(\lambda)$. 
\end{thm}

Abusing the notation, we also call the set of endpoints of a marker a
marker. Let $C$ be a cylinder in $E_\infty$. In other words, if $I$ is an
interval in $L$, then $C$ is a cylinder foliated by circles at infinities for
the leaves corresponding to points in $I$. First thing to observe is
that no two markers contained in $C$ are either disjoint or their
union is an interval transverse to the circle fibers in $C$. 

This is actually a consequence of the tangential geometry of
$\mathcal{F}$ (more precisely the existence of a separation constant
$\epsilon$). Suppose two markers $m_1, m_2$ intersect at a point in
$S^1_\infty(\lambda)$ but have distinct endpoints on $S^1_\infty(\mu)$
for some leaves $\lambda, \mu \in I$. On $\lambda$, the geodesic rays
of $m_i$'s become arbitrarily close to each other, since they have the
same endpoint on the ideal boundary. Hence, by shortening the markers
horizontally, we may assume that they are within $\epsilon/3$-distance
from each other on $\lambda$ with respect to the metric on
$\widetilde{M}$.
Since each marker is $\epsilon/3$-thin, the geodesic rays of $m_i$'s
on $\mu$ are within $\epsilon$-distance  from each other again with
respect to the metric on $\widetilde{M}$. However, those rays diverge
on $\mu$, hence with respect to the hyperbolic metric on $\mu$, the
rays get arbitrarily far away from each other. This contradicts to the
fact that $\mu$ is quasi-isometrically embedded in its
$\epsilon$-neighborhood. 

From this fact together with the leaf pocket theorem, we can start
constructing sections on $C$. First, pick a set $T$ of finitely many markers
on $C$ so that each non-boundary circle fiber of $C$ intersects at least one
marker at an interior point of the marker, and the boundary circle
fibers meet at least one marker at the endpoint. 

To make our description simple, let's parametrize $I$ (recall that $C$
is a circle bundle over an interval $I$ in $L$) to be the closed
interval [0, 1], and the leaf corresponding to point $t \in [0,1]$ is
denoted by $\lambda_t$.  
Let $p \in S^1_\infty(\lambda_t) \subset C$ for some $t$.  We can choose a
``left-most'' path through $p$ with respect to $T$ in the following way: 
On $ S^1_\infty(\lambda_t)$, we start from $p$ and move anti-clockwise
until we hit a marker. At the marker, follow the marker upward
(increasing the parameter $t$) until the end of the marker. At the
end, move anti-clockwise as much as one can until one hits another
marker. Follow the marker upward. That way we construct a path from
$p$ to $S^1_\infty(\lambda_1)$.

Let call this path $\gamma_{p,T}$. 
Now make the set $T$ bigger by adding more markers on $C$ to get a new
set $T'$ of markers. If new markers do not intersect the path
$\gamma_{p,T}$, there is nothing to do in the sense that $\gamma_{p,
  T} = \gamma_{p, T'}$. Hence let's assume that a new
marker $m$ intersects the path $\gamma_{p,T}$. This means that at some
$t$, $\gamma_{p, T}$ moves horizontally on $S^1_\infty(\lambda_t)$ but
  the marker $m$ crosses it vertically. Hence, when we formulate the
  path with respect to the set $T \cup \{m\}$ of markers, our path
  should stop at $m \cap S^1_\infty(\lambda_t)$ and follow $m$ upward,
  and then move horizontally anti-clockwise again until hitting other
  markers in the set. Then one can observe that the path $\gamma_{p,
    T'}$ is slightly perturbed to the right compared to $\gamma_{p,
    T}$. To make this more precise, one can unwrap the circle fibered
  of $C$ to the real line $\mathbb{R}$ to get a simply-connected cover
  of $C$ which is now foliated by horizontal lines. Here we see this
  cover so that on each line fiber, moving to the left corresponds to moving
  anti-clockwise on a circle fiber on $C$. Then clearly the new path $\gamma_{p,
    T'}$ is on the right compared to $\gamma_{p,
    T}$ (here one should fix a lift $\widetilde{p}$ of $p$ and consider the
  lifts of the paths passing through $p'$). An important point is that
  they cannot cross each other, although they are likely to coalesce. 
  
  Now for any two paths $\gamma$ and $\delta$ on $C$, we say $\gamma \leq \delta$ if 
  the lift of $\delta$ through $\widetilde{p}$ is on the right side to the lift of $\gamma$ through $\widetilde{p}$ in the universal cover of $C$. 
  This gives a partial ordering on the set of paths on $C$.  
 For any two sets of markers $T \subset T'$, we get $\gamma_{p, T} \leq \gamma_{p, T'}$.  
 Now we define a section $\tau_p: I \to C$ by $\tau_p(\nu) = \sup \{ \min \gamma_{p, T} \cap S^1_\infty(\nu) : T \mbox { is a set of markers.} \}$. 
 Here the minimum means the projection of the left-most point in the universal cover of $C$, and supremum exists 
because the lifts of paths $\gamma_{p, T}$ to the cover of $C$ through $\widetilde{p}$ are bounded from above by 
the vertical line through $\widetilde{p}$. 
This new path $\tau_p$ is continuous since the set of markers meet each circle fiber
  at a dense subset. Consequently, we get a continuous section
  $\tau_p$ of the circle-bundle $C$ over $I$ and call it a left-most
  section starting from $p$. 

Starting from $p$, one can also move downward in the leaf space $L$. In this case, instead of using the left-most paths, we take right-most paths 
by moving clockwise on each circle fiber and following markers downward. This is called a right-most section starting from $p$. Hence, for each embedded line $A$ in $L$, 
one can get a section $\tau_p$ of the bundle $E_\infty|_A$ over $A$ by taking left-most section when we move upward from $p$ along $A$, and taking right-most section when we move downward from $p$ along $A$. But we would like to extend $\tau_p$ as a section for the bundle $E_\infty \to L$. 

Before we proceed, we need one definition. Consider a sequence $(\mu_i)$ of
leaves of $\widetilde{\mathcal{F}}$ which are contained in a single
totally ordered segment of $L$ and increasing with respect to that
order. We say such a sequence \textsf{monotone ordered}. Suppose there
exists a collection of leaves $\{\lambda_j\}$ of
$\widetilde{\mathcal{F}}$ such that $\mu_i$ converges on compact
subsets of $\widetilde{M}$ in the Hausdorff topology to the union of
leaves $\lambda$, then we call the collection $\{\lambda_j\}$ together
with the monotone ordered sequence  $(\mu_i)$ a \textsf{cataclysm}. Here
the convergence means that for any compact subset $K$ of $M$, $\mu_i
\cap K$ converges to $(\cup_j \lambda_j) \cap K$ in the Hausdorff
topology. In fact, it is more appropriate to consider the cataclysms up to some
natural equivalence relation because the sequence $(\mu_i)$ are not an
essential part of the data. So as long as we have two monotone ordered
sequences  contained in a single totally ordered segment of $L$ which
converges to the same collection of leaves $\{\lambda_j\}$, we say
those two cataclysms are equivalent. Abusing the notation, we will
just call the collection $\{\lambda_j\}$ a cataclysm.

Let $\lambda$ be a leaf so that $p \in S^1_\infty(\lambda)$ and let $\mu$ be any other leaf in $L$. There exists a unique broken path from $\lambda$ in $\mu$ which is obtained in the following way: first collapse each cataclysm to a point in $L$ to get an actual tree $Y$, take the unique path from $\lambda$ to $\mu$ in $Y$, and pull back it to $L$.  This broken path is a union of embedded intervals in $L$ with occasional jumps between two leaves in the same cataclysm. 
Say, in this broken path, $\tau_p$ comes down to $\lambda_1$ and it jumps to $\lambda_2$ which is in the same cataclysm with $\lambda_1$ and then move upward from there. 
Say $\mu_i$ is a monotone ordered sequence converging to $\lambda_1$ and $\lambda_2$. 

Suppose $I_1, I_2$ are two intervals in $L$ such that they coincide in an half-open interval $I$ and differ
at only one vertex, $\mu_i$ are in $I$, and $I_i = I \cup \{\lambda_i\}$ for $i = 1, 2$. 
 For each $i$, let $m_i, m_i'$ be any two markers so that they have one endpoint on $S^1_\infty(\lambda)$
and the rest lie in the circle-bundle $C$ over $I$. For later use,
let's call the circle-bundle over $I_i$ $C_i$ for each $i$. 
First note that $m_1$ and $m_2$ are disjoint on $C$. Otherwise, since
they are $\epsilon/3$-thin, again we get a contradiction to the fact
that $\epsilon$ is a separation constant for $\mathcal{F}$.
Also, for each $\mu_j$ which intersects all the markers $m_1, m_1',
m_2, m_2'$, the pairs $(m_1, m_1')$ and $(m_2, m_2')$ are unlinked in
the circle $S^1_\infty(\mu_j)$. If they are linked, since $\lambda_1$
gets close to $\mu_j$ near the pair $(m_1, m_1')$ and $\lambda_2$
gets close to $\mu_j$ near the pair $(m_2, m_2')$, either $\lambda_1$ and
$\lambda_2$ are comparable in $L$ or they must intersect. We know
$\lambda_1$ and $\lambda_2$ are incomparable, so this is impossible.
Consequently, one can take disjoint arcs $J_1, J_2$ of
$S^1_\infty(\mu_j)$ so that the set of endpoints of the markers in
$C_i$ on $S^1_\infty(\mu_j)$  are completely contained in $J_i$. 

Let $S^1_{\lambda_1 \lambda_2}$ be the circle obtained from $S^1_\infty(\mu_i)$ 
by collapsing each connected component of the complement of the closure of the set of intersections with the markers through either $\lambda_1$ or $\lambda_2$. 
Then for each $i$, naturally there exists a monotone map $\phi_i : S^1_{\lambda_1 \lambda_2} \to S^1_\infty(\lambda_i) $. 
For instance, $\phi_1$ collapses the arc obtained from the image of the closure of the set of intersections with the markers through $\lambda_2$ under the monotone map 
$S^1_\infty(\mu_i) \to S^1_{\lambda_1 \lambda_2}$, and similarly for $\phi_2$. Then the preimage of $\tau_p(\lambda_1)$ under $\phi_1$ gets mapped to a single point in $S^1_\infty(\lambda_2$ 
via $\phi_2$. Let this point be $\tau_p(\lambda_2)$. We continue by constructing a left-most section starting at $\tau_p(\lambda_2)$. 
This procedure allows us to construct $\tau_p$ along the broken path from $\lambda$ to $\mu$, therefore we get a well-defined value $\tau_p(\mu)$. 
We call a section for $E_\infty \to L$ a \textsf{special section} if it is $\tau_p$ for some $p \in E_\infty$ and constructed as above. 

Let $\mathcal{S}$ be the set of all special sections. First of all, they are built upon the set of markers which is preserved under the $\pi_1(M)$-action, since the markers are constructed using the geometry of the foliation. One can also check easily that $\mathcal{S}$ admits a natural cyclic order. 
For a triple $(\tau_{p_1}, \tau_{p_2}, \tau_{p_3})$, there must exists $\mu \in L$ so that $\tau_{p_1}(\mu), \tau_{p_2}(\mu), \tau_{p_3}(\mu)$ are distinct. 
Hence they inherit a cyclic order from the orientation on $S^1_\infty(\mu)$. Clearly this cyclic order is preserved by the $\pi_1(M)$-action, 
since the cyclic order on each cataclysm is determined by the geometry of the foliation as well. 
Of course we put many details under the rug, and for the full detail of the proof, see Section 6 of \cite{CalDun03}. 

By taking the completion of the set of special sections of
$E_\infty$ as an ordered set, one gets a circle $S^1_{univ}$ where $\pi_1(M)$ acts
by order-preserving homeomorphisms. 

Recall the definition of a universal circle given as a set of data in the introduction. We also need a monotone 
map $\phi_\lambda: S^1_{univ} \to S^1_\infty(\lambda)$ for each leaf $\lambda$ of $\widetilde{\mathcal{F}}$. 
For a point $p$ on $S^1_{univ}$ corresponding to a special section, $\phi_\lambda(p)$ is just the evaluation of the section at $\lambda$. 
From the construction, it is clear that
$\phi_\lambda$ is monotone, and commutativity of the diagram in the definition of the universal circle holds. Also, for incomparable leaves $\lambda_1, \lambda_2$, 
since the core of $\phi_{\lambda_1}$ is the closure of the points in $S^1_{univ}$ corresponding to the special sections through a point on $S^1_\infty(\lambda_1)$
and they are entirely collapsed to a single point in $S^1_\infty(\lambda_2)$ (recall the part where we constructed the circle $S^1_{\lambda_1 \lambda_2}$ above), 
the core of $\phi_{\lambda_1}$  must be contained in a gap of $\phi_{\lambda_2}$. This is actually contained in a single 
gap because the fact that the markers through $S^1_\infty(\lambda_1)$ are unlinked with the markers through $S^1_\infty(\lambda_2)$ implies that the same fact holds for special sections. 

One last thing to check is that the action on $S^1_{univ}$ is
faithful. In the case of $\mathbb{R}$-covered foliations, one can find
a transverse pseudo-Anosov flow, and in that case the faithfulness can
be verified using the ideas in \cite{calegari2000geometry}. See also
Section \ref{sec:flows} to see the detail of the pseudo-Anosov flow
case. Hence we concern only the case that $\mathcal{F}$ has
branching. Let $H$ be the kernel of the action $\rho_{univ} : \pi_1(M) \to
\Homeo{S^1_{univ}} $.

Suppose $H$ is nontrivial. 
Let $h$ be any nontrivial element of $H$ and let
$\lambda$ be any leaf of $\widetilde{\mathcal{F}}$.
We have the following commutative diagram:
   \[ \begin{tikzcd}   
 S^1_{univ} \arrow{r}{ \rho_{univ}(h) } \arrow[swap]{d}{ \phi_\lambda } &
 S^1_{univ} \arrow{d}{ \phi_{h(\lambda)} } \\%
S^1_\infty(\lambda) \arrow{r}{h }& S^1_\infty(h(\lambda))
\end{tikzcd}
\]
But $h$ acts trivially on the universal circle, the top map
$\rho_{univ}(h)$ is the identity map.

If $h(\lambda) =
\lambda$, then by the above diagram, we know that $h$ acts trivially on
$S^1_\infty(\lambda)$. This implies that $h$ acts on $\lambda$ as the
identity but it is impossible since $h$ is a nontrivial element of
$\pi_1(M)$. Hence we know $h(\lambda)$ is different from $\lambda$. 

Second, we observe that $\lambda$ and $h(\lambda)$ are
comparable. Suppose they are incomparable. By the commutativity of the
above diagram, any gap associated with $\lambda$ is contained in a gap
associated with $h(\lambda)$, but also the core associated with
$\lambda$ is contained in a single gap associated with $h(\lambda)$, a
contradiction. Therefore, the leaves $\lambda$ and $h(\lambda)$ are
comparable.

Let $\lambda, \mu$ be two distinct leaves contained in the same
cataclysm in $L$. From above discussion, $H \lambda$ is an infinite set
contained in a line $X$ of $L$, and similarly, $H \mu$ is an infinite 
set contained in a line $Y$ of $L$. For each $h \in H$, then
$h(\lambda)$ and $h(\mu)$ are two distinct leaves contained in the
same cataclysm. This shows that there exists infinitely many pairs of
points $(x, y) \in X \times Y$ such that $x$ and $y$ are contained in
the same cataclysm. But this is impossible for two lines in $L$, since
 there cannot be a nontrivial cycle in $L$. This is a contradiction, so
we conclude that $H$ must be trivial, i.e., the $\pi_1(M)$-action on
$S^1_{univ}$ is faithful.

\section{In the case of quasi-geodesic and pseudo-Anosov flows}
\label{sec:flows} 

Let $\mathfrak{F}$ be a flow in the closed hyperbolic 3-manifold
$M$. As we lifted a taut foliation in the 3-manifold to the covering
foliation of the universal cover, we can consider the lifted flow
$\widetilde{\mathfrak{F}}$ in the universal cover of $M$. 
We say $\mathfrak{F}$ is a quasi-geodesic flow if each flow line of
$\widetilde{\mathfrak{F}}$ is a quasi-geodesic in $\widetilde{M}$
which can be identified with the hyperbolic 3-space $\mathbb{H}^3$.

Pseudo-Anosov flows form another important class of flows in
3-manifolds. A flow $\mathcal{F}$ is pseudo-Anosov if it locally looks
like a branched covering of an Anosov flow. 

These two notions are closely related. First, Steven Frankel \cite{Frankelpreprint} announced
the resolution of Calegari's flow conjecture which says that any quasi-geodesic flow on a closed hyperbolic
3-manifold can be deformed to a flow that is both quasi-geodesic and
pseudo-Anosov. On the other hand, not every pseudo-Anosov flow is
quasi-geodesic. Fenely
\cite{fenley1994anosov} constructed a large class of Anosov flows in
hyperbolic 3-manifolds which are not quasi-geodesic. Later he 
gave a necessary and sufficient condition for a pseudo-Anosov flow to
be quasi-geodesic in \cite{fenley2016quasigeodesic}. These are optimal
results. 

Calegari \cite{calegari2006universal} showed that if $M$ admits a
quasi-geodesic flow, then $\pi_1(M)$ acts faithfully on the circle
where the circle is the boundary of the group-equivariant
compactification of the space of flow lines of the covering flow
$\widetilde{\mathfrak{F}}$. In some sense, the work of Ghys in \cite{ghys1984flots} is a prototype of the result of Calegari. 
Roughly speaking, Ghys proved that the leaf space of the weak stable foliation of an Anosov flow on a circle bundle is a line, 
and established a map from the leaf space into the circle. 

On the other hand, Calegari-Dunfield \cite{CalDun03} 
showed the same result in the case $M$ admits a pseudo-Anosov flow. 
Hence, modulo Frankel's upcoming paper, the construction of the action on the circle for quasi-geodesic flows can be reduced 
to the one for pseudo-Anosov flows. In this section, we will review the work of Calegari-Dunfield for the 3-manifolds admitting a pseudo-Anosov flow. 

As shown in the seminal paper of Cannon and Thurston \cite{CannonThurston}, the suspension flow of hyperbolic mapping
tori can be chosen to be both quasi-geodesic and pseudo-Anosov. They
used this to show that lifts of surface fibers of a fibered
hyperbolic 3-manifold extend continuously to the ideal boundary of
$\widetilde{M}$ (therefore their boundaries give group-equivariant
surjections from $S^1$ to $S^2$, which are commonly called Cannon-Thurston maps). This was later generalized by Fenley \cite{fenley2012ideal}. 
Hence, it might be instructive to consider the suspension flows when we think of a pseudo-Anosov flow. 
In the case of a suspension flow for a hyperbolic mapping torus $M$, one can consider the suspension of stable and unstable singular measured foliations on the surface for the monodromy to obtain 2-dimensional stable and unstable singular foliations in $M$. 
Analogously, in the case of a general pseudo-Anosov flow, $M$ has 2-dimensional stable and unstable singular foliations. 

Let $\mathcal{F}^u$ be the unstable foliation in $M$ for a pseudo-Anosov flow $\mathfrak{F}$. 
$\mathcal{F}^u$  can be split open to a lamination $\Lambda$. 
$\Lambda$ can be obtained from $\mathcal{F}^u$ by first removing the singular leaves, and for each singular leaf removed, 
we insert a finite-sided ideal polygon bundle over  the circle so that the leaves of $\Lambda$ are precisely the nonsingular leaves of 
$\mathcal{F}^u$ together with one leaf for every face of a singular leaf of $\mathcal{F}^u$. 
Just like in the case of the taut foliations, one can consider the lifted lamination $\widetilde{\Lambda}$ in $\widetilde{M}$
and the leaf space $L$ of $\widetilde{\Lambda}$. 
One caution here is that a vertex in $L$ is either a non-boundary leaf or a closed complementary region of $\widetilde{\Lambda}$. 
Since a complementary region comes from a singular leaf, it is natural to identify the whole thing as a single point in the space of leaves. 

At a point in $L$, it does not locally look like an open interval of the real line, but instead 
each point of $L$ has a neighborhood which is totally orderable, and
between any two points, there exists a unique path which is a
concatenation of such orderable segments. This structure is called an
order tree.  

One of the key statements in \cite{CalDun03} the following: 
\begin{thm}[Calegari-Dunfield \cite{CalDun03}] 
Let $M$ be a closed hyperbolic 3-manifold. If $M$ admits a very full lamination with orderable cataclysms, then $\pi_1(M)$ acts faithfully on the circle by orientation-preserving homeomorphisms. 
\end{thm} 
\begin{proof}[Sketch of the proof] 
We remark that Calegari-Dunfield showed a stronger result by weakening the assumption that the lamination is very full. 
They allowed the complementary regions of the lamination to be so-called solid torus guts, and in that case, it is shown that one can fill in the lamination with additional leaves 
to get a very full lamination while preserving many nice properties. 

As we explained above, the laminations we obtain from pseudo-Anosov flows (including the stable and unstable laminations in the hyperbolic mapping tori) 
are very full which means each complementary region is a finite-sided ideal polygon bundle over the circle. 
To see how this condition is used, we first fix orientations on the core curves of the complementary regions of $\Lambda$. 
This determines a natural cyclic order on the faces of each cataclysm. By formulation of the vertices of the leaf space $L$, this gives a natural 
cyclic order on the set of segments sharing exactly one vertex, and this order is $\pi_1(M)$-invariant by construction. 

The second condition of having orderable cataclysms means that 
there exists an ordering on each cataclysm which is invariant under the action of the stabilizer of the cataclysm in $\pi_1(M)$. 
A set of segments of $L$ which differ only by a single vertex correspond to a cataclysm, so they also have natural ordering which 
is $\pi_1(M)$-invariant  by definition of orderable cataclysms.

In summary, a set of segments of L which share exactly one vertex are
cyclically ordered and a set of segments of L which differ only at a
vertex are linearly ordered. Furthermore, these orderings are
$\pi_1(M)$-invariant. From this data, one can realize $L$ as a
``planar order tree''.
There are three types of points in $L$: first a cataclysm point, i.e.,
a point corresponding to a leaf in a cataclysm, second a singular
point which corresponds to a closed complementary region, and finally
an ordinary point which belongs to none of the previous two cases. 
Let $p$ be an arbitrary point in $L$. To be
concrete, let's assume $p$ is an ordinary point. Draw the point $p$ as an arbitrary point in
$\mathbb{R}^2$, maybe the origin, and let $I$ be the orderable segment
containing $p$ such that endpoints are either cataclysm points or
singular points but any other points are ordinary points. If an
endpoint is singular, one can draw the incident segments so that the
cyclic order on them matches with the cyclic order on their
realization inherited from the plane. If an endpoint is a cataclysm
point, again one can draw the the other segments ``incident'' at the
cataclysm with respect the linear order on them. Continuing this
process, we can realize $L$ as an order tree on the plane.

Let $e_1, e_2, e_3$ be three distinct ends of $L$. Pick a point $p$ in $L$ and let $r_i$ be the ray from $p$ to $e_i$ for $i = 1, 2, 3$.
Since $e_1, e_2, e_3$ all distinct, $r_i$'s must get separated at some
point, and form a subtree of $L$. Based on our realization of $L$ on
$\mathbb{R}^2$, then the rays $r_i$ are naturally cyclically ordered,
which gives a cyclic ordering on the triple $(e_1, e_2, e_3)$. 
Note that the ordering on the trip $(e_1, e_2, e_3)$ does not depend on the choice of $p$. 

This defines a cyclic ordering on the set $E$ of ends of $L$, and by construction, it is $\pi_1(M)$-invariant. Hence we obtained a cyclically ordered set $E$ where $\pi_1(M)$ acts 
by order-preserving maps. $E$ is equipped with the topology determined by its order: 
for $e \in E$, the sets $\{ x \in E \setminus \{a, b\} | (b, x, a) \mbox{ is positively oriented.} \}$ for some $a, b \in E$ where $(a, e, b)$ is positively oriented form a basis for the topology on $E$. Then there exists a unique continuous order-preserving embedding of $E$ into $S^1$ up to homeomorphisms. By collapsing each connected component of the complement of the closure of the image of $E$, we get a circle where $\pi_1(M)$ acts by orientation-preserving homeomorphisms. Here the circle is obtained as the order-completion of $E$, and we will denote it as $\overline{E}$. 

Suppose a nontrivial element $\alpha$ of $\pi_1(M)$ acts trivially on this circle. For each complementary region of $\widetilde{\Lambda}$, 
let $p$ be the vertex of $L$ corresponding to the complementary region. Consider all infinite rays in $L$ starting at $p$, and this defines a subset of $E$. 
The fact that $\alpha$ fixes this set implies that $\alpha$ actually fixes $p$. In other words, when we consider the action of $\alpha$ on $\widetilde{M}$, 
it preserves the given complementary region. Hence, all complementary
regions are preserved by $\alpha$. Each complementary region of
$\widetilde{\Lambda}$ is $\mathbb{Z}$-cover of a complementary region
of $\Lambda$. Hence if $\alpha$ preserves a complementary region of
$\widetilde{\Lambda}$, there it admits an invariant quasi-geodesic. If
$\alpha$ preserves another complementary region, $\alpha$ would admit
another quasi-geodesic axis who endpoints are disjoint from the one
we already had, a contradiction. We have shown that the $\pi_1(M)$-action on the circle constructed above is faithful. 
\end{proof}

To apply the above theorem to our case, it remains to see that our lamination $\Lambda$ has orderable cataclysms. 
This observation is due to Fenley \cite{fenley1998structure}. 
Note that each leaf of $\widetilde{\Lambda}$ is foliated by the flow lines of $\widetilde{\mathfrak{F}^u}$ contained in that leaf. Whenever 
we talk about the foliation on a leaf, we refer to this foliation coming from $\widetilde{\mathfrak{F}^u}$. 
Let $\{\lambda_j\}$ be (an equivalence class of) a cataclysm and let $(\mu_i)$ be a monotone ordered sequence of nonsingular leaves of $\widetilde{\Lambda}$ converging to 
$\{\lambda_j\}$ on compact subsets of $\widetilde{M}$. For each $j$, choose a sequence of points $p_{ij} \in \mu_i$ 
so that $p_{ij}$ converges to a point $q_j$ in $\lambda_j$ as $i$ tends to $\infty$. 

Candel's theorem again applies here: $M$ admits a metric so that each $\mu_i$ is isometric to $\mathbb{H}^2$. Then the foliation on $\mu_i$ from $\widetilde{\Lambda}$
is a foliation by bi-infinite geodesics which all share one endpoint (this is an unstable lamination so the flow lines are oriented so that it flows from this common endpoint). 
Hence the leaf space of the foliation on each $\mu_i$ is $\mathbb{R}$, hence naturally totally ordered. 
The set $\{p_{ij}\}$ of points on $\mu_i$ has a natural order on the indices $j$ with respect to this order. For each $j$, we can take a small product chart $U_j$ around $q_j$. 
For all large enough $i$, the plaque $P_j$ obtained as the intersection $U_j \cap \mu_i$ contains $p_{ij}$ and $P_j$ converges to $U_j \cap \lambda_j$ as foliated disks. 
Hence, the order relation between $p_{ij}$ and $p_{ij'}$ remains the same for all sufficiently large $i$. Hence, this gives an ordering on the set $\{q_j\}$ which can be used as an ordering on the cataclysm $\{\lambda_j\}$. Since the flow lines of $\widetilde{\mathcal{F}^u}$ are preserved under $\pi_1(M)$-action, our ordering on the cataclysm 
is invariant under the action of its stabilizer in $\pi_1(M)$. Hence, the unstable lamination for a pseudo-Anosov flow has orderable cataclysms so the above theorem applies. 
We finally obtain 

\begin{thm}[Calegari-Dunfield \cite{CalDun03}]  
Let $M$ be a closed hyperbolic 3-manifold which admits a pseudo-Anosov flow. Then $\pi_1(M)$ acts faithfully on the circle by orientation-preserving homeomorphisms. 
\end{thm}

\section{Invariant laminations for the universal circles and laminar groups}
\label{sec:invlamination}
A lamination $\Lambda$ on $S^1$ is defined to be a closed subset of
the set of all unordered pairs of two distinct points of $S^1$ so that
any two elements are unlinked. Here two pairs $(a, b)$ and $(c, d)$ of
points of the circle are unlinked if both $a, b$ are contained in the
closure of a single connected component of $S^1
\setminus \{c, d\}$. Note that if $a = c$ and $b \neq d$, the pairs $(a, b)$ and
$(c,d)$ are still unlinked according to our definition.

One can visualize $\Lambda$ by identifying the circle with the ideal
boundary of $\mathbb{H}^2$ and then realize each element as the
endpoints of a bi-infinite geodesic. We call this geodesic lamination
a \textsf{geometric realization} of $\Lambda$. Since the geometric
realization is unique up to isotopy, we will freely go back and forth
between a lamination on the circle and its geometric realization to
discuss its properties. 

We first consider the case that $M$ is a closed hyperbolic 3-manifold
and $\mathcal{F}$ is a co-orientable taut foliation with a branching.
In Section \ref{sec:specialsections}, we saw that there
exists the set of special sections which has a $\pi_1(M)$-invariant
cyclic order and it can be completed to get a universal circler
$S^1_{univ}$ where $\pi_1(M)$ acts faithfully by
orientation-preserving homeomorphisms.

Now we see that this action
preserves laminations. We will construct a lamination $\Lambda^+$
assuming the leaf space $L$ is branched in the forward direction. In
the case $L$ has a branching in the backward direction, one can
construct another lamination $\Lambda^-$ in a completely analogous
way. For each leaf $\lambda$ in $L$, let $L^+(\lambda)$ denote the
connected component of $L$ containing at least one leaf $\mu$ with
$\mu > \lambda$. For a subset $X$ of $L$, we say $core(X)$ is the
union of the cores associated with the leaves in $X$. Let
$\Lambda(core(X))$ be the boundary of the convex hull of the closure
of $core(X)$ in $\mathbb{H}^2$. Finally, define $\Lambda^+(\lambda)$
to be $\Lambda(core(L^+(\lambda)))$, and $\Lambda^+$ to be the closure
of the union $\cup_{\lambda \in L} \Lambda^+(\lambda)$. Note that
$\Lambda^+$ is completely determined by the structure of $L$.

To see this is indeed a lamination, we need to show that for $\lambda,
\mu \in L$, no leaf o$\Lambda^+(\lambda)$ is linked with a leaf of
$\Lambda^+(\mu)$. This is easy to see when $\lambda, \mu$ are
comparable, since one of o$\Lambda^+(\lambda)$ and $\Lambda^+(\mu)$ is
contained in the other. When they are incomparable, there are two
cases. One case is that $\lambda \notin \Lambda^+(\mu)$ and $\mu
\notin \Lambda^+(\lambda)$. In this case, $\Lambda^+(\mu)$ and
$\Lambda^+(\lambda)$ are disjoint, so again straightforward. Finally,
let us assume that $\lambda \in \Lambda^+(\mu)$ and $\mu
\in \Lambda^+(\lambda)$. In this case, $\Lambda^+(\lambda) \cup
\Lambda^+(\mu) = L$. Hence $core(L) = core(\Lambda^+(\lambda)) \cup
core(\Lambda^+(\mu))$, so the boundaries of the convex hulls do not
cross in $\mathbb{H}^2$.

Up to here, we did not really need to assume that $L$ is branched in
the forward direction. To see $\Lambda^+$ is nonempty, we need.
From the assumption that $L$ has a branching in the forward direction,
there exist leaves $\mu, \lambda$ so that $\lambda \notin \Lambda^+(\mu)$ and $\mu
\notin \Lambda^+(\lambda)$. As we noted above, $\Lambda^+(\mu)$ and
$\Lambda^+(\lambda)$ are disjoint, so their cores are unlinked. In
particular, $core(\Lambda^+(\lambda))$ is not dense in $S^1_{univ}$,
which is sufficient to conclude that $\Lambda^+$ is nonempty.

Now we get an invariant lamination for the universal circle action for
the pseudo-Anosov flow. Let's consider the setup of Section
\ref{sec:flows}. 
Let $p_1, \ldots, p_k$ be points in $L$ corresponding to a set of
representative of orbits of cataclysm points under
$\pi_1(M)$-action. Say each $p_i$ corresponds to a complementary
region which is an ideal $n_i$-gon bundle over the circle. Then $L
\setminus \{p_i\}$ consists of $n_i$ subtrees of $L$. 
Choose $q_1, \ldots, q_{n_i}$ on $\overline{E}$ which separate the
ends of distinct subtrees of $L \setminus \{p_i\}$. We may assume that 
the $n_i$-tuple $(q_1, \ldots, q_{n_i})$ is positively oriented with
respect to the cyclic order on $\overline{E}$. Then we consider the
pairs $(q_j, q_{j+1})$ for each $j = 1, \ldots, n_i \mod
n_i+1$. We do this for each $p_i$ and take the union of
$\pi_1(M)$-orbits of all those pairs, and call it $\Lambda$. This
process can be done so that elements of $\Lambda$ are pairwise
unlinked. By taking a closure of $\Lambda$ in the space of unordered
pairs of points of $\overline{E}$, we get a $\pi_1(M)$-invariant
lamination. 

In summary, 
\begin{thm} Let $M$ be a closed hyperbolic 3-manifold with either a
  taut foliation, a quasi-geodesic flow or a pseudo-Anosov flow. Then $\pi_1(M)$ acts
  faithfully on the circle by orientation-preserving homeomorphisms
  with an invariant lamination. 
\end{thm} 

From this result, it is natural to ask if a group acts faithfully on the circle by orientation-preserving homeomorphisms with 
invariant laminations, does it have any interesting property? We call such a group a \textsf{laminar group}. 

One might first wonder whether there are some other natural examples of laminar groups other than 3-manifold groups we have seen. 
In fact, all surface groups are laminar groups. Let $S_g$ be a closed connected orientable surface of genus $g \geq 2$, and fix a hyperbolic metric on $S_g$. 
The deck group action of $\pi_1(S_g)$ on $\mathbb{H}^2$ extends to an action on $\partial_\infty \mathbb{H}^2$ by homeomorphisms. 
In this case, any geodesic lamination on $S_g$ defines a lamination on $\partial_\infty \mathbb{H}^2$ which is $\pi_1(S_g)$-invariant. In this case, one can easily construct
infinitely many invariant laminations with a lot of structures. 

A lamination $\Lambda$ on $S^1$ is called very full if when it is realized as a geodesic lamination on $\mathbb{H}^2$ via arbitrary identification of $S^1$ with $\partial_\infty \mathbb{H}^2$, all the complementary regions are finite-sided ideal polygons. For the later use, let's call this geodesic lamination on $\mathbb{H}^2$ a geometric realization of $\Lambda$. In the case of $\pi_1(S_g)$, there are infinitely many very full invariant laminations on $\partial_\infty \mathbb{H}^2$. One way to get a very full lamination is to start with a pants-decomposition by simple closed geodesics and then decompose each pair of pants into two ideal triangles by three bi-infinite geodesics which spiral toward boundary components. Then all complementary regions of the resulting lamination are ideal triangles. Since there are infinitely many different pants-decompositions, we get infinitely many different very full invariant laminations. In fact, this argument can be easily generalized to any (complete) hyperbolic surface except the three-punctured sphere even the one with infinite area. 

In \cite{BaikFuchsian}, the first author showed that this is actually the characterizing property for hyperbolic surface groups. In fact, we only need three invariant laminations instead of infinitely many invariant laminations.  Roughly speaking, a group acting faithfully on the circle acts like a hyperbolic surface group if and only if it admits three different very full invariant laminations. 
Via arbitrary identification of $S^1$ with $\partial_\infty \mathbb{H}^2$, we always identify $\PR$ with a subgroup of $\Homeop{S^1}$.
A precise version of this theorem is the following: 
\begin{thm}[Baik \cite{BaikFuchsian}]\label{thm:baikmainthm}
Let $G < \Homeop{S^1}$ be torsion-free discrete subgroup.  
Then $G$ is conjugated into $\PR$ by an element of $\Homeop{S^1}$ so that $\mathbb{H}^2/G$ is not a three-punctured sphere if and only if $G$ admits three very full invariant laminations $\Lambda_1, \Lambda_2, \Lambda_3$ where a point $p$ of $S^1$ is a common endpoint of leaves from $\Lambda_i$ and $\Lambda_j$ for $i \neq j$ if and only if it is a cusp point of $G$ (i.e., a fixed point of a parabolic element). 
\end{thm} 

One can deduce the following simplified version immediately from the above theorem. 
\begin{cor}[Characterization of cusp-free hyperbolic surface groups] 
Let $G < \Homeop{S^1}$ be torsion-free discrete subgroup.  Then $G$ is conjugated into $\PR$ by an element of $\Homeop{S^1}$ so that $\mathbb{H}^2/G$ has no cusps if and only if $G$ admits three very full invariant laminations $\Lambda_1, \Lambda_2, \Lambda_3$ so that leaves from $\Lambda_i$ and $\Lambda_j$ with $i \neq j$ do not share an endpoint. 
\end{cor} 

The proof is pretty long so we do not try to recall it here, but we would like to talk about some key ingredients. 
One very important observation on the very full laminations $\Lambda$ is that each point $p$ on $S^1$ which is not an endpoint of any leaf of $\Lambda$ has a nested sequence of neighborhoods $(I_j)$ so that $I_j$ shrinks to $p$ and for each $j$ there exists a leaf $\ell_j$ of $\Lambda$ whose endpoints are precisely the endpoints of $I_j$. Such a sequence of leaves $(\ell_j)$ is called a rainbow at $p$. In short, 
\begin{lem}[Baik \cite{BaikFuchsian}] 
Let $\Lambda$ be a very full lamination on $S^1$. For each $p \in S^1$, either there exists a leaf of $\Lambda$ which has $p$ as an endpoint, or there exists a rainbow at $p$. 
\end{lem} 

Another key ingredient is actually a big hammer called convergence group theorem. 
Let $G$ be a group acting on a compactum $X$. We say the $G$-action is a convergence group action if the induced diagonal action of $G$ on $X \times X \times X - \Delta$ 
where $\Delta$ is the big diagonal is properly discontinuous. 
\begin{thm}[Convergence group theorem (Gabai), (Casson-Jungreis), (Tukia), (Hinkkanen), .... ] 
Suppose a group $G$ acts on $S^1$ as a convergence group. Then $G$ is conjugated into $\PR$. 
\end{thm}

Due to this remarkable theorem, one only needs to check that if $G$ admits three very full laminations, then $G$ acts on $S^1$ as a convergence group. 
Suppose not. By definition, this means that there exist a sequence $( (x_i, y_i, z_i) )$ of three distinct points in $S^1$ and a sequence $(g_i)$ of elements of $G$ such that 
$(x_i, y_i, z_i) \to (x_\infty, y_\infty, z_\infty)$ and $(g_i x_i, g_i y_i, g_i z_i) \to (x_\infty', y_\infty', z_\infty')$ where $(x_\infty, y_\infty, z_\infty)$ and $(x_\infty', y_\infty', z_\infty')$ 
are triples of distinct points in $S^1$. 
One can then check that for various possibilities for $x_\infty, y_\infty, z_\infty, x_\infty', y_\infty', z_\infty'$ 
either being an endpoint of leaves or having rainbows in each
$\Lambda_i$, each case cannot happen by finding a leaf which
is forced to be mapped to a pair which is linked to the given
leaf. For details, consult \cite{BaikFuchsian}.

\section{Basic notions and notations to study the group action on the
  circle}
\label{sec:basicnotions}
So far we have provided a brief review on previously known
results. Starting from this section, we now move toward some
recent results on this topic. First we need to review some basic
notions and set up notations. 

Let $S^1$ be a multiplicative topological subgroup of $\CC$ defined as 
$$S^1=\{z\in \CC : |z|=1 \}.$$ In this section, we define some notions
on $S^1$. So far we have used the term cyclic order, but from now on,
we will call it a circular order, since it is more suitable for the
context. To give more precise definitions, let us consider the
stereographic projection $p: S^1 \setminus \{ 1 \} \to \RR$ defined as 
$$p(z)=\frac{Im(z)}{Re(z)-1}.$$
Obviously, $p$ is homeomorphism under standard topologies. For convenience, we define the degenerate set $\Delta_n(G)$ of a set $G$ to be the set 
$$\Delta_n(G)=\{ (g_1, \cdots, g_n)\in G^n : g_i=g_j \ for \ some \ i \neq j \}.$$
of all $n$-tuples with some repeated elements. 

\begin{defn}
For $n\geq 3$, an element $(x_1, \cdots, x_n)$ in $(S^1)^n-\Delta_n(S^1)$ is \textsf{positively oriented n-tuple} on $S^1$ if for each $i \in \{2, \cdots, n-1\}$, $p(x_1^{-1}x_i)<p(x_1^{-1}x_{i+1}).$  An element $(x_1, \cdots, x_n)$ in $(S^1)^n-\Delta_n(S^1)$ is \textsf{negatively oriented n-tuple} on $S^1$ if for each $i \in \{2, \cdots, n-1\}$, $p(x_1^{-1}x_{i+1})<p(x_1^{-1}x_{i}).$ 
\end{defn}  

We use the definition of circular orders in the following form. 
\begin{defn}
\textsf{A circular order} on a set $G$ is a map $\phi: G^3\rightarrow \{ -1, 0, 1 \}$ with the following properties:
\begin{enumerate}
\item[(DV)] $\phi$ kills precisely the degenerate set, i.e. 
$$\phi^{-1}(0)=\Delta_3(G).$$
\item[(C)] $\phi$ is a $2$-cocycle, i.e. 
$$\phi(g_1, g_2, g_3)-\phi (g_0,g_2, g_3)+\phi(g_0,g_1,g_3)-\phi(g_0, g_1, g_2)=0$$
for all $g_0,g_1,g_2, g_3 \in G$.  
\end{enumerate} 
Furthermore, if $G$ is a group, then  a \textsf{left-invariant circular order} on $G$ is a circular order on $G$ as set that also satisfies the homogeneity property:
\begin{enumerate}
\item[(H)] $\phi$ is homogeneous, i.e. 
$$\phi(g_0,g_1,g_2)=\phi(hg_0, hg_1, hg_2)$$
for all $h\in G$ and $(g_0,g_1,g_2)\in G^3$. 
\end{enumerate}
\end{defn}

By abuse of language, we will refer to "left-invariant circular order of a group" simply as a "circular order."
To learn about invariant circular orders of groups, see \cite{baik2018spaces}. 

Let's define a circular order $\phi$ on a multiplicative group $S^1$
as the following way. For $p\in \Delta_3(S^1)$, $\phi(p)=0$. When
$p\in (S^1)^3 - \Delta_3(S^1)$, we assign $\phi(p)=1$ if $p$ is
positively oriented in $S^1$, and $\phi(p)=-1$ if $p$ is negatively
oriented in $S^1$. We can easily check that $\phi$ is a circular order of a group $S^1$.

We also set up terminologies for the intervals on $S^1$. The reason is
that in the rest of the paper, we will use a different perspective on laminations on $S^1$
to view them as sets of intervals with certain conditions. Using this new
perspective, we will give detailed discussion of laminar groups.
First, we call a nonempty proper connected open subset of $S^1$ an
\textsf{open interval} on $S^1$. Technically, we distinguish the following two. 

\begin{defn} Let $u, v$ be two elements of $S^1$. 
\begin{enumerate}
\item If $u\neq v$, $(u,v)_{S^1}$ is the set 
$$(u,v)_{S^1}=\{ p\in S^1 : \phi(u,p,v)=1 \}.$$ We call it a \textsf{nondegenerate open interval} on $S^1$.
\item If $u=v$, $(u,v)_{S^1}$ is the set 
$$(u,v)_{S^1}=S^1-\{ u \}.$$ We call it a \textsf{degenerate open interval} on $S^1$.
\end{enumerate}
If $(u,v)_{S^1}$ is a nondegenerate open interval, then we denote $(v,u)_{S^1}$ by $(u,v)_{S^1}^*$, and call it the \textsf{dual interval} of $(u,v)_{S^1}$.
\end{defn}

We can check that the set of all nondegenerate open intervals of $S^1$ is a base for a topology of $S^1$ which is induced from the standard topology of $\CC$.  
For convenience,  we also use the following list of notations. Let $(u,v)_{S^1}$ be a nondegenerate open interval. Then, we denote 
\begin{enumerate} 
\item for $z\in S^1$, $z(u,v)_{S^1}=(zu, zv)_{S^1}$
\item $[ u, v )_{S_1}=\{u\} \cup (u,v)_{S^1}$,
\item $(u,v]_{S^1}= (u,v)_{S^1} \cup \{ v \}$, and 
\item $[u,v]_{S^1}= (u,v)_{S^1} \cup \{u, v \}$.
\end{enumerate}

Obviously, we can derive the following list of properties about dual intervals. Let $I$ and $J$ be  two nondegenerate open intervals.  
\begin{enumerate}
\item $(I^*)^*=I$,
\item $I^c=\overline{I^*}$,
\item If $I\subseteq J$, then $J^*\subseteq I^*$,
\item If $I\cap J = \phi$, then $I \subseteq J^*$,
\item If $I\cap J =\phi$, then $| \bar{I} \cap \bar{J} | \leq 2$,

\end{enumerate}
where $\bar{I}$ is the closure of $I$.
Recall that the every open subset of $\RR$ can be obtained by the countable disjoint union of open intervals of $\RR$. Likewise, it is also true in $S^1$. 

\begin{prop}
Every proper open set of $S^1$ is an at most countable union of disjoint open intervals.  
\end{prop}


\section{Lamination systems on $S^1$ and Laminar groups of $\Homeop(S^1)$}

Using the notations and terminologies defined in the previous section,
we introduce the notion of a lamination system on $S^1$. This is a set
of intervals on $S^1$ with certain conditions which corresponds to
leaves of our usual notion of a lamination on $S^1$. Before defining lamination systems, we need the following definition which  is analogous to the unlinkedness condition in laminations on $S^1$.

\begin{defn}
Let $I$ and $J$ be two nondegenerate open intervals. If  $I\subseteq J$ or $I^* \subseteq J$, then we say that two points set $\{ I, I^*\}$ \textsf{lies on} $J$(see Figure \ref{fig:lie}). If  $\bar{I}\subseteq J$ or $\overline{I^*} \subseteq J$, then 
 we say that two points set $\{ I, I^*\}$ \textsf{properly lies on} $J$. 
\end{defn}

\begin{figure}
  \includegraphics[width=12cm]{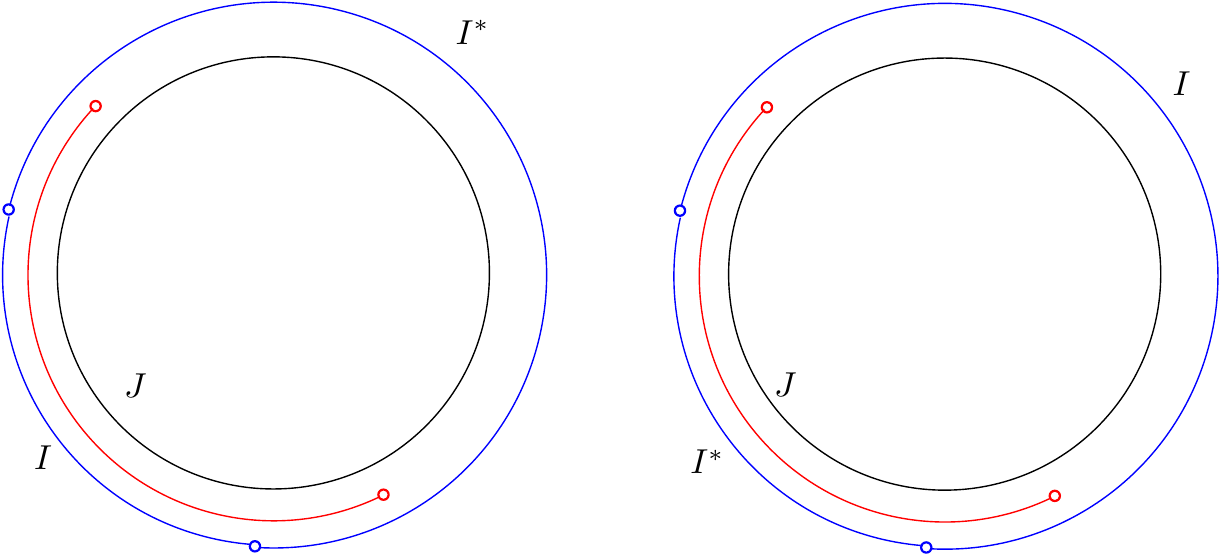}
  \caption{The red segment represents the nondegenerate open interval $J$ and the blue parts represent $I$ and $I^*$. Two figures show  all possible cases that $\{I,I^*\}$ lies on $J$.}
  \label{fig:lie}
\end{figure}

Let us define the lamination system. 
\begin{defn}
Let $\mathcal{L}$ be a nonempty family of nondegenerate open intervals of $S^1$. $\mathcal{L}$ is called 
a \textsf{lamination system} on $S^1$ if it satisfies the following three properties :
\begin{enumerate}
\item If $I\in \mathcal{L}$, then $I^* \in \mathcal{L}$. 
\item For any $I, J\in \mathcal{L}$, $\{I,I^*\}$ lies on $J$ or $J^*$.
\item If there is a sequence $\{I_n\}_{n=1}^{\infty}$ on $\mathcal{L}$ such that for $n\in \NN$, $I_n\subseteq I_{n+1}$, and $\displaystyle \bigcup_{n=1}^{\infty} I_n$ is a nondegenerate open interval in $S^1$, then $\displaystyle \bigcup_{n=1}^{\infty} I_n\in \mathcal{L}$.
\end{enumerate}  
\end{defn}

The original definition of laminations on $S^1$ is a closed subset of the set of all two points subsets of
$S^1$ with unlinkedness condition. In a lamination system, each two points set corresponds to the set of  two connected components of the complement of the two points. In this sense, we define leaves and gaps on a lamination system $\LL$ as the followings.

\begin{figure}
  \includegraphics[width=6cm]{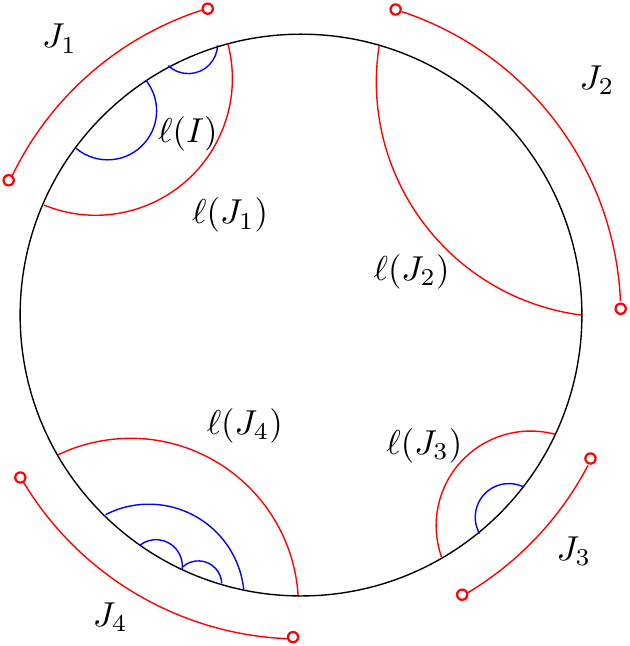}
  \caption{The red chords on the disk are the geometric realization of $\ell(J_i)$.  In this figure, the geodesic lamination is the union of red and blue chords and a gap is $\{J_1,J_2,J_3,J_4\}$. Note that all blue chord $\ell(I)$ lies on $J_i$. }
  \label{fig:gap}
\end{figure}

A subset $\GG$ of $\LL$ is  a \textsf{leaf} of $\LL$ if  $\GG=\{ I,I^*\}$ for some $I\in \LL$. We denote such  a leaf  $\GG$ by $\ell(I)$. With this definition of leaves, we can see that the second condition of lamination system implies unlinkedness of leaves of laminations of $S^1$. Likewise, a subset $\GG$ of $\LL$ is a \textsf{gap} of $\LL$ if $\GG$ satisfies the following two conditions: 
\begin{enumerate}
\item Elements of $\GG$ are disjoint. 
\item For any $I\in\LL$, there is $J$ in $\GG$ on which $\ell(I)$ lies (see Figure \ref{fig:gap}). 
\end{enumerate}
By the second condition of gaps, every gap is nonempty.   Obviously,  a leaf is also a gap with two elements. Then, we denote $S^1- \displaystyle \bigcup_{I \in \GG} I$ as $v(\GG)$ and call it a \textsf{vertex set} of $\GG$ or a  \textsf{end points set} of $\GG$. And each element of a vertex set is called a \textsf{vertex} or an \textsf{end point}. Note that in general, a vertex set  need not be  a discrete subset of $S^1$. Geometrically, in $\overline{\HH^2}$, $conv(v(\GG))$ is the geometric realization of a gap $\GG$ where $conv(A)$ is the convex hull of a set $A$ in $\overline{\HH^2}$. 

The third condition of lamination systems is analogous to the closedness of laminations on $S^1$. From now on, to describe the limit of a sequence of leaves,  we define the notion of convergence of a sequence of leaves. 

\begin{defn}
Let $\LL$ be a lamination system, and $\{\ell_n\}_{n=1}^{\infty}$ be a sequence of leaves on 
$\LL$. Let $J$ be a nondegenerate open interval. We say that $\{\ell_n\}_{n=1}^{\infty}$ \textsf{converges to} $J$ if there is a sequence $\{I_n\}_{n=1}^{\infty}$ on $\LL$ such that 
for each $n\in \NN$, $\ell_n=\ell(I_n)$, and 
$$J\subseteq \liminf I_n \subseteq \limsup I_n \subseteq \overline{J}.$$ We denote this by 
$\ell_n\rightarrow J$.
\end{defn}
This definition is symmetric in the following sense.

\begin{prop}\label{prop:symmetry of sequences}
Let $\LL$ be a lamination system and $\{\ell_n\}_{n=1}$ be a sequence of leaves on $\LL$. Let $J$ be a nondegenerate open interval. Suppose that  there is a sequence $\{I_n\}_{n=1}^{\infty}$ on $\LL$ such that for each $n\in \NN$, $\ell_n=\ell(I_n)$ and $$J\subseteq \liminf I_n \subseteq \limsup I_n \subseteq \overline{J}.$$ Then $$J^*\subseteq \liminf I_n^* \subseteq \limsup I_n^* \subseteq \overline{J^*}.$$
\end{prop}
\begin{proof}
Since $J\subseteq \liminf I_n = \displaystyle \bigcup_{k=1}^{\infty}\bigcap_{n=k}^{\infty} I_n \subseteq \overline{J},$ so  
$J^*=\overline{J}^c \subseteq \displaystyle \bigcap_{k=1}^{\infty}\bigcup_{n=k}^{\infty}  I_n^c\subseteq  J^c=\overline{J^*}.$  So, $ \displaystyle \limsup I_n^*=\bigcap_{k=1}^{\infty}\bigcup_{n=k}^{\infty} I_n^*\subseteq\bigcap_{k=1}^{\infty}\bigcup_{n=k}^{\infty}  \overline{I_n^*}= \displaystyle \bigcap_{k=1}^{\infty}\bigcup_{n=k}^{\infty}  I_n^c \subseteq J^c=\overline{J^*}.$ The remain part is to show $J^*\subseteq \liminf I_n^*.$
Since $J\subseteq \limsup I_n =\displaystyle \bigcap_{k=1}^{\infty}\bigcup_{n=k}^{\infty} I_n\subseteq \overline{J}$,
so $J^*=\overline{J}^c\subseteq \displaystyle \bigcup_{k=1}^{\infty}\bigcap _{n=k}^{\infty}I_n^c= \bigcup_{k=1}^{\infty}\bigcap _{n=k}^{\infty}\overline{I_n^*}\subseteq J^c=\overline{J^*}.$  Denote $J^*=(u,v)_{S^1}$ and choose $w\in J.$  For each $n\in \NN$, define $(u_n,v_n)_{S^1}$ as the following :
 $$(u_n,v_n)_{S^1}=wp^{-1}((p(w^{-1}u)+\frac{L}{3n}, p(w^{-1}v)-\frac{L}{3n}))$$
 where $p$ is the stereographic projection and $L=p(w^{-1}v)-p(w^{-1}u).$
 Then for all $n\in \NN$, $[u_n,v_n]_{S^1}\subseteq (u_{n+1},v_{n+1})_{S^1} \subseteq J^*.$
 
 Fix $m\in \NN$. Since $[u_m,v_m]_{S^1}\subseteq J^*\subseteq \displaystyle \bigcup_{k=1}^{\infty}\bigcap_{n=k}^{\infty}\overline{I_n^*},$ 
 there is a natural number $N_m$  such that $\{u_m,v_m\} \subseteq \displaystyle \bigcap_{n=N_m}^{\infty} \overline{I_n^*}.$  Note that since $\displaystyle \bigcap_{k=1}^{\infty}\bigcup_{n=k}^{\infty}\overline{I_n^*} \subseteq \overline{J^*}$, there is a natural number $N$ such that $w\notin\displaystyle \bigcup_{n=N}^{\infty}\overline{I_n^*}.$ Let $M_m=\max \{N,N_m\}.$ Then for all $k\geq M_m$, $\{u_m,v_m\} \subseteq \overline{I_k^*}$ and $w\notin \overline{I_k^*}.$ 

 From now on, we show that for all $k\geq M_m$, $(u_m,v_m)_{S^1}\subseteq I_k^*.$ Fix $k\geq M_m$ and denote $I_k^*=(a_k,b_k)_{S^1}$. Note that  since $[u_m,v_m]_{S^1}\subseteq J^*$ and $w\in J$, $w\notin [u_m,v_m]_{S^1}.$ If $\{a_k,b_k\}=\{u_m,v_m\}$, then $(a_k,b_k)_{S^1}=(u_m,v_m)_{S^1}$ since $w\notin [u_m,v_m]_{S^1}.$ If not, there is an element $v\in \{u_m, v_m\} -\{a_k,b_k\}.$ First, consider the case $v=u_m.$ since $\{u_m,v_m\} \subseteq \overline{I_k^*}$, it is $v_m\in [a_k, u_m)_{S^1}$ or $v_m\in (u_m,b_k]_{S^1}.$ If  $v_m\in [a_k, u_m)_{S^1}$, then $(v_m,u_m)_{S^1}\subseteq [a_k,u_m)_{S^1}\subseteq [a_k,b_k]_{S^1}$. However, it is a contradiction since $w\in (v_m,u_m)_{S^1}$ and $w\notin [a_k,b_k]_{S^1}$. Therefore, $v_m\in (u_m,b_k]_{S^1}$ and so $(u_m,v_m)_{S^1}\subseteq (u_m, b_k]_{S^1}\subseteq [a_k,b_k]_{S^1}.$ Thus, $(u_m, v_m)_{S^1}\subseteq (a_k,b_k)_{S^1}=I_k^*.$ Similarly, we can prove the case $v=v_m.$
 
Therefore, for all $m\in \NN$, $(u_m,v_m)_{S^1}\subseteq \displaystyle \bigcap_{n=M_m}^{\infty} I_n^*\subseteq \displaystyle \bigcup_{k=1}^{\infty} \bigcap_{n=k}^{\infty}I_n^*=\liminf I_n^*.$ Thus, $J^*=\displaystyle \bigcup_{m=1}^{\infty}(u_m,v_m)_{S^1}\subseteq \liminf I_n^*.$
 
\end{proof}






Since the third condition of lamination systems guarantees  that the limit of an ascending sequence on a lamination system is in the lamination system,  we need to consider descending sequences to say about  limits of arbitrary sequences on lamination systems. 
The following lemma implies closedness of   descending sequences in a lamination system $\LL$. 
\begin{lem}\label{lods}
Let $\{ I_n\}_{n=1}^{\infty}$ be a sequence on a lamination system $\LL$ such that $I_{n+1}\subseteq I_n$ for all $n\in \NN$, and $\displaystyle \bigcup_{n=1}^{\infty} I_n^* =J \in \LL$. Then $Int\Big( \displaystyle\bigcap_{n=1}^{\infty} I_n \Big)=J^*\in \LL$. 
\end{lem}
\begin{proof}
Since $ \displaystyle \bigcup_{n=1}^{\infty} I_n^* =J$, so 
$$\overline{J^*}=J^c = \displaystyle \bigcap_{n=1}^{\infty} (I_n^*)^c= \bigcap_{n=1}^{\infty} \overline{I_n} .$$
So, $\displaystyle \bigcap_{n=1}^{\infty} I_n \subseteq  \bigcap_{n=1}^{\infty} \overline{I_n}= \overline{J^*}$. Since $\displaystyle \bigcup_{n=1}^{\infty} I_n^* =J $, so 
 for all $n\in \NN$, $I_n^* \subseteq J$ and so $J^*\subseteq I_n$. Therefore,  
$J^*\subseteq \displaystyle \bigcap_{n=1}^{\infty} I_n$. Thus, $J^*\subseteq \displaystyle \bigcap_{n=1}^{\infty} I_n = \overline{J^*} $ and so $Int\Big( \displaystyle\bigcap_{n=1}^{\infty} I_n \Big)=J^*$.
\end{proof}

With this lemma, the following proposition shows the closedness of  lamination systems.

\begin{prop}\label{c}
If a sequence $\{\ell_n\}_{n=1}^{\infty}$ of leaves of  a lamination system $\LL$ converges to a nondegenerate open interval $J$, then $J\in \LL$.
\end{prop}
\begin{proof} Since the sequence $\{\ell_n\}$ converges to $J$, there is a sequence $\{I_n\}_{n=1}^{\infty}$ on the lamination system $\LL$ such that  for all $n\in \NN$, $\ell_n=\ell(I_n)$ and 
$$J\subseteq \liminf I_n \subseteq \limsup I_n \subseteq \overline{J}.$$
Since $\displaystyle\bigcap_{k=1}^{\infty}\displaystyle\bigcup_{n=k}^{\infty}I_n \subseteq \overline{J}$, and so $J^*=\overline{J}^c\subseteq \displaystyle \bigcup_{k=1}^{\infty}\bigcap_{n=k}^{\infty}I_n^c$, there is $N\in \NN$ such that $\displaystyle \bigcap_{n=N}^{\infty}I_n^c \neq \phi. $ Then we can get 
$\Big(\displaystyle\bigcup _{n=k} ^{\infty} I_n \Big)^c \ne \emptyset$ for all $k\geq N$ since 
$\displaystyle\bigcap _{n=k} ^{\infty} I_n ^c \subseteq \bigcap_{n=k+1}^{\infty} I_n^c$ for all $k\ge N$. On the other hand, since $J\subseteq \displaystyle \bigcup_{k=1}^{\infty}\bigcap_{n=k}^{\infty}I_n$, there is $N'\in \NN$ such that $\displaystyle \bigcap_{n=N'}^{\infty}I_n\neq \phi$. Choose $p\in \displaystyle \bigcap_{n=N'}^{\infty}I_k$, and set $M=max\{N, N'\}$. Since for $n\ge M$, $I_n$ is a connected open subset of $S^1$, and contains $p$, so for $k\ge M$, $\displaystyle \bigcup_{n=k}^{\infty}I_n$ is nonempty open connected subset of $S^1$. Since for all $k\geq M$, $\Big(\displaystyle\bigcup _{n=k} ^{\infty} I_n \Big)^c \ne \phi$ and so $\displaystyle\bigcup _{n=k} ^{\infty} I_n $ is a proper subset of $S^1$ , so for each $k \geq M$, $\displaystyle \bigcup_{n=k}^{\infty}I_n$ is an open interval.   If for all $k\ge M$, $\displaystyle \bigcup_{n=k}^{\infty}I_n$ is a degenerate open interval, then
 $\displaystyle \bigcap_{k=1}^{\infty}\bigcup_{n=k}^{\infty}I_n=\bigcap_{k=M}^{\infty}\bigcup_{n=k}^{\infty}I_n$ is a degenerate open interval, but it contradicts to $\displaystyle \bigcap_{k=1}^{\infty}\bigcup_{n=k}^{\infty}I_n \subseteq \overline{J}$. Therefore, there is $K\in \NN$ such that $K\ge M$, and for all $k\ge K$, $\displaystyle \bigcup_{n=k}^{\infty}I_n$ is nondegenerate. Let $L=\displaystyle \bigcup_{n=K}^{\infty}I_n$. Then, for all $n\ge K$, $I_n\subseteq L$ and so $L^*\subseteq I_n^*$. Therefore, we get that for all $n\ge K$, $p \in I_n \subseteq L$. 
 
 \
 
 From now on, we show that for any $n,m\ge K$, $I_n\subseteq I_m$ or $I_m\subseteq I_n$. Choose $n, m \geq K$.  If $I_n \subseteq I_m^*$, then $p\in I_m^*$, and so it contradicts to $p\in I_m$. If $I_m^*\subseteq I_n$, then $L^*\subseteq I_m^* \subseteq I_n$, and so it contradicts to $ I_n \subseteq L $. So, we can get what we want.

 \
 
 Then, we show that $\displaystyle \bigcup_{n=k}^{\infty} I_n \in \LL$ for all $k\ge K$. Fix $k$ with $k\ge K$. 
 Let $J_i=\displaystyle \bigcup_{n=k}^{k+i} I_n$ for all $i\in \NN$. Then $J_1=I_k \cup I_{k+1}$. Since $I_k\subseteq I_{k+1}$ or $I_{k+1}\subseteq I_k$, so $J_1=I_k$ or $J_1=I_{k+1}$, and so $J_1\in \LL$.  Assume that $J_m\in \LL$ for some $m\in\NN$. If $I_{k+m+1}\subseteq I_j$ for some $j\in \{ k, k+1 , \cdots,  k+m\}$, then $J_{m+1}=J_{m}\in \LL$. If not, $I_j\subseteq I_{k+m+1}$ for all $j\in \{ k , k+1 , \cdots,  k+m\}$, and so $J_{m+1}=I_{k+m+1}\in \LL$. Therefore, by the mathmatical induction, $J_i\in \LL$ for all $i\in \NN$. Moreover, $J_i\subseteq J_{i+1}$ for all $i\in \NN$. Since by the condition of $K$, $\displaystyle \bigcup_{i=1}^{\infty}J_i=\bigcup_{n=k}^{\infty} I_n$ is a nondegenerate open interval and so $\displaystyle \bigcup_{i=1}^{\infty}J_i\in \LL$, so  $\displaystyle \bigcup_{n=k}^{\infty} I_n\in\LL$.
 
 \
 From now on, let $L_i=\displaystyle \bigcup_{n=K+i}^{\infty}I_n$ for $i\in \NN$. Then, for all $i\in \NN$,  $L_i\in\LL$ by the above, and $L_{i+1}\subseteq L_i$. Since $J\subseteq \displaystyle \bigcap_{k=1}^{\infty}\bigcup_{n=k}^{\infty}I_n \subseteq \bar{J}$, and 
 $\displaystyle \bigcap_{k=1}^{\infty}\bigcup_{n=k}^{\infty}I_n= \bigcap_{k=K+1}^{\infty}\bigcup_{n=k}^{\infty}I_n
 =\bigcap_{i=1}^{\infty}L_i$, so $J\subseteq \displaystyle \bigcap_{i=1}^{\infty} L_i \subseteq \bar{J}$. So,  
 $J^*\subseteq \displaystyle \bigcup_{i=1}^{\infty} \overline{L_i^*} \subseteq \overline{J^*}$, and so 
 $ \displaystyle \bigcup_{i=1}^{\infty} L_i^*  \subseteq \bigcup_{i=1}^{\infty} \overline{L_i^*}\subseteq \overline{J^*}$. Therefore,  $ \displaystyle \bigcup_{i=1}^{\infty} L_i^*$ is nondegenerate. Since for all $i\in \NN$, $L_i^*\in \LL$, and $L_i^*\subseteq L_{i+1}^*$,    $ \displaystyle \bigcup_{i=1}^{\infty} L_i^*\in \LL$. Thus by Lemma \ref{lods}, $Int\Big(  \displaystyle \bigcap_{i=1}^{\infty} L_i\Big)=\Big( \bigcup_{i=1}^{\infty} L_i^*\Big)^*\in \LL$. Since 
 $\displaystyle J\subseteq \bigcap_{i=1}^{\infty} L_i\subseteq \bar{J}$, $J=Int\Big(  \displaystyle \bigcap_{i=1}^{\infty} L_i\Big)\in \LL$. 
\end{proof}

Moreover, by Proposition \ref{prop:symmetry of sequences} we can prove that if $\ell_n \rightarrow J$, then $\ell_n \rightarrow J^*$, and so $J^*\in \LL$. 
So, we can make the following definition.

\begin{defn}
Let $\LL$ be a lamination system, and $\{\ell_n\}_{n=1}^{\infty}$ be a sequence of leaves on $\LL$. Let $\ell$ be a leaf of $\LL$. Then, we say that $\{\ell_n\}_{n=1}^{\infty}$ \textsf{converges to} $\ell$ if $\ell_n \rightarrow I$ for some $I\in \ell$. 
\end{defn}

So far, we have talked about the definition of a lamination system. From now on, we discuss about the shape of lamination systems. First, as we can see in the proof of Proposition \ref{c}, the following structure is useful to deal with  configuration of leaves.

\begin{defn}
Let $\LL$ be a lamination system on $S^1$ and $I$ be an nondegenerate open interval. Then, for $p\in I$, we define $C_p^I$ as the set 
$C_p^I=\{ J\in \LL : p\in J \subseteq I \}.$
\end{defn}

As we observed in the proof of Proposition \ref{c}, $C_p^I$ is totally ordered by inclusion. 

\begin{prop}\label{a}
$C_p^{I}$ is totally ordered by the set inclusion $\subseteq$. 
\end{prop}
\begin{proof}
If $C_p^I$ has at most one element, it is true. Assume that $C_p^I$ has at least two elements. Let $J$ and $K$ be two distinct elements of $\LL$. If $J\subseteq K^*$, then $p\in L\subseteq K^*$  and so it is a contradiction since $p\in K$. If $K^*\subseteq J$, then $K^*\subseteq J \subseteq I$, and so $I^*\subseteq K$. But it contradicts to $K\subseteq I$. Thus, $J\subseteq K$ or $K\subseteq J$.
\end{proof}

The following lemma tells about the maximal and minimal elements of $C_p^I$.

\begin{lem}\label{loc}
Let $\LL$ be a lamination system on $S^1$, and $I$ be an nondegenerate open interval. Let $x$ be an element of $I$. Assume that $C_x^{I}$ is nonempty. Then, there is a sequence $\{J_n\}_{n=1}^{\infty}$ on $C_x^{I}$ such that for all $n\in \NN$, $J_n\subseteq J_{n+1}$, and $\displaystyle \bigcup_{n=1}^{\infty} J_n= \bigcup_{K\in C_x^{I}}K $. Also, there is a sequence 
$\{K_n\}_{n=1}^{\infty}$  on $C_x^{I}$  such that for all $n\in \NN$,  $K_{n+1}\subseteq K_n$, and $\displaystyle\bigcap_{n=1}^{\infty} K_n= \bigcap_{K\in C_x^{I}}K $.
\end{lem}
\begin{proof}
First, we show the first statement. Since $x\in K$ for all $K\in C_x^I$, $\displaystyle \bigcup_{K\in C_x^I}K$ is a connected open set and since $\displaystyle \bigcup_{K\in C_x^{I}}K \subseteq I$ and so $I^c \subseteq \Big( \displaystyle \bigcup_{K\in C_x^{I}}K \Big)^c$,  $\displaystyle \bigcup_{K\in C_x^{I}}K$ is a nondegenerate open interval. So, 
we can write $ \displaystyle \bigcup_{K\in C_x^{I}}K =(u,v)_{S^1}$ for some $u, v\in S^1$ with $u\neq v$. Choose $z\in (v,u)_{S^1}$. We define a sequence $\{I_n\}_{n=1}^{\infty}$ of nondegenerate intervals as
$$I_n=zp^{-1}((p(z^{-1}u)+\frac{L}{3n}, p(z^{-1}v)-\frac{L}{3n}))$$
where $p$ is the stereographic projection map used in the definition of orientation, $L=p(z^{-1}v)-p(z^{-1}u)$. 
Then, for all $n\in \NN$, $I_n\subseteq I_{n+1}$ and $\displaystyle \bigcup_{n=1}^{\infty}I_n=(u,v)_{S^1}$. 

\

From now on, we construct  a sequence $\{K_n\}_{n=1}^{\infty}$ in $C_x^I$ such that $I_n\subseteq K_n$.  
For $n\in \NN$, we denote $ I_n=(p_n, q_n)_{S^1}$. Then since $\partial I_n\subseteq (u,v)_{S^1}$ and  $ \displaystyle \bigcup_{K\in C_x^{I}}K =(u,v)_{S^1}$, there are $K_{p_n}$ and $K_{q_n}$ in $C_x^I$ such that $p_n\in K_{p_n}$, and $q_n\in K_{q_n}$. By Proposition \ref{a}, 
$K_{p_n}\subseteq  K_{q_n}$ or $K_{q_n} \subseteq K_{p_n}$.  If $K_{q_n} \subseteq K_{p_n}$, then $\partial I_n \subset K_{p_n}$. So $(p_n, q_n)_{S^1}\subseteq K_{p_n}$ or $(q_n,p_n)_{S^1}\subseteq K_{p_n}$. Since $z\notin K_{p_n}$, so $I_n= (p_n, q_n)_{S^1}\subseteq K_{p_n}$. In this case, we set $K_n=K_{p_n}$ Likewise, if $K_{p_n}\subseteq K_{q_n}$, then $I_n\subseteq K_{q_n}$, and so we set $K_n=K_{q_n}$. 
\

For $n\in \NN$, we define $J_n$ as $J_n=\displaystyle \bigcup_{m=1}^{n} K_m$. As in the argument in Proposition \ref{c}, $J_n\in C_x^I$ for all $n\in \NN$. Then $\{J_n\}_{n=1}^{\infty}$ is a sequence on $C_x^I$ such that for all $n\in \NN$, $I_n\subseteq J_n\subseteq J_{n+1}\subseteq (u, v)_{S^1}$. Therefore, $\displaystyle \bigcup_{n=1}^{\infty} J_n=\bigcup_{n=1}^{\infty} I_n=(u,v)_{S^1}= \bigcup_{K\in C_x^{I}}K $.

\

Second statement can be also proved in the similar way. But it is more subtle. Let $A=\displaystyle \bigcap_{K\in C_x^I}K.$ Then, 
$$\displaystyle A=\bigcap_{K\in C_x^I}K\subseteq \bigcap_{K\in C_x^I}\overline{K}$$
and so 
$$ \displaystyle \bigcup_{K\in C_x^I} K^*=\bigcup_{K\in C_x^I}\overline{K}^c  \subseteq A^c
=\bigcup_{K\in C_x^I} K^c= \bigcup_{K\in C_x^I} \overline{K^*} \subseteq \overline{ \bigcup_{K\in C_x^I} K^* }.$$
So we get 
$$  \displaystyle \bigcup_{K\in C_x^I} K^* \subseteq A^c \subseteq   \overline{ \bigcup_{K\in C_x^I} K^* }.$$
Since for all $K\in C_x^I$, $x\in K \subseteq I$ and so $x\notin K^*$ and $I^*\subseteq K^*$, $ \displaystyle \bigcup_{K\in C_x^I} K^*$ is a  nonempty proper connected open set and so it is open interval in $S^1$.  Then, we can write
 $ \displaystyle \bigcup_{K\in C_x^I} K^*=(v,u)_{S^1}$ for some $u,v \in S^1$. Then 
 $$(v,u)_{S^1}\subseteq A^c \subseteq \overline{(v,u)}_{S^1}$$
 and so 
 $$ \overline{(v,u)}_{S^1}^c\subseteq A \subseteq (v,u)_{S^1}^c. $$
  Since $(v,u)_{S^1}$ is homeomorphic to $\RR$,  it is Lindelöff. So, there is a sequence $\{K_n\}_{n=1}^{\infty}$ on $C_x^I$ such that 
 $ \displaystyle \bigcup_{n=1}^{\infty} K_n^*=\displaystyle \bigcup_{K\in C_x^I} K^*=(v,u)_{S^1}$
 since $ \displaystyle \bigcup_{K\in C_x^I} K^*$ is an open cover of $(v,u)_{S^1}$.\\
 
 If $u=v$, then $\phi= \overline{(v,u)}_{S^1}^c\subseteq A \subseteq (v,u)_{S^1}^c=\{u\}.$ Since 
 $$\{x\} \subseteq \displaystyle \bigcap_{n=1}^{\infty} K_n\subseteq  \bigcap_{n=1}^{\infty} \overline{K_n}= \bigcap_{n=1}^{\infty} (K_n^*)^c=\Big( \bigcup_{n=1}^{\infty} K_n^* \Big)^c =(v,u)_{S^1}^c =\{u\},$$ so $\{x\}\subseteq \displaystyle \bigcap_{n=1}^{\infty} K_n \subseteq \{u\}.$ Therefore, 
 $\{x\}=\displaystyle \bigcap_{n=1}^{\infty} K_n = \{u\}.$ Thus, since $x\in A$ and $A\subseteq \{u\}$, 
and so $\{x\}=A=\{u\}$,  $\displaystyle \bigcap_{n=1}^{\infty} K_n =A=\{x\}.$ For each $n\in \NN$, define $J_n=\displaystyle \bigcap_{m=1}^{n}K_m$. Then, for all $n\in \NN$, $J_n \in C_x^I$ and $J_{n+1}\subseteq J_n$. Thus,  since $\displaystyle \bigcap_{n=1}^{\infty}K_n=\bigcap_{n=1}^{\infty}J_n$ , the sequence $\{J_n\}_{n=1}^{\infty}$ is a sequence that we want.  \\
 
 If $u\neq v$, then 
  $$(u,v)_{S^1}= \overline{(v,u)}_{S^1}^c\subseteq A \subseteq (v,u)_{S^1}^c=[u,v]_{S^1}. $$
There are four cases :   $A=(u,v)_{S^1}$, $A=[u,v]_{S^1}$, $A=[u,v)_{S^1}$ and $A=(u,v]_{S^1}.$

First, if  $A=[u,v]_{S^1}$, then 
$$[u,v]_{S^1}=A=\displaystyle \bigcap_{K\in C_x^I}K\subseteq \displaystyle \bigcap_{n=1}^{\infty}K_n\subseteq \bigcap_{n=1}^{\infty}\overline{K_n}=\bigcap_{n=1}^{\infty}(K_n^*)^c=\Big(\bigcup_{n=1}^{\infty} K_n^*\Big)^c =[u,v]_{S^1}.$$ Therefore, 
$$[u,v]_{S^1}=A=\displaystyle \bigcap_{n=1}^{\infty}K_n.$$ Then, for each $n\in \NN$, define $J_n=\displaystyle \bigcap_{m=1}^{n}K_m$. By the construction of $\{J_n\}_{n=1}^{\infty}$,  for all $n\in \NN$, $J_n \in C_x^I$ and $J_{n+1}\subseteq J_n$. Thus, since $\displaystyle \bigcap_{n=1}^{\infty}K_n=\bigcap_{n=1}^{\infty}J_n$, the sequence $\{J_n\}_{n=1}^{\infty}$ is a sequence that we want.  

Next,  if $A=(u,v)_{S^1}$, then there are $K_u$ and $K_v$ in $C_x^I$ such that $u\notin K_u$ and $v\notin K_v$. Since $C_x^I$ is totally ordered, so $K_u\subseteq K_v$ or $K_v\subseteq K_u$. Therefore, one of $K_u$ and $K_v$, say $K'$, does not intersect with $\{u,v\}$. Since
 $$(u,v)_{S^1}=A\subseteq K' \subseteq S^1 -\{u,v\}=(u,v)_{S^1} \cup (v,u)_{S^1} $$
 and $K'$ is connected, so $K'=(u,v)_{S^1}$. Therefore, $(u,v)_{S^1} \in C_x^I$ and so, for each 
 $n\in \NN$, define $J_n=(u,v)_{S^1}$. Then the sequence $\{J_n\}_{n=1}^{\infty}$ is a sequence that we want.  
 
 If $A=(u,v]_{S^1}$, then there is an element $L$ in $C_x^I$ such that $u \notin L$. Define $C=\{K\in C_x^I : K\subseteq L\}.$ Since $C\subseteq C_x^I,$ $C$ is also totally ordered by the inclusion. Note that $A=\displaystyle \bigcap_{K\in C_x^I} K=\bigcap_{K\in C} K$ and $\displaystyle \bigcup_{K\in C_x^I} K^*=\bigcup_{K\in C} K^*$. So, since $\displaystyle \bigcup_{K\in C}K^*$ is a nondegenerate open interval and so is Lindelöff, there is a sequence $\{L_n\}_{n=1}^{\infty}$ of $C$ such that $\displaystyle \bigcup_{K\in C}K^*= \bigcup_{n=1}^{\infty}L_n^*.$
 Then 
 $$(u,v]_{S^1}=A=\bigcap_{K\in C} K \subseteq \bigcap_{n=1}^{\infty} L_n  \subseteq \bigcap_{n=1}^{\infty} \overline{L_n}=  \bigcap_{n=1}^{\infty} (L_n^*)^c=\Big(\bigcup_{n=1}^{\infty} L_n^*\Big)^c=[u,v]_{S^1}.$$ 
 So we get 
 $$(u,v]_{S^1}=A \subseteq  \displaystyle \bigcap_{n=1}^{\infty} L_n \subseteq [u,v]_{S^1}.$$
 Since for all $n\in \NN$, $u\notin L_n$, $\displaystyle (u,v]_{S^1}=A=\bigcap_{n=1}^{\infty} L_n.$ Then, for each $n\in \NN$, define $J_n=\displaystyle \bigcap_{m=1}^{n}L_m$. By the construction of $\{J_n\}_{n=1}^{\infty}$,  for all $n\in \NN$, $J_n \in C_x^I$ and $J_{n+1}\subseteq J_n$. Thus, the sequence $\{J_n\}_{n=1}^{\infty}$ is a sequence that we want. The proof of the case $A=[u,v)_{S^1}$ is similar to the case $A=(u,v]_{S^1}.$

\end{proof}

Note that if $C_p^{I}\neq \phi$, then by Lemma \ref{loc}, $\displaystyle \bigcup_{J\in C_p^{I}}J\in \LL$.
On a lamination system $\LL$, when a sequence of leaves converges to a leaf $\ell$, it approaches to $\ell$ in two different sides of $\ell$. Geometrically, if there is no converging sequence of leave in one side, then there is a non leaf gap in that side. To describe this situation, we use the following definition.
\begin{defn}
Let $\LL$ be a lamination system on $S^1$, and $I\in \LL$. Let $\{ \ell_n\}_{n=1}^{\infty}$ be a sequence of leaves of $\LL$. Then we call $\{\ell_n \}_{n=1}^{\infty}$ a \textsf{$I$-side sequence} if for all $n\in \NN$, $I\notin \ell_n$, and $\ell_n$ lies on $I$, and $\ell_n\rightarrow I$. And we say that $I$ is \textsf{isolated} if there is no $I$-side sequence on $\LL$. Moreover, a leaf $\ell$ is \textsf{isolated} if each elements of $\ell$ is isolated. 
\end{defn}

The following Lemma show that the previous statement is true. 

\begin{lem}\label{eog}
Let $\LL$ be a lamination system on $S^1$ and $I\in \LL$. Suppose that $I^*$ is isolated. Then, there is a non leaf gap $\GG$ such that $I \in \GG$.
\end{lem}
\begin{proof}
For any point $p$ in $I^*$ at which $2\leq|C_p^{I^*}|$, we define $J_p=\displaystyle \bigcup_{J\in C_p^{I^*}-\{I^*\}} J$. Then $J_p$ is a nondegenerate open interval since  $J_p$ is a nonempty connected open subset of $S^1$ and $I\subseteq {J_p} ^c$. By applying  Lemma $\ref{loc}$ to $C_p^{J_p}$, $J_p\in \LL$ and since $I^*$ is isolated, $J_p\subsetneq I^*$.  
\\
Note that $J_p$ and $J_q$ are disjoint or coincide whenever $p\neq q\in I^*$ .
\\
Define $\GG=\{I\} \cup \{ J_p : p\in I^* \  and \ 2 \leq |C_p^{I^*}| \}$. Then if there is  $K\in \LL$ with $\ell(K)\neq \ell(I)$, then $\ell(K)$ lies on $I$ or $I^*$. If $\ell(K)$ lies on $I^*$, then there is $L$ in $\ell(K)$ which is contained on $I^*$. Since for any $x$ in $L$, $L\in C_x^{I^*}$, $\ell(K)$ lies on $J_x$ for any $x\in L$. Thus, $\GG$ is the non-leaf gap in which $I$ is.

\end{proof}

The following lemma is about the configuration of two gaps. It is a kind of generalization of unlinkedness condition of two leaves to unlinkedness  condition of two gaps.

\begin{lem}\label{tg} 
Let $\LL$ be a lamination system on $S^1$, and $\GG$, $\GG'$ be two gaps with $|\GG|,|\GG'|\ge2$. Then, 
$\GG=\GG'$ or there are $I$ in $\GG$, and $I'$ in $\GG'$ such that $I^*\subseteq I'$, and for all $J\in \GG$, $\ell(J)$ lies on $I'$ ,and for all $J'\in \GG'$, $\ell(J')$ lies on $I$.   
\end{lem}
\begin{proof}
Assume that $\GG\neq \GG'$. If $\GG \subsetneq \GG'$, then there is $I$ in $\GG' -\GG$. Then, since 
$\GG$ is a gap, there is $J$ in $\GG$ on which $\ell(I)$ lies. Since $\GG'$ is a gap, so for all $K\in \GG'-\{I\}$, $I\cap K=\phi $, and so $I\cap J=\phi$. Since $I\cap J=\phi$, $I\subseteq J^*$and since $\ell(I)$ lies on $J$  and $I\cap J=\phi$, so $I^*\subseteq J$ and $J^*\subseteq I$. Therefore, $I=J^*$, and so  $\{I, I^*\} \subseteq  \GG'$. Thus, $\GG'$ is a leaf $\ell(I)$, and so $\GG$ is a one point subset of $\GG'$. It is not an our case. Also, we can get that $\GG'\subsetneq \GG$ is not possible. 
\

So, there are $J$ in $\GG-\GG'$ and $J'$ in $\GG'-\GG$. Since $\GG$ and $\GG'$ are gaps, there are $K$ in $\GG$ and $K'$ in $\GG'$ such that $\ell(J')$ lies on $K$ and $\ell(J)$ lies on $K'$. Since $\GG'$ is a gap, $\ell(K)$ lies on $L'$ for some $L'\in\GG'$. Likewise, since $\GG$ is a gap, $\ell(K')$ lies on $L$ for some $L\in \GG$.
\

 First, consider the case $L'\neq J' $ . Since $\ell(K)$ lies on $L'$, $K\subseteq L'$ or $K^*\subseteq L'$. If $K\subseteq L'$, then $J' \subseteq K$ can not occur, and so $(J')^*\subseteq K\subseteq L'$ since $\ell(J')$ lies on $K$. Since by the assumption, $L'\cap J'=\phi$, so $L'\subseteq (J')^*$. Therefore $(J')^*=K=L'$, and so $\GG'$ should be a leaf $\ell(J')$. Then since $K=(J')^*$, and $I\subseteq K^*$ for all $I\in \GG-\{K \}$ , so  $I\subseteq J'$ for all $I\in \GG -\{K\}$. Therefore, for all $I\in \GG$, $\ell(I)$ lies on $J'$. Since $K^*=J'\subseteq J'$ and, trivially, $\ell(J')$ lies on $K$, $K$ and $J'$ are the elements that we want to find. On the other hands, if $K^* \subseteq L'$, then  for all $I\in \GG-\{K\}$, $I\subseteq K^* \subseteq L' $ since $I\subseteq K^*$. Therefore, for all $I\in \GG$, $\ell(I)$ lies on $L'$. And by the assumption, $(L')^*\subseteq K$. Likewise, for all $I'\in \GG'$, $\ell(I')$ lies on $K$. So, in this case, $K$ and $L'$ are the elements that we want.
 \
 
 Second, consider the case $L'=J'$. Then $\ell(J')$ lies on $K$ and $\ell(K)$ lies on $J'$. If $J'\subseteq K$, then $K^*\subseteq (J')^*$. Since $\ell(K)$ lies on $J'$, there are two possibility. One is $K\subseteq J'$, and so $K=J'$. However, it contradicts to $J' \in \GG' - \GG$. The other is $K^*\subseteq J'$, but it also contradicts to $K^*\subseteq (J')^*$. Therefore, $(J')^*\subseteq K$. Since $\GG'$ is a gap, for all $I'\in \GG'-\{J'\}$, $I'\subseteq (J')^*$, and so for all $I'\in \GG'$, $\ell(I')$ lies on $K$. And by the assumption, $K^*\subseteq J'$. Likewise, for all $I\in \GG$, $\ell(I)$ lies on $J'$. Thus, $K$ and $J'$ are the elements that we want.

\end{proof}


On a lamination system $\LL$ on $S^1$,   a gap $\GG$ with $|v(\GG)| < \infty$ is called 
 an \textsf{ideal polygon}. In an ideal polygon $\GG$, since $v(\GG)$ is a finite set, we can write $v(\GG)=\{x_1, x_2, \cdots ,x_n\}$ where $|v(\GG)|=n$, and $(x_1, x_2, \cdots, x_n)$ is positively oriented $n$-tuple. Moreover, we can represent $\GG=\{ (x_1,x_2)_{S^1}, (x_2, x_3)_{S^1}, \cdots, (x_{n-1},x_n)_{S^1}, (x_n,x_1)_{S^1} \}$. Then we say that a lamination system  $\LL$ is \textsf{very full} if every gap of $\LL$ is an ideal polygon. 
 Let $E(\LL)=\displaystyle \bigcup_{I\in \LL}v(\ell(I))$ and call it the \textsf{end points set} of $\LL$. A lamination system $\LL$ is called \textsf{dense} if $E(\LL)$ is a dense subset of $S^1$. Let $p\in S^1$ and $\LL$ be a dense lamination system. Suppose that there is a sequence $\{I_n\}_{n=1}^{\infty}$ on $\LL$ such that for
all $n\in \NN$, $I_{n+1}\subseteq I_n$, and $\displaystyle
\bigcap_{n\in \NN} I_n=\{p\}$. We call such a sequence a
\textsf{rainbow} at $p$. In \cite{BaikFuchsian}, it is observed that
very full laminations have abundant rainbows (see Theorem \ref{er} for
a precise statement). We recall some results from \cite{BaikFuchsian}
and \cite{alonso2019laminar} about invariant laminations in the rest
of the section, and give alternative proofs in the language of
lamination systems. 

\begin{thm}[\cite{BaikFuchsian}]\label{er}
Let $\LL$ be a very full lamination system. For $p\in S^1$, either $p$ is in $E(\LL)$ or $p$ has a rainbow. These two possibilities are mutually exclusive. 
\end{thm}
\begin{proof}
Let $p$ be a point of $S^1$. First we show that if there is no $I\in \LL$ such that $p\in v(\ell(I))$, $p$ has a rainbow. Assume that there is no $I$ in  $\LL$ such that $p\in v(\ell(I))$. Since $\LL$ is nonempty, there is an element $I$ in $\LL$. Then,  by the assumption, $p\notin v(\ell(I))=\partial I$. Since $S^1$ has a partition $\{I , \partial I , I^*\}$, so $p$ belongs to either $I$ or $I^*$. Say that $p\in I$. Then, $C_p^{I}$ is nonempty. By Lamma \ref{loc}, there is a sequence $\{K_n\}_{n=1}^{\infty}$ on $C_p^{I}$ such that for all $n\in \NN$, $K_{n+1}\subseteq K_n$, and $\displaystyle \bigcap_{n=1}^{\infty}K_n=\bigcap_{K\in C_p^{I}} K$.  If $ \displaystyle \bigcap_{n=1}^{\infty}K_n =\{p\}$, we are done. If not, $\displaystyle \bigcup _{n=1}^{\infty}K_n^*$ is a nondegenerate open interval $J$ with $p\in J^c=\overline{J^*}$ and $J\in \LL$, and by Lemma \ref{lods}, $int\Big(\displaystyle \bigcap_{n=1}^{\infty}K_n \Big)=J^*\in \LL$. If $p\in \partial J^*$,then $p\in E(\LL)$ and so it contradicts to the assumption. So $p\in J^*$. Then $p \in J^* \subseteq I$, and so $J^*$ is the minimal element of $C_p^I$. Now, we want to show that $J^*$ is isolated. Suppose that there is a $J^*$-side sequence $\{\ell_n\}_{n=1}^{\infty}$. So there is a sequence $\{I_n\}_{n=1}^{\infty}$ on $\LL$ such that for all $n\in \NN$, $\ell_n=\ell(I_n)$, and 
$$J^*\subseteq \liminf I_n \subseteq \limsup I_n \subseteq \overline{J^*}.$$
So, since $p\in J^*\subseteq \liminf I_n$, there is $m \in \NN$ such that $\displaystyle p\in \bigcap_{n=m}^{\infty} I_n. $ Therefore, for all $k\geq m$, $p \in I_k$ and so $I_k\nsubseteq J$. And then choose $q \in J$. Since $\limsup I_n \subseteq \overline{J^*}$, there is $m' \in \NN$ such that $\displaystyle q\notin \bigcup_{n=m'}^{\infty}I_n $. Therefore, for all $k\geq m'$, $J \nsubseteq I_k$. So, for $k\geq \max\{m, m'\}$, $J^*\subseteq I_k$ or $I_k\subseteq J^*$. Since for all $n\in \NN$, $\ell_n\neq \ell(J^*)$, so  for $k\geq \max\{m, m'\}$, $J^*\subsetneq I_k$ or $I_k\subsetneq J^*$. Since for all $n\in \NN$, $\ell_n$ lies on $J^*$, so for $k\geq \max\{m, m'\}$,  $I_k\subsetneq J^*$. Moreover, for $k\geq \max\{m, m'\}$,  $p\in I_k\subsetneq J^*\subseteq I$, and so $I_k\in C_p^{I}$. It is a a contradiction to minimality of $J^*$ on $C_p^{I}$. Therefore, $J^*$ is isolated. Then by Lemma \ref{eog}, there is the non leaf gap $\GG$ such that $J\in \GG$. Since $\LL$ is very full, $v(\GG)$ is finite, and so $\displaystyle \bigcup_{I\in \GG}\partial I=v(\GG)$. Note that  $S^1$ has a partition $\GG \cup \{ v(\GG)\}$. By the assumption,  $p\notin v(\GG)$,  and so there is $K\in \GG - \{J\}$ such that $p\in K$. Since $K\subsetneq J^*\subseteq I$,  $K\in C_p^I$. but it contradicts to the minimality of $J^*$ on $C_p^{I}$. Thus, $ \displaystyle \bigcap_{n=1}^{\infty}K_n =\{p\}$.
\

Finally, we want to show that if there is a leaf $\ell$ such that $p\in v(\ell)$, $p$ has no rainbow. Suppose that there are a rainbow $\{I_n\}_{n=1}^{\infty}$, and a leaf $\ell$ such that $p\in v(\ell)$. Since for all $n \in \NN$, $p\in I_n$, so $\ell$ lies on $I_n$. Choose $n\in \NN$. Then there is an element $I$ in $\ell$ such that $I\subsetneq I_n$. If $I^*\subsetneq I_{n+1}$, then $I^*\subsetneq I_{n+1} \subseteq I_n$, but it is a contradiction. So, $I\subsetneq I_{n+1}$. Therefore, $I\subseteq \displaystyle \bigcap_{n=1}^{\infty} I_n $ , but it is not possible since $\displaystyle \bigcap_{n=1}^{\infty} I_n = \{p\}$.   
\end{proof}

\begin{cor}[\cite{BaikFuchsian}]\label{ed}
Let $\LL$ be a very full lamination system of $S^1$. Then, $E(\LL)$ is dense on $S^1$.
\end{cor}
\begin{proof}
Suppose that $E(\LL)$ is not dense. 
Then, there is a point $p$ in $S^1$ which has an open neighborhood $K$ which is a nondegenerate open interval with $E(\LL)\cap \overline{K}=\phi$. And by Theorem \ref{er}, there is a rainbow $\{I_n\}_{n=1}^{\infty}$ at $p$. Fix $n\in \NN$ and denote $I_n=(u_n,v_n)_{S^1}$ and $K=(s,t)_{S^1}$. Let $\phi$ be the circular order of $S^1$. Note that $\phi(u_n,p,v_n)=1.$
 Since  $E(\LL)\cap \overline{K}=\phi$, $\phi(s, u_n, t)=\phi(s, v_n, t)=-1$ and so $\phi(t, u_n, s)=\phi(t, v_n, s)=1$. Since $\phi(s, p, t)=1$ and $\phi(t, s, p)=1$, so $\phi(t, u_n, p)=(t, v_n, p)=1$. Since by the cocycle condition on the four points $(t,u_n,v_n,p)$,
$$\phi(u_n,v_n,p)-\phi(t, v_n, p)+\phi(t,u_n,p)-\phi(t,u_n, v_n)=0,$$ so  $$-1-1+1-\phi(t,u_n,v_n)=0,$$ Hence $\phi(t, u_n, v_n)=-1$. Therefore, $\phi(u_n,t,v_n)=1$. Likewise, $\phi(s, u_n, t)=\phi(s, v_n, t)=-1$ and so $\phi(s, t, u_n)=\phi(s, t, v_n )=1$. Since $\phi(s,p,t)=1$, so $\phi(s,p,u_n)=\phi(s, p, v_n)=1$. Since by the cocycle condition on the four points $(s, p, u_n , v_n )$,
 $$\phi(p,u_n, v_n)-\phi(s,u_n, v_n)+\phi(s, p, v_n)-\phi(s, p,u_n)=0,$$
  
so $$-1-\phi(s,u_n,v_n)+1-1=0.$$ Hence  $\phi(s,u_n, v_n)=-1$. Therefore, $\phi(u_n,s,v_n)=1$. 
\\
We have shown that $\phi(u_n, s, v_n)=\phi(u_n, t,v_n)=1$ and we have that $\phi(s,u_n, t)=\phi(s, v_n, t)=-1$ sine $E(\LL)\cap \overline{K}=\phi$.
From now on, we show that $K\subseteq I_n$. Let $q$ be a point in $K$. Then $\phi(s,q,t)=1$ and since $\phi(s,u_n, t)=\phi(s, v_n, t)=-1$ and so $\phi(s,t, u_n)=\phi(s, t,v_n)=1$, we get that $\phi(s,q,u_n)=\phi (s,q, v_n)=1$. Then by applying the cocycle condition to four points $(u_n, s,q,v_n)$, 
$$\phi(s, q, v_n)-\phi(u_n, q, v_n)+\phi(u_n, s, v_n)-\phi(u_n, s, q)=0.$$
Since $\phi(u_n, s, q)=\phi(s, q, u_n)$, 
$$\phi(s, q, v_n)-\phi(u_n, q, v_n)+\phi(u_n, s, v_n)-\phi(u_n, s, q)=1-\phi(u_n, q, v_n)+1-1=0.$$
Therefore, $\phi(u_n, q, v_n)=1$ and so $q\in I_n$. We are done.
 It implies that for all $n\in \NN$, $K\subseteq I_n$, so $K\subseteq \displaystyle \bigcap_{n=1}^{\infty} I_n$, but it is a contradiction by the definition of a rainbow. Thus, $E(\LL)$ is dense. 
\end{proof}

Indeed, very fullness does not guarantee the existence of non-leaf gaps. More precisely, a lamination system, of which the geometric realization is a geodesic lamination which foliates the whole hyperbolic plane, is very full, but there is no non leaf gap. So, we need some notions to rule out this situation and to guarantee the existence of a gap on a lamination system. So, the following  definitions on a lamination system describe the situation  which is analogous to that in $\overline{\HH^2}$, there is no open disk foliated by  leaves on a given geodesic lamination which is a geometric realization of a lamination on $S^1$ .
\begin{defn}
Let $\LL$ be a lamination system and $\{I, J\}$ be a subset of $\LL$. Then, $\{I, J\}$ is called a \textsf{distinct pair} if $I\cap J =\phi$, and $\{I,J\}$ is not a leaf. 
A distinct pair $\{I,J\}$ is \textsf{separated} if there is a non-leaf gap $\GG$ such that $I\subseteq K$ and $J\subseteq L$ for some $K, L\in \GG$,not necessarily, $K\neq L$. And  $\LL$ is \textsf{totally disconnected} if every distinct pair is separated. 
\end{defn}

Two lamination systems $\LL_1$ and $\LL_2$ have \textsf{distinct endpoints} if $E(\LL_1)\cap E(\LL_2)=\phi$. When we study with two lamination systems, the distinct endpoints condition enforces  totally disconnectedness to lamination systems.

\begin{lem}[\cite{alonso2019laminar}]\label{tot dis}
If two dense lamination systems have distinct endpoints, then each of the lamination systems is totally disconnected. 
\end{lem}
\begin{proof}
Let $\LL_1$ and $\LL_2$ be two dense lamination systems with distinct endpoints. First, we show that  $\LL_1$ is totally disconnected. Suppose that a subset $\{I,J\}$ of $\LL_1$ is a distinct pair. $I^* \cap J^*$ is a non-empty open set. Since $E(\LL_2)$ is dense on $S^1$, so we can choose $p\in I^*\cap J^*\cap E(\LL_2)$. If there is $K\in \LL_1$ such that $p\in K \subseteq I^*\cap J^*$, then $K\in C_p^{I^*} \cap C_p^{J^*}$  and so $C_p^{I^*}\cap C_p^{J^*}$ is nonempty where we consider  $C_p^{I^*}$ and $C_p^{J^*}$ on $\LL_1$ . Note that $C_p^{I^*}\cap C_p^{J^*}$ is totally ordered by $\subseteq$. Let $M$ be the union of elements of $C_p^{I^*}\cap C_p^{J^*}$. Then $M$ is a nondegenerate open interval with $p\in  M \subseteq I^*\cap J^*$. So, $C_p^M$ is equal to $C_p^{I^*}\cap C_p^{J^*}=\{ K \in \LL_1 : p\in K\subseteq I^*\cap J^*\}$. Moreover, by Lemma \ref{loc}, $M\in \LL_1$. 
 
  We want to show that $M^*$ is isolated on $\LL_1$. Suppose that  $\{\ell_n\}_{n=1}^{\infty}$ be a $M^*$-side sequence of leaves of $\LL_1$. Then, there is $\{I_n\}_{n=1}^{\infty}$ such that for all $n\in \NN$, $\ell_n=\ell(I_n)$, and 
$$M^*\subseteq \liminf I_n \subseteq \limsup I_n\subseteq  \overline{M^*}.$$  Choose  $p_I\in I$ and $p_J\in J$. Since $M^*\subseteq \liminf I_n $ and $\{p_I, p_J\}\subseteq I\cup J \subseteq M^*$, there is  $m\in \NN$ such that $\displaystyle \{p_I, p_J\} \subseteq \bigcap_{n=m}^{\infty} I_n$. Therefore, for all $n\geq m$, $\{p_I,p_J\} \subseteq I_n$ and so $I_n\nsubseteq M$. Also, since  $\limsup I_n\subseteq  \overline{M^*}$ and $p \in M$, there is $m'$ such that $\displaystyle p\notin \bigcup_{n=m'}^{\infty}I_n$. Therefore, for all $n\geq m'$, $p\notin I_n$ and so $M\nsubseteq I_n$. Hence, for all $n\geq \max \{m, m'\}$, $I_n\subseteq M^*$ or $M^*\subseteq I_n$. Moreover, for $n\geq \max\{m,m'\}$, since $\ell_n\neq \ell(M^*)$, $I_n\subsetneq M^*$ or $M^*\subsetneq I_n$ and since $\ell_n$ lies on $M^*$,   $I_n\subsetneq M^*$ is the possible case. Thus, for all $n\geq \max \{m, m'\}$, $\{p_I,p_J\}\subseteq I_n \subsetneq M^*$. 

 Then, fix $n\geq \max \{m, m'\}$. If $I^*\subseteq I_n$ or $J^*\subseteq I_n$, then $M\subseteq I_n$, and it contradicts to $I_n\subseteq M^*$. If $I_n\subseteq I^*$ or $I_n \subseteq J^*$, then $\{p_I, p_J\} \subseteq I^*$ or $\{p_I, p_J\} \subseteq J^*$, respectively, and so it is also a contradiction since $I\cap I^*= \phi$ and $J\cap J^*=\phi$.   If $I_n\subseteq I$ or $I_n \subseteq J$, then $\{p_I, p_J\} \subseteq I$ or $\{p_I, p_J\} \subseteq J$, respectively, and so it is also a contradiction since $I\cap J =\phi$. Therefore, $I\cup J \subseteq I_n$. Hence $M\subsetneq I_n^*\subseteq I^*\cap J^*$. It contradicts to the maximality of $M$ on $C_p^{I^*}\cap C_p^{J^*}$. Thus, $M^*$ is isolated. 
 
 Finally, by Lemma \ref{eog}, there is a non leaf gap $\GG$ such that $M\in \GG$. If there is $L\in\GG$ such that $I^*\subseteq L$ or $J^*\subseteq L$, then $M\subseteq I^*\subseteq  L$ or $M\subseteq J^*\subseteq  L$ and so $M=I^*=L$ or  $M=J^*=L$, respectively. Then, $J\subseteq I^*= M$ or $I\subseteq J^* = M$, respectively. However, it is a contradiction since $M\cap \{p_I,p_J\}=\phi$. Therefore, so by the definition of gap, there are $L$ and $L'$ in $\GG$ such that $I\subseteq  L$ and $J\subseteq L'$. So, $\{I, J\}$ is separated.
\

Next, assume that $C_p^{I^*}\cap C_p^{J^*}=\phi$. Choose $K\in C_p^{I^*}$. If $K \subseteq J$, then it contradicts to $p\in J^*$. If $K \subseteq J^*$, then $K\in C_p^{J^*}$, and so it is a contradiction by the assumption. If $J^* \subseteq K$, then $J^* \subseteq K \subseteq I^*$, and $I\subseteq J$, so it is a contradiction by the definition of distinct pairs. Therefore, $J\subsetneq K$, and so $(\{p\}\cup J )\subseteq K$. By  Lemma \ref{loc}, there is a sequence $\{F_n\}_{n=1}^{\infty}$ such that for all $ n\in \NN$, $F_{n+1}\subseteq F_n$ and $\displaystyle \bigcap_{n=1}^{\infty}F_n=\bigcap_{F\in C_p^{I^*} } F$ and since $(\{p\}\cup J )\subseteq F$ for all $F\in C_p^{I^*}$, $\displaystyle \bigcup_{n=1}^{\infty}F_n^*$ is a nondegenerate open interval and so $\displaystyle \bigcup_{n=1}^{\infty}F_n^*\in \LL_1$. Therefore, by Lemma $\ref{lods}$, $Int \Big(\displaystyle \bigcap_{n=1}^{\infty} F_n \Big)=Int \Big( \bigcap_{F\in C_p^{I^*}} F \Big)$ is a nondegenerate open interval $N$ with $p\in \overline{N}$, $J\subseteq N$  and $N\in \LL_1$. Since $p$ is not in $E(\LL_1)$, so $p\in N$. 

 Then, we want to show that $N$ is isolated. Suppose that there is a $N$-side sequence $\{\ell_n\}_{n=1}^{\infty}$ on $\LL_1$. So, there is a sequence $\{I_n\}_{n=1}^{\infty}$ on $\LL_1$ such that for all $n\in \NN$, $\ell_n=\ell(I_n)$ and 
$$N\subseteq \liminf I_n \subseteq \limsup I_n \subseteq \overline{N}.$$
Choose $q\in I$. Since $\limsup I_n\subseteq \overline{N}$, there is $m\in \NN$ such that $q \notin \displaystyle \bigcup_{n=m}^{\infty} I_n$. Also, since $N \subseteq \liminf I_n $, there is $m'\in \NN$ such that $p \in \displaystyle \bigcap_{n=m'}^{\infty}I_n$. Therefore, for all $n\geq \max \{m,m'\}$, $p\in I_n$ and $q\in I_n^c$. 

Fix $n\geq \max \{m, m'\}$. If $I_n \subseteq N^*$, then $p\in I_n \subseteq N^*$ and it is a contradiction since $p\in N$. If $N^*\subseteq I_n$, then $q\in I_n^c \subseteq (N^*)^c$ and it is a contradiction since $q\in I \subseteq N^*$. Therefore, $I_n\subseteq N$ or $N\subseteq I_n$. Moreover, since $\ell_n\neq \ell(N)$,  $I_n\subsetneq N$ or $N\subsetneq I_n$ and  since $\ell_n$ lies on $N$, $I_n\subsetneq N$ is the possible case. But $p \subseteq I_n \subsetneq N$, and it contradicts to the minimality of $N$ on $C_p^{I^*}$. So, by Lemma \ref{eog}, there is a non-leaf gap $\GG$ such that $N^*\in \GG$. It is enough to show that there is $K$ in $\GG$ such that $J \subseteq K$. Suppose that there is $K'$ in $\GG$ such that $J^* \subseteq K'$. then $N^* \subseteq J^*\subseteq K'$ and so $N^* =J^*=K'$. But, it implies $N=J$ and since $p\in N$ and $p\notin J$, it is a contradiction.  Therefore, there is $K$ in $\GG$ such that $J\subseteq K$. Thus, $\{I,J\}$ is separated and so $\LL_1$ is totally disconnected. In the same reason, $\LL_2$ is also totally disconnected.  
\end{proof}


We have introduced lamination systems as a
model for laminations on the circle. In this perspective, laminar
groups are groups acting on the circle with invariant lamination
systems. We end this section, discussing about actions on lamination systems . A homeomorphism $f$ on $S^1$ is \textsf{orientation
  preserving} if for any positively oriented triple $(z_1,z_2, z_3)$,
$(f(z_1),f(z_2), f(z_3))$ is a positively oriented triple.
We denote the set of orientation preserving homeomorphisms on $S^1$ as $\Homeop (S^1)$ and the set of fixed points of $f$ as $\Fix_f$. 
Note that if $f \in \Homeop(S^1)$, for $u,v\in S^1$, we have $f((u,v)_{S^1})=(f(u), f(v))_{S^1}$ and if $u\neq v$, then $f((u,v)_{S^1}^*)=f((u,v)_{S^1})^*$. 

\begin{defn}
Let $\LL$ be a lamination system, and $G$ be a subgroup of $\Homeop(S^1)$. $\LL$ is called a \textsf{G-invariant lamination system} if for any $I\in \LL$ and $g\in G$, $g(I)\in \LL$.
When $\LL$ is a $G$-invariant lamination system, the action of $G$ on
$\LL$ is said to be \textsf{minimal} if for any two leaves $\ell, \ell' \in \LL$, there is a sequence $\{g_n\}_{n=1}^{\infty}$ on $G$ such that $g_n(\ell')\rightarrow \ell$.
\end{defn}

First, note that on a $G$-invariant lamination system, every gap is mapped to a gap and the converging property is preserved under the given $G$-action. And when we consider a minimal action on a lamination system which has a non-leaf ideal polygon, the orbit of end points of the ideal polygon is usually dense in $S^1$. This denseness gives the following lemma which is useful to analyze the action.

\begin{lem}\label{mg}
Let $G$ be a subgroup of $\Homeop(S^1)$, and $\LL$ be a $G$-invariant lamination system. Assume that there is an ideal polygon $\GG$ which is not a leaf, and $v_G(\GG)=\displaystyle \bigcup_{g\in G} v(g(\GG))$ is dense in $S^1$. Then, for each $I\in \GG$, there is an element $g_I$ in $G$ such that for any $J\in g_I(\GG)$, $\ell(J)$ properly lies on $I$, equivalently  $v(g_I(\GG))\subseteq I$. 
\end{lem}
\begin{proof}
Choose $I\in \GG$. Since $v_G(\GG)$ is dense in $S^1$, there is $p\in v_G(\GG) \cap I$. By the definition of $v_G(\GG)$, there is an element $g$ in $G$ such that $p\in v(g(\GG))$. Note that $\GG\neq g(\GG)$ since $v(\GG)\neq v(g(\GG))$. Since $\GG$ is an ideal polygon and so $3\leq |\GG|$, by Lemma \ref{tg}, there are $J$ in $\GG$, and $J'$ in $g(\GG)$ such that $(J')^*\subseteq J$. Hence, for all $K\in \GG$, $\ell(K)$ lies on $J'$, and for all $K'\in g(\GG)$, $\ell(K')$ lies on $J$. Since for all $K'\in g(\GG)$, $\ell(K')$ lies on $J$, so $v(g(\GG)) \subseteq \bar{J}$. Therefore, since $p \in I$, so $J$ should be $I$. Since $(J')^*\subseteq J$, for each $I'\in g(\GG)-\{ J' \}$, $I'\subseteq J$. Choose $I'$ in $g(\GG)-\{ J'\}$.  If $I'=J$, then $J'\subseteq (I')^* =J^* $ and so $J'=(I')^* = J^* $since  $(J')^*\subseteq J$. But in this case, $g(\GG)$ is the leaf $\ell(I')$, and so it contradicts to the assumption. Therefore, $I'\subsetneq J= I$. 
\

Then we do the same process in $I'$. Since $v_G(\GG)$ is dense in $S^1$, there is $p'\in v_G(\GG)\cap I'$.
By the definition of $v_G(\GG)$, there is an element $g'$ in $G$ such that $p'\in v(g'(\GG))$. By Lemma \ref{tg}, there are $L$ in $g(\GG)$ and $L'$ in $g'(\GG)$ such that $(L')^*\subseteq L$. Hence, for all $K\in g(\GG)$, $\ell(K)$ lies on $L'$ and for all $K'\in g'(\GG)$, $\ell(K')$ lies on $L$. Since for all $K'\in g'(\GG)$, $\ell(K')$ lies on $L$, so $v(g'(\GG))\subseteq \overline{L}$. Therefore, since $p'\in I'$, so $L$ should be $I'$. If $\overline{I'}\subseteq I$, then $g'$ is the element that we want. Assume that $\overline{I'} \nsubseteq I$. Since $I'\subsetneq I$, so $\partial I'\nsubseteq I$, that is, there is an element $x$ in $\partial I'$ such that $x\notin I$. Since $\overline{I'}\subseteq \overline{I}$ and so $\partial I' \subseteq \overline{I}$. Therefore, $x\in \partial I$ since $x\notin I$. So, $x\in \partial I \cap \partial I' $. If $\partial I = \partial I'$, then $I'=I$ or $I'=I^*$ since $I$ and $I'$ are nondegenerate open intervals. However, it is a contradiction since $I'\subsetneq I$. Thus,  $\partial I' \cap \partial I =\{x\}$. 

In this case, if $x\notin v(g'(\GG))$, then we are done. Assume that $x\in v(g'(\GG))$. Note that $v(g'(\GG))\subseteq \overline{(L')^*}$. Since $(L')^*\subseteq I'$, $x\in \overline{(L')^*}\subseteq \overline{I'}$ and $x\in \partial I'$, so $x\in \partial L'$. There is a unique element $M$ in $g'(\GG)-\{ L' \}$ such that $x\in \partial M$. Therefore, for any $N$ in $g'(\GG)-\{ L', M \}$, $\overline{N}\subseteq I$. Choose $I''$  in $g'(\GG)-\{ L', M \}$. Finally,  since $v_G(\GG)$ is dense in $S^1$, there is $p''\in v_G(\GG)\cap I''$.
By the definition of $v_G(\GG)$, there is an element $g''$ in $G$ such that $p''\in v(g''(\GG))$. Like the previous argument, we can conclude that $v(g''(\GG))\subseteq \overline{I''}$. Therefore, $v(g''(\GG))\subseteq I$.
\end{proof}

\section{Not virtually abelian laminar groups}
In this section, we prove the following theorem which gives a condition which guarantees that a laminar group is not virtually abelian.
\begin{thm}
Let $G$ be a subgroup of $\Homeop(S^1)$ and $\LL$ be a $G$-invariant lamination system. Suppose that there is an ideal polygon $\GG$ on $\LL$ which is not a leaf. If $v_G(\GG)$ is dense in $S^1$, then $G$ is not virtually abelian.   
\end{thm}
The denseness of $v_G(\GG)$ allows some movements of  intervals by an element of $G$. So,  the strategy of the proof of the above theorem is analyzing the fixed point set of some element of $G$ and making  contradictory configurations of fixed points by using the denseness.  
Before proving the theorem, we define some notions about non-leaf ideal polygons.

\begin{defn}
Let $G$ be a subgroup of $\Homeop(S^1)$ and $\LL$ a $G$-invariant lamination system. For each $g\in G$ and an ideal polygon $\GG$ of $\LL$, we define \textsf{$g$-types} of $\GG$ as the followings : 
\begin{enumerate}
	\item The g-type of $\GG$ is \textsf{$g$-free} if $|v(\GG)\cap Fix_g|=0$
	\item The g-type of $\GG$ is \textsf{$g$-sticky} if $|v(\GG)\cap Fix_g|=1$
	\item The g-type of $\GG$ is \textsf{$g$-fixed} if $v(\GG)\subseteq Fix_g$
\end{enumerate}  
\end{defn}
In the following proposition, we can see that for each element $g$ of $G$ and any non-leaf gap $\GG$ of a lamination system $\LL$, $\GG$  is one of these $g$-types. 
\begin{prop}\label{fixed gap}
Let $G$ be a subgroup of $\Homeop(S^1)$ and $\LL$ a $G$-invariant lamination system. Suppose that there is an ideal polygon $\GG$ on $\LL$ which is not a leaf. For $g\in G$, if there are three distinct elements 
$I_1,I_2 \ and \  I_3$ in $\GG$ such that for all $i \in \ZZ_3$ ($i\in \ZZ_n$ means that the indices are modulo $n$), $\overline{I_i}$ contains a fixed point of $g$, then $\GG$ is $g$-fixed. 
\end{prop}
\begin{proof}
By Lemma \ref{tg}, $g(\GG)=\GG$ or there are $I$ in $\GG$ and $I'$ in $g(\GG)$ such that $I^*\subseteq I'$. First, we consider the later case. If $I'=g(I)$, then $I^* \subseteq g(I)$. Denote $I=(a,b)_{S^1}$. Then $(b,a)_{S^1}\subseteq (g(a),g(b))_{S^1}$ and so $a$ and $b$ are not fixed points of $g$. Moreover, for $z\in (b,a)_{S^1}$, $z\in (b,a)_{S^1} \subseteq (g(a),g(b))_{S^1}$ and so $g^{-1}(z)\in (a,b)_{S^1}$. Then since $(a,b)_{S^1}$ and $(b,a)_{S^1}$ are disjoint, $z\neq g^{-1}(z)$ and so $g(z)\neq z$. Therefore, $(b,a)_{S^1}\subseteq S^1-Fix_g$ and so $[b,a]_{S^1}\subseteq S^1-Fix_g$. However, it implies that there is the only one element in $\GG$ of which the closure contains a fixed point of $g$ and it is a contradiction by the assumption. 
\\
So, we assume that $I'\neq g(I)$. Choose $K$ in $\GG-\{I, g^{-1}(I')\}$. Denote $K=(x,y)_{S^1}$. Since $(I')^*\subseteq I$ and $g(K)$ and $I'$ are disjoint, $g(K)\subseteq (I')^*\subseteq I$ and since $I\subseteq K^*$, $g(K)\subseteq K^*$, that is, $(g(x),g(y))_{S^1} \subseteq (y,x)_{S^1}$. It implies that $x$ and $y$ are not fixed points of $g$. And for all $w\in (x,y)_{S^1}$, $g(w)\in (g(x),g(y))_{S^1}\subseteq (y,x)_{S^1}$ and since $(x,y)_{S^1}$ and $(y,x)_{S^1}$ are disjoint, $g(w)\neq w$. Therefore, $(x,y)_{S^1}\subseteq S^1-Fix_g$ and so $[x,y]_{S^1}\subseteq S^1-Fix_g$. However, there are exactly two elements $I$ and $g^{-1}(I')$ of which the closures can contains fixed points of $g$ and it is a contradiction by the assumption. Thus, $g(\GG)=\GG$  is the only possible case. 
\\
We denote $\GG=\{(x_1,x_2)_{S^1}, (x_2, x_3)_{S^1},\cdots, (x_{n-1},x_n)_{S^1}, (x_n, x_1)_{S^1}\}$ and  use $\ZZ_n$ as the index set. Note that since $g(\GG)=\GG$, there is $k\in \ZZ_n$ such that $(g(x_i),g(x_{i+1}))_{S^1}=(x_{i+k},x_{i+1+k})_{S^1}.$ If $k\neq 0$ on $\ZZ_n$, then there is no fixed point of $g$ since for all $i\in \ZZ_n$, $g(x_i)=x_{i+k}\neq x_i$ and $g((x_i,x_{i+1})_{S^1})=(x_{i+k},x_{i+1+k})_{S^1}\subseteq (x_i,x_{i+1})_{S^1}^* .$ It is a contradiction since $Fix_g\neq \phi$ by the assumption. Therefore, $k=0$ on $\ZZ_n$. Thus, for all $i\in \ZZ_n$, $g(x_i)=x_i$ and so $\GG$ is $g$-fixed. 
\end{proof}

\begin{cor}\label{cor : g-fixed}
Let $G$ be a subgroup of $\Homeop(S^1)$ and $\LL$ be a $G$-invariant lamination system. Let $g$ be a nontrivial element of $G$ with $Fix_g\neq \phi$. Suppose that there is an ideal polygon $\GG$ of $\LL$ with $2\leq |v(G)\cap Fix_g|$. Then, $\GG$ is $g$-fixed. 
\end{cor}

By Corollary \ref{cor : g-fixed}, we can see that a non-leaf gap is classified as the definition of $g$-types.
When we deal with the fixed points of two commute elements of $\Homeop(S^1)$, the following proposition will be frequently used. 

\begin{prop}\label{preserve fp}
Let $g$ and $h$ be two elements of $\Homeop(S^1)$ and $x$ an element of $S^1$. Suppose that $gh=hg$. Then, $x$ is a fixed point of $g$ if and only if $h(x)$ is  a fixed point of $g$. 
\end{prop}
\begin{proof}
Suppose that  $x$ is a fixed point of $g$. $g(h(x))=h(g(x))=h(x)$ and so $h(x)$ is a fixed point of $g$. Conversely, suppose that $h(x)$ is a fixed point of $g$. Then, $h(g(x))=g(h(x))=h(x)$ and since $h$ is a bijection, $g(x)=x$. And so $x$ is a fixed point of $g$.
\end{proof}
To start the proof of the main theorem, we should take a non-trivial element of $G$ which has a fixed points. The following lemma shows that there is a non-trivial element of $G$ under the condition of the main theorem. 
\begin{lem}\label{eone} 
Let $G$ be a subgroup of $\Homeop(S^1)$ and $\LL$ be a $G$-invariant lamination system. Suppose that there is an ideal polygon $\GG$ on $\LL$ such that $\GG$ is not a leaf and $v_G(\GG)$ is dense in $S^1$. Then, there is a nontrivial element $g$ of $G$ such that $Fix_g\neq \phi$.
\end{lem}
\begin{proof}
Assume that there is no non-trivial element of $G$  which has a fixed point. Choose $I\in \GG$. By Lemma \ref{mg}, there is an element $g$ in $G$ such that for any  $K\in g(\GG)$, $\ell(K)$ properly lies on $I$. By Lemma \ref{tg}, there is the element $I'$ in $g(\GG)$ such that $(I')^*\subseteq I$ and since $\ell(I')$ properly lies on $I$, $\overline{(I')^*}\subseteq I$. If $I'\neq g(I)$, then $g(I)\subseteq (I')^*$ and so $g(\overline{I})\subseteq \overline{(I')^*}\subseteq I\subseteq \overline{I}.$ It implies that there is a fixed point of $g$ in $I$, but it is a contradiction by the assumption. So, $I'=g(I)$ is the possible case. Then, choose $J$ in $g(\GG)$ such that $\overline{J} \subset I.$ By Lemma \ref{mg}, there is an element $h$ in $G$ such that for any $K\in hg(\GG)$, $\ell(K)$ properly lies on $J$. By Lemma \ref{tg}, there is the element $J'$ in $hg(\GG)$ such that $(J')^*\subseteq J$ and since $\ell(J')$ properly lies on $J$, $\overline{(J')^*}\subseteq J$. If $J'\neq h(J)$, then $h(J) \subseteq (J')^*$ and so  $h(\overline{J}) \subseteq \overline{(J')^*} \subseteq J\subseteq \overline{J}$. It implies that there is a fixed point of $h$ in $J$, but it is a contradiction by the assumption. So $J'=h(J)$. Since $g(I)=I'$, $g(I)\subseteq J^*$ and so $h(g(I))\subseteq h(J^*)=h(J)^*=(J')^*$. Then, 
$$h(g(\overline{I}))=\overline{h(g(I))}\subseteq \overline{(J')^*}\subseteq J \subseteq \overline{J} \subseteq I\subseteq \overline{I}.$$
It implies that the nontrivial element $hg$ has a fixed point in $I$ but it is a contradiction by the assumption. Thus, there is a nontrivial element of $G$ which has a fixed point. 
\end{proof}

First, before proving the virtual case, we show that $G$ is non-abelian under the condition of the main theorem. 

\begin{thm}\label{thm:nonabelian}
Let $G$ be a subgroup of $\Homeop(S^1)$ and $\LL$ be a $G$-invariant lamination system. Suppose that there is an ideal polygon $\GG$ on $\LL$ which is not a leaf. If $v_G(\GG)$ is dense in $S^1$, then $G$ is non-abelian.   
\end{thm}
\begin{proof}
Assume that $G$ is abelian. By Lemma \ref{eone}, there is a nontrivial element $g$ in $G$ with $Fix_g\neq \phi.$ First, if there are three distinct elements in $\GG$ such that the closure of each element contains a fixed point, then by Proposition \ref{fixed gap}, $\GG$ is $g$-fixed. Since for all $h\in G$, $h(v(\GG))\subseteq Fix_g$, so  by Proposition \ref{preserve fp}, $v_G(\GG)\subseteq Fix_g$. By the assumption, $v_G(\GG)$ is dense and so $Fix_g$ is dense in $S^1$. Since $Fix_g$ is closed on $S^1$, $Fix_g=S^1$ , but it implies that $g$ is the trivial element of $G$ and so it is a contradiction. 

If there are exactly two distinct elements $I$ and $J$ in $\GG$ such that $\overline{I}\cap Fix_g \neq \phi$ and $\overline{J}\cap Fix_g\neq \phi$, then there is an element $K$ in $\GG$ such that $\overline{K} \cap Fix_g=\phi$. By Lemma \ref{mg}, there is an element $h$ in $G$ such that for any $L\in h( \GG)$,  $\ell(L)$ properly lies on $K$. By Lemma \ref{tg}, there is the element $K'$ in $h(\GG)$ such that $(K')^*\subseteq K$ and since $\ell(K')$ properly lies on $K$, $\overline{(K')^*}\subseteq K$. Then at least one of $h(I)$ and $h(J)$ is not $K'$. Without loss of generality, we may assume that $h(I)\neq K'$. Then $h(I)\subseteq (K')^* \subseteq K$ and so $h(\overline{I}) \subseteq \overline{K}$. However, by Proposition \ref{preserve fp}, 
$h(\overline{I})\cap Fix_g \neq \phi$ since $\overline{I}\cap Fix_g \neq \phi$ , and so $\overline{K}\cap Fix_g\neq \phi$. It is a contradiction since $\overline{K}\cap Fix_g=\phi$.

Finally, if there is a unique element $M$ in $\GG$ such that $\overline{M}\cap Fix_g\neq \phi$, that is, $Fix_g\subseteq \overline{M}$, then there are two  distinct elements $O_1$ and $O_2$ in $\GG$ such that $\overline{O_1}\cap Fix_g=\phi$ and  $\overline{O_2}\cap Fix_g=\phi.$ For each $i\in \ZZ_2$, by Lemma \ref{mg}, there is an element $f_i$ in $G$ such that for any $P\in f_i(\GG)$, $\ell(P)$ properly lies on $O_i$. Fix $i\in \ZZ_2$.
By Lemma \ref{tg}, there is the element $O_i'$ in $f_i(\GG)$ such that $(O_i')^*\subseteq O_i$. 
If $O_i'\neq f_i(M)$, then $f_i(M)\subseteq (O_i')^*\subseteq O_i$ and so $f_i(\overline{M})=\overline{f_i(M)}\subseteq \overline{O_i}$. However, by Proposition \ref{preserve fp}, $f_i(\overline{M})\cap Fix_g\neq\phi$ since $\overline{M}\cap Fix_g \neq \phi$ and so it is a contradiction since $\overline{O_i}\cap Fix_g=\phi.$ Therefore, for all $i\in \ZZ_2$, $f_i(M)=O_i'$. Then, we can get the following relations:
\begin{enumerate}
	\item $f_1(O_1)\subseteq f_1(M^*)=f_1(M)^*=(O_1')^*\subseteq O_1$
	\item $f_1(O_2)\subseteq f_1(M^*)=f_1(M)^*=(O_1')^*\subseteq O_1$
	\item $f_2(O_1)\subseteq f_2(M^*)=f_2(M)^*=(O_2')^*\subseteq O_2$
	\item $f_2(O_2)\subseteq f_2(M^*)=f_2(M)^*=(O_2')^*\subseteq O_2$
\end{enumerate}
 Let us consider two elements $f_1f_2$ and $f_2f_1$.
 $$f_1f_2(O_1)\subseteq f_1(O_2)\subseteq O_1$$
 and 
 $$f_2f_1(O_1)\subseteq f_2(O_1)\subseteq O_2 .$$
 However, it implies $f_1f_2(O_1)\neq f_2f_1(O_1)$ since $O_1$ and $O_2$ are disjoint, and so it is a contradiction by the assumption that $G$ is abelian. Thus, $G$ is non-abelian.  
\end{proof}

To improve this theorem, we need to the following lemma. When we prove the virtual case, we will take a finite index subgroup $H$ of $G$ and construct new lamination system which is preserved by $H$. In this construction of the $H$-invariant lamination system, we will collapse the original circle on which the $G$-invariant lamination system is defined. The following lemma guarantees that there is a non-leaf gap of the $H$-invariant lamination system. 

\begin{lem}\label{ndg} 
Let $G$ be a subgroup of $\Homeop(S^1)$ and $\LL$ be a $G$-invariant lamination system in which there is a non-leaf ideal polygon $\GG_0$. Suppose that $v_G(\GG_0)$ is dense in $S^1$. If $H$ is a finite index subgroup of $G$, then there is a non-leaf ideal polygon $\GG$ in $\LL$ which is $g(\GG_0)$ for some $g\in G$ and has  three elements $I_1$, $I_2$ and $I_3$ such that for all $i\in \ZZ_3$,
$I_i\cap \overline{ v_H(\GG)}$ has nonempty interior. 
\end{lem}
\begin{proof}
Since the case $G=H$ is obvious, we assume that $H$ is a proper subgroup of $G$.
Assume that for each $g\in G$, there are at most two elements in $g(\GG_0)$ which contain interior points of $\overline{v_H(g(\GG_0))}$. Since $H$ has a finite index, 
we can denote $H\backslash G=\{Hg_1, Hg_2, \cdots, Hg_n\}$ for some $\{g_1, g_2, \cdots ,g_n\}\subseteq G.$ Then, $v_G(\GG_0)=\displaystyle \bigcup_{i=1}^n v_H(g_i(\GG_0)).$ So,
$$S^1=\overline{v_G(\GG_0)}=\overline{\displaystyle \bigcup_{i=1}^n v_H(g_i(\GG_0))}=\displaystyle \bigcup_{i=1}^{n} \overline{v_H(g_i(\GG_0))}$$
since $v_G(\GG_0)$ is dense in $S^1$. Since a finite union of nowhere dense sets is nowhere dense and $S^1$ is not nowhere dense, there is $\alpha_1 \in \{1,2, \cdots, n\}$ such that $\overline{v_H(g_{\alpha_1}(\GG_0))}$ has non-empty interior. Without loss of generality, we may assume $\alpha_1=1$. Since $\overline{v_H(g_1(\GG_0))} $ has non-empty interior, there is a nondegenerate interval $J_1$ on $S^1$ such that $J_1\subseteq \overline{v_H(g_1(\GG_0))}$. Denote $J_1=(u_1,v_1)_{S^1}$. Since $J_1 \cap v_H(g_1(\GG_0))$ is dense in $J_1$, there is a gap $\GG_1$ such that $\GG_1=h_1g_1(\GG_0)$ for some $h_1\in H$ and $J_1\cap v(\GG_1) \neq \phi.$ 
Choose $p_1\in J_1 \cap v(\GG_1).$ By the assumption, there are exactly two elements in $\GG_1$ which contain $(u_1,p_1)_{S^1}$ or $(p_1,v_1)_{S^1}.$ Then, we can choose an elements $K_1$ in $\GG_1$ such that $K_1\cap v_H(g_1(\GG_0))$ is  nowhere dense in $K_1$. Then, 

$$K_1=K_1\cap S^1= K_1\cap \overline{v_G(\GG_0)}=K_1 \cap \overline{\displaystyle \bigcup_{i=1}^n v_H(g_i(\GG_0))}=K_1 \cap \displaystyle \bigcup_{i=1}^n \overline{v_H(g_i(\GG_0))}=\displaystyle \bigcup_{i=1}^n K_1 \cap \overline{v_H(g_i(\GG_0))}.$$

Since a finite union of nowhere dense sets is nowhere dense and $K_1$ is not nowhere dense, there is $\alpha_2 \in \{2, \cdots, n\}$ such that $K_1\cap \overline{v_H(g_{\alpha_2}(\GG_0))}$ has non-empty interior. Without loss of generality, we may assume $\alpha_2=2$. Since $K_1 \cap \overline{v_H(g_2(\GG_0))} $ has non-empty interior, there is a nondegenerate interval $J_2$ on $K_1$ such that $J_2\subseteq K_1 \cap \overline{v_H(g_2(\GG_0))}$. Denote $J_2=(u_2,v_2)_{S^1}.$ Since $v_G(\GG_0)$ is dense in $S^1$ and so $E(\LL)$ is dense in $S^1$ , there is a point $q_1$ in $E(\LL)\cap J_2$. There is a leaf $\ell_1$ such that $q_1\in v(\ell_1)$. By Lemma \ref{tg}, there is $L_1$ in $\ell_1$ such that $L_1\subseteq K_1$. Then, one of $L_1 \cap (u_2,q_1)_{S^1}$ and $L_1\cap (q_1, v_2)_{S^1}$ is non-empty and so $J_2\cap L_1$ is non-empty. Likewise, $J_2\cap L_1^*$ is also non-empty. Since $J_2\cap v_H(g_2(\GG_0))$ is dense in $J_2$, there is a gap $\GG_2$ such that $\GG_2=h_2g_2(\GG_0)$ for some $h_2\in H$ and $ J_2\cap L_1 \cap v(\GG_2)\neq \phi$. 
By Lemma \ref{tg}, there is $M_1$ in $\GG_2$ such that $M_1^*\subseteq L_1.$ Since $J_2\cap L_1^*$ is non-empty, $J_2\cap L_1^*\subseteq L_1^* \subseteq M_1$ and so $M_1\cap \overline{v_H(\GG_2)}$ has non-empty interior. By the assumption, it implies that there is $K_2$ in $\GG_2$ such that  $K_2 \subseteq M_1^*$ and   $K_2\cap v_H(\GG_2)$ is nowhere dense in $K_2.$ Moreover, $K_2 \subseteq M_1^* \subseteq L_1 \subseteq K_1$. Therefore,  $K_2\cap v_H(\GG_1)$ and $K_2\cap v_H(\GG_2)$ are nowhere dense in $K_2$. If $n=2$, then 
\begin{align*}
K_2&=K_2\cap S^1\\
&= K_2\cap \overline{v_G(\GG_0)}\\
&=K_2 \cap \overline{\displaystyle \bigcup_{i=1}^2 v_H(g_i(\GG_0))}\\
&=K_2 \cap \displaystyle \bigcup_{i=1}^2 \overline{v_H(g_i(\GG_0))}\\
&=\displaystyle \bigcup_{i=1}^2 K_2 \cap \overline{v_H(g_i(\GG_0))}\\
&=[K_2 \cap \overline{v_H(g_1(\GG_0))} ] \cup [K_2 \cap \overline{v_H(g_2(\GG_0))} \big]\\
&=\big[K_2 \cap \overline{v_H(\GG_1)} ] \cup [K_2 \cap \overline{v_H(\GG_2)} ].
\end{align*}
However, it is a contradiction since a finite union of nowhere dense sets is nowhere dense. 

If $n$ is greater than $3$, choose  $m \in \{2,\cdots, n-1\} .$ Assume that for each $i \in \{1,2, \cdots ,  m\}$, there is a gap $\GG_i$ which is $h_i g_i(\GG_0)$ for some $h_i\in H$ and there is $K_m$ in $\GG_m$ such that 
for all $i \in \{1,2, \cdots, m\}$, $K_m \cap v_H(\GG_i)$ are nowhere dense in $K_m$.  Then,
\begin{align*}
K_m&=K_m\cap S^1\\
&= K_m\cap \overline{v_G(\GG_0)}\\
&=K_m \cap \overline{\displaystyle \bigcup_{i=1}^n v_H(g_i(\GG_0))}\\
&=K_m \cap \displaystyle \bigcup_{i=1}^n \overline{v_H(g_i(\GG_0))}\\
&=\displaystyle \bigcup_{i=1}^n K_m \cap \overline{v_H(g_i(\GG_0))}
\end{align*}

Since a finite union of nowhere dense sets is nowhere dense and $K_m$ is not nowhere dense, there is $\alpha_{m+1} \in \{m+1, \cdots, n\}$ such that $K_m\cap \overline{v_H(g_{\alpha_{m+1}}(\GG_0))}$ has non-empty interior. Without loss of generality, we may assume $\alpha_{m+1}=m+1$. Since $K_m \cap \overline{v_H(g_{m+1}(\GG_0))} $ has non-empty interior, there is a nondegenerate interval $J_{m+1}$ on $K_m$ such that $J_{m+1}\subseteq K_m \cap \overline{v_H(g_{m+1}(\GG_0))}$. Denote $J_{m+1}=(u_{m+1},v_{m+1})_{S^1}.$ Since $E(\LL)$ is dense in $S^1$ , there is a point $q_m$ in $E(\LL)\cap J_{m+1}$. There is a leaf $\ell_m$ such that $q_m\in v(\ell_m)$. By Lemma \ref{tg}, there is $L_m$ in $\ell_m$ such that $L_m\subseteq K_m$. Then, one of $L_m \cap (u_{m+1},q_m)_{S^1}$ and $L_m\cap (q_m, v_{m+1})_{S^1}$ is non-empty and so $J_{m+1}\cap L_m$ is non-empty. Likewise, $J_{m+1}\cap L_m^*$ is also non-empty. Since $J_{m+1}\cap v_H(g_{m+1}(\GG_0))$ is dense in $J_{m+1}$, there is a gap $\GG_{m+1}$ such that $\GG_{m+1}=h_{m+1}g_{m+1}(\GG_0)$ for some $h_{m+1}\in H$ and $ J_{m+1}\cap L_m \cap v(\GG_{m+1})\neq \phi$.  
By Lemma \ref{tg}, there is $M_m$ in $\GG_{m+1}$ such that $M_m^*\subseteq L_m.$ Since $J_{m+1}\cap L_m^*$ is non-empty, $J_{m+1}\cap L_m^*\subseteq L_m^* \subseteq M_m$ and so $M_m\cap \overline{v_H(\GG_{m+1})}$ has non-empty interior. By the assumption, it implies that there is $K_{m+1}$ in $\GG_{m+1}$ such that  $K_{m+1} \subseteq M_m^*$ and   $K_{m+1}\cap v_H(\GG_{m+1})$ is nowhere dense in $K_{m+1}.$ Moreover, $K_{m+1} \subseteq M_m^* \subseteq L_m \subseteq K_m$. Therefore,  for all $i\in \{1,2, \cdots, m+1\}$, $K_{m+1}\cap v_H(\GG_{i})$ are nowhere dense in $K_{m+1}$. 

Finally, for each $i \in \{1,2, \cdots ,  n\}$, there is a gap $\GG_i$ which is $h_i g_i(\GG_0)$ for some $h_i\in H$ and there is $K_n$ in $\GG_n$ such that 
for all $i \in \{1,2, \cdots, n\}$, $K_n\cap v_H(\GG_i)$ are nowhere dense in $K_n$. However, 
\begin{align*}
K_n&=K_n\cap S^1\\
&= K_n\cap \overline{v_G(\GG_0)}\\
&=K_n \cap \overline{\displaystyle \bigcup_{i=1}^n v_H(g_i(\GG_0))}\\
&=K_n \cap \displaystyle \bigcup_{i=1}^n \overline{v_H(g_i(\GG_0))}\\
&=\displaystyle \bigcup_{i=1}^n K_n \cap \overline{v_H(g_i(\GG_0))}\\
&=\displaystyle \bigcup_{i=1}^n K_n \cap \overline{v_H(\GG_i)}
\end{align*}
 and so it is a contradiction since a finite union of nowhere dense sets is nowhere dense. We are done.

\end{proof} 
Let us prove the main theorem. 
\begin{thm}\label{thm:virtuallynonabelian}
Let $G$ be a subgroup of $\Homeop(S^1)$ and $\LL$ be a $G$-invariant lamination system. Suppose that there is an ideal polygon $\GG_0$ on $\LL$ which is not a leaf. If $v_G(\GG_0)$ is dense in $S^1$, then $G$ is not virtually abelian.    
\end{thm}
\begin{proof}
Suppose that $H$ is a finite index subgroup of $G$. By Lemma \ref{ndg}, there is a gap $\GG$ which is $g(\GG_0)$ for some $g\in G$ and has three elements $I_1,I_2$ and $I_3$ such that for all $i\in \ZZ_3$, $I_i\cap \overline{v_H(\GG)}$ has non empty interior on $S^1$. Since $\overline{v_H(\GG)}$ has non-empty interior, we can define $m:S^1\rightarrow S^1$ as the monotone map which collapses each closure of connected component of $S^1-\overline{v_H(\GG)}.$ Then,  for each element $h\in H$, there is a unique element $g_h$ in $\Homeop(S^1)$ which makes the following diagram commute :  
\[
\begin{tikzcd}
 S^1 \arrow[d,"h"] \arrow[r, "m"] & S^1 \arrow[d, "g_h"]\\
 S^1 \arrow[r,"m"] & S^1 
\end{tikzcd}
\]
since $\overline{v_H(S^1)}$ is preserved by the action of $H$. Define $G_H\equiv \{ g_h \in \Homeop(S^1): h\in H  \}$. Then, $G_H$ is isomorphic to some quotient group of $H.$ 
Let us define $\LL_H$ as the family of nondegenerate open intervals $(u,v)_{S^1}$ such that there is $I$ in $\LL$ such that $m(v(\ell(I)))=\{u,v\}.$ By the construction of $\LL_H$, $\LL_H$ is a $G_H$-invariant lamination system. Moreover, since $\GG$ has three elements $I_1,I_2$ and $I_3$ such that for all $i\in \ZZ_3$, $I_i\cap \overline{v_H(\GG)}$ has non-empty interior, there is a non-leaf ideal polygon $\GG_H$ in $\LL_H$ such that $m(v(\GG))=v(\GG_H)$. By the construction of $G_H$ and $\LL_H$, $v_{G_H}(\GG_H)$ is dense in $S^1$. Therefore, by Theorem \ref{thm:nonabelian}, $G_H$ is non-abelian and so $H$ is also non-abelian. Thus, $G$ is not virtually abelian. 
	  
\end{proof}

\section{Existence of a non-abelian free subgroup in the tight pairs}
In 2001, Calegari wrote a lecture note entitled 'Foliations and the
geometrization of 3-manifolds' \cite{calegari2001foliations}, and
later a large chunk of this note became the book \cite{Calebook}. In
this note, Calegari introduced the notion of a tight pair to study
special types of laminar groups. We rephrase the definition below in
terms of lamination systems. 

\begin{defn}
Let $G$ be a subgroup of $\Homeop(S^1)$, and $\LL$ be a $G$-invariant lamination system.  The pair $(\LL,G)$ is \textsf{tight} if $\LL$ is very full and totally disconnected and for each $I\in \LL$, $\ell(I)$ is not isolated, $G$ acts on $\LL$ minimally and the set of non-leaf gaps consists of finitely many orbit classes under this action.
\end{defn}

In  \cite{calegari2001foliations}, Calegari showed that there are two
types of tight pairs, sticky pairs and slippery pairs. $(\LL, G)$ is a
sticky pair if every gap of $\LL$ has a vertex shared with another
non-leaf gap, and is a slippery pair if no non-leaf gap of $\LL$ shares a vertex with
other gaps. He constructed a dual $\mathbb{R}$-tree to $\LL$ in case of the sticky
pairs and by analyzing the $G$-action on this dual tree, the following
theorem was obtained.

\begin{thm}[Calegari  \cite{calegari2001foliations}]
  Suppose $(\LL, \pi_1(M))$ is a sticky pair for some closed
  irreducible 3-manifold $M$. Then $M$ is Haken. 
\end{thm} 

Full detail of the proof of above theorem is also presented in Master's thesis of Te
Winkel \cite{Winkel16}. In this section, we study a general feature of
tight pairs.

\begin{prop}\label{dol}
Let $(\LL, G)$ be a tight pair. Then for any leaf $\ell$, $v_G(\ell)=\displaystyle \bigcup_{g\in G} v(g(\ell))$ is dense on $S^1$. 
\end{prop}
\begin{proof}
Suppose that there is a leaf $\ell$ of $\LL$ such that $v_G(\ell)$ is not dense in $S^1$. Then there is a connected component $K$ of $S^1-\overline{v_G(\ell)}$. Since, by Corollary \ref{ed}, $E(\LL)$ is dense on $S^1$, there is $p$ in $E(\LL)\cap K$ and so there is a leaf $\ell'$ of $\LL$ with $p\in v(\ell')$. Since the action of $G$ is minimal, there is a sequence $\{g_n\}_{n=1}^{\infty}$ of $G$ such that $g_n(\ell)\rightarrow \ell'$. Then there is a sequence $\{J_n\}_{n=1}^{\infty}$ of $\LL$ such that for all $n\in \NN$, $g_n(\ell)=\ell(J_n)$ and 
$$I'\subseteq \liminf J_n \subseteq \limsup J_n \subseteq \overline{I'}$$
for some $I'\in \ell'$.
For each $n\in \NN$, it is either $K\subseteq J_n$ or $K\subseteq J_n^*$ by the choice of $K$. Note that $I'\cap K$ is not empty. Choose $q\in I'\cap K$. Since $I'\subseteq \liminf J_n$, there is $N$ in $\NN$ such that 
$q\in \displaystyle \bigcap_{n=N}^{\infty}J_n$. Therefore, for any $n\geq N$, $q\in K\subseteq J_n$ and so $K\subseteq \displaystyle \bigcap_{n=N}^{\infty}J_n\subseteq \liminf J_n$. However, $K$ is not contained in $\overline{I'}$ and so it is a contradiction. 
\end{proof}

\begin{cor}\label{dog} 
Let $(\LL, G)$ be a tight pair. Then for any non leaf gap $\GG$, $v_G(\GG)$ is dense on $S^1$. 
\end{cor}

\begin{prop}\label{eg}
Let $(\LL, G)$ be a tight pair.  There is a non leaf gap $\GG$. 
\end{prop}
\begin{proof}
Since $\LL$ is not empty, there is an element $I \in \LL$. By the definition of lamination system, $I^*\in \LL$. By Corollary \ref{ed}, there is $p$ in $E(\LL) \cap I^*$. So there is $J\in \LL$ such that $p\in v(\ell(J))$. If $J\subseteq I$, then $p\in \overline{J} \subseteq \overline{I}$ and it is a contradiction since $p\in I^*$. If $J^* \subseteq I$, then $p\in \overline{J^*} \subseteq \overline{I}$ and it is also a contradiction since $p\in I^*$. Therefore, either $I \subsetneq J$ or $I\subsetneq J^*$. So, $\{I,J^*\}$ or $\{I, J\}$ is a distinct pair, respectively. Thus, since $\LL$ is totally disconnected, there is a non leaf gap which makes the distinct pair  be separated. 
\end{proof}

By Theorem \ref{thm:virtuallynonabelian}, tight pairs are
not virtually abelian. Our goal here is to show that a tight pair
actually contains a non-abelian free
subgroup as long as it does not admits a global fixed point. We will use the following famous theorem of Margulis which is 
an analogy of the Tits alternative. 

\begin{thm}[Margulis \cite{margulis2000free}]\label{margulis}
Let $G$ be a subgroup of $\Homeop(S^1)$. At least one of the following properties holds : 
\begin{enumerate}
\item $G$ contains a non abelian free subgroup.
\item There is a Borel probability measure on the circle which is $G$-invariant. 
\end{enumerate}
\end{thm}

Let $\mu$ be a Borel probability measure on $S^1$. We define the \textsf{support} of $\mu$ as the complement of the union of measure zero open sets and denote it as $supp(\mu)$.  
we can get the following facts : 
\begin{enumerate}
\item $supp(\mu)$ is a closed subset of $S^1$.
\item For each $p\in supp(\mu)$ and each open neighborhood $U$ of $p$, $\mu(U)>0$.
\item If $\mu$ is also $G$-invariant where $G$ is a subgroup of $\Homeop(S^1)$, then $supp(\mu)$ is also $G$-invariant, that is,  for each $g\in G$, $g(supp(\mu))=supp(\mu).$ 
\end{enumerate}

\begin{lem}\label{mig}
Let $(\LL, G)$ be a tight pair. Suppose that there is a Borel probability measure $\mu$ on $S^1$ which is $G$-invariant. Then for each non leaf gap $\GG$ of $\LL$, there is a unique element $I$ in $\GG$ such that $\mu(I)=1$.
\end{lem}

\begin{proof}
Let $\GG$ be a non leaf gap. First, we want to show that there is at most two positive measure elements in $\GG$.  Suppose that  there are three elements $I_0, I_1 \ and \ I_2$  in $\GG$ which are positive measure. Say that $\{I_i\}_{i\in \ZZ_3}$ and choose $i\in \ZZ_3$. By Lemma \ref{mg}, there is $g_i\in G$ such that for any $J\in g_i(\GG)$, $\ell(J)$ properly lies on $I_i$ and  by Lemma \ref{tg}, there are $L_i$ in $\GG$ and $L_i'$ in $g_i(\GG)$ such that $ (L_i')^*\subseteq L_i $. Since for all $J\in g_i(\GG)$ which is not $L_i'$, $J\subseteq L_i$ and so $\ell(J)$ lies on $L_i$, so $L_i=I_i$. If $g_i(I_i)\neq L_i'$, then $g_i(I_i)\subseteq (L_i')^*\subseteq   I_i $. Then, at least one of $g_i(I_{i+1})$ and $g_i(I_{i+2})$ is contained in $I_i$.  If, for some $j \in \ZZ_3-\{i\}$, $g(I_i)\cup g(I_j) \subseteq I_i$, then  
$\mu(g(I_i))+\mu(g(I_{j}))\leq \mu(I_i)$ and since $\mu$ is $G$-invariant, $\mu(I_i)+\mu(I_{j})\leq \mu(I_i)$ , and so $\mu(I_j)\leq 0$. It is a contradiction since $0<\mu(I_j)$. Therefore, $g_i(I_i)=L_i'$. Then for all $i\in \ZZ_3$, $g_i(I_{i+1})\cup g_i(I_{i+2}) \subseteq g_i(I_i)^*=(L'_i)^* \subseteq L_i =I_i$ and so 
$\mu(I_{i+1})+\mu(I_{i+2})=\mu(g_i(I_{i+1}))+\mu(g_i(I_{i+2})) \leq \mu(I_i)$. However, 
$$\mu(I_1)\geq \mu(I_2)+\mu(I_3)\geq \{ \mu(I_3)+\mu(I_1)\} +\{ \mu(I_1)+\mu(I_2) \}$$ 
and so
$$0\geq \mu(I_1)+\mu(I_2)+\mu(I_3).$$
It is a contradiction since $ \mu(I_1)+\mu(I_2)+\mu(I_3)>0.$ Therefore, there are at most two positive measure elements in $\GG$. It implies that there is  at least one measure-zero element $J$ in $\GG$ since $\GG$ is a non leaf gap. 
Now, we show that there is a element $I$ such that $\mu(I)=1$. By Lemma \ref{mg}, there is $g\in G$ such that for any $K\in g(\GG)$, $\ell(K)$ properly lies on $J$ and  by Lemma \ref{tg}, there are $L$ in $\GG$ and $L'$ in $g(\GG)$ such that $ (L')^*\subseteq L $. Then, $L=J$ and $\overline{(L')^*} \subseteq L=J$. So 
$\mu((L')^c)=\mu(\overline{(L')^*}) \leq \mu(J)=0$ and it implies that $\mu(L')=1$. Thus, $L'$ is the element $I$ which we want.
\end{proof}

Then, we can get the same result in a leaf as the following lemma.

\begin{lem}\label{mil}
Let $(\LL, G)$ be a tight pair. Suppose that there is a Borel probability measure $\mu$ on $S^1$ which is $G$-invariant. Then for each leaf $\ell$, there is a unique element $I$ in $\ell$ such that $\mu(I)=1$. 
\end{lem}
\begin{proof}
Let $\ell$ be a leaf of $\LL$. By Proposition \ref{eg}, there is a non leaf gap $\GG$ of $\LL$. By the definition of gaps,  $\ell$ lies on an element $J$ of $\GG$. Say that $\ell=\ell(I)$ and $I\subseteq J$. Choose $K$ in $\GG$ which is not $J$. By Corollary \ref{dog} and Lemma \ref{mg}, there is $g$ in $G$ such that $v(g(\GG))\subseteq K$. There is $K'$ in $g(\GG)$ such that $\overline{K'}\subseteq K$. By Proposition \ref{dol}, there is $g'$ such that $v(g'(\ell))\cap K'\neq \phi$. Choose $p$ in $v(g'(\ell))\cap K'$. If $K'\subseteq g'(I)$, then $p\in \overline{g'(I^*)}\subseteq \overline{(K')^*}=(K')^c$ and it is a contradiction since $p\in K'$. If $K'\subseteq g'(I^*)$, then $p
\in \overline{g'(I)}\subseteq \overline{(K')^*}$ and it is also a contradiction since $p\in K'$. Therefore, it is either 
$ g'(I)\subsetneq K'$ or $ g'(I^*)\subsetneq K'$. So,   it is either $ \overline{g'(I)}\subseteq K$ or $ \overline{g'(I^*)}\subseteq K$. 

First, if $I$ is positive measure, then $J$ is also positive measure and, by Lemma \ref{mig}, $\mu(J)=1$. Moreover, $K$ is measure zero. So  $ \overline{g'(I^*)}\subseteq K$ is the case and 
 $\mu(\overline{I^*})=\mu(g'(\overline{I^*}))=\mu(\overline{g'(I^*)})\leq \mu(K)=0$. Therefore, $\mu(\overline{I^*})=0$ and so $\mu(I)=1$. 
 
 Next, assume that $\mu(I)=0$. If $\mu(J)=0$, then by Lemma \ref{mig} $\mu(\overline{J})=0$, and so $\mu(\overline{I}) \leq \mu(\overline{J})=0$. Therefore, $\mu(\overline{I})=0$ and so $\mu(I^*)=1$. If $\mu(J)=1$, then $\mu(K)=0$ by Lemma \ref{mig}. 
 Since $1=\mu(\overline{I^*})=\mu(g'(\overline{I^*}))=\mu(\overline{g'(I^*)})$, $ \overline{g'(I^*)}\subseteq K$ is not possible and so $\overline{g'(I)}\subseteq K$ is the possible case. Therefore, $\mu(\overline{g'(I)})\leq \mu(K)=0$, and so $\mu(\overline{I})=0$. Thus, $\mu(I^*)=1$.
\end{proof}

Finally, we prove the main theorem. 

\begin{thm}\label{im}
Let $(\LL,G)$ be a tight pair. Suppose that there is a Borel probability measure $\mu$ on $S^1$ which is $G$-invariant. Then, the support $supp(\mu)$  of the measure $\mu$ is a one point set. 
\end{thm}
\begin{proof}
Let $p$ be a point in $supp(\mu)$. First, if $p\in E(\LL)$, then there is a leaf $\ell$ with $p\in v(\ell)$. By Lemma \ref{mil}, there is a unique element $I$ in $\ell$ such that $\mu(I)=1$. So, $supp(\mu)\cap I^* =\phi$ by the definition of the support. 

By Proposition \ref{eg}, there is a non leaf gap $\GG$ and by Corollary \ref{dog}, $v_G(\GG)$ is dense in $S^1$. So, there is $g$ in $G$ such that $v(g(\GG))\cap I^*\neq \phi$. Moreover, by Lemma \ref{tg}
, there is $J$ in $g(\GG)$ such that $J^* \subseteq I^*$.  Since $I^*$ is measure zero, $\mu(J)=1$ by Lemma \ref{mil}. And since $g(\GG)$ is a non-leaf gap, there is $K$ in $g(\GG)$ such that $K\subseteq J^*$ and $\mu(K)=0$. Then by Corollary \ref{dog} and Lemma \ref{mg}, there is $h$ in $G$ such that $v(h(g(\GG)))\subseteq K$. Therefore, we can choose $L$ in $h(g(\GG))$ such that $\overline{L}\subset K$ and so $\overline{L}\subset I^*$.

 By Proposition \ref{dol}, $v_G(\ell)$ is dense so there is $k$ in $G$ such that $v(k(\ell))\cap L \neq \phi$. Then by Lemma \ref{tg}, $M\subseteq L$ for some $M\in k(\ell)$ and it implies $k(p)\in v(k(\ell))\subset \overline{L} \subset I^*$. However, since $k(p) \in supp (\mu)$, $0<\mu(I^*)$ and it is a contradiction. Thus $p\notin E(\LL).$

So, by Lemma \ref{er}, there is a rainbow $\{I_n\}_{n=1}^{\infty}$ at $p$. Applying Lemma \ref{mil} to each $\ell(I_n)$, since $p\in I_n$, $\mu(I_n)=1$ for all $n\in \NN$. Therefore,  $\mu(\{p\})=\mu(\displaystyle \bigcap _{n=1}^{\infty}I_n)=\lim_{n\rightarrow \infty} \mu(I_n)=1$ since $\mu$ is a finite measure. Thus, $supp(\mu)=\{p\}$.
\end{proof}

The following is an immediate corollary of the above theorem, since if
there are more than one global fixed point, one can find an invariant
probability measure supported on those points. 
\begin{cor}
Let $(\LL,G)$ be a tight pair. There is at most one global fixed point. 
\end{cor}

Now we state the main result of this section. 
\begin{cor}\label{cor: a free group in a tight pair}
Let $(\LL,G)$ be a tight pair without global fixed points. Then, $G$ contains a non abelian free subgroup. 
\end{cor}
\begin{proof}
Suppose that there is a $G$-invariant Borel probability measure $\mu$. Then $supp(\mu)$ is a one point set by Lemma \ref{im}. Since $supp(\mu)$ is $G$-invariant, so the element of $supp(\mu)$ is a global fixed point. By the assumption, it is a contradiction. Therefore, there is no such a measure. By Theorem \ref{margulis}, $G$ contains a non abelian free subgroup. 
\end{proof}

\section{Loose Laminations}

A very full lamination system is \textsf{loose} if for any two non-leaf  gaps
$\GG$ and $\GG'$ with $\GG \neq \GG' $, $v(\GG)\cap v(\GG')=\phi$. There are equivalent conditions in totally disconnected very full lamination systems. 
\begin{lem}[\cite{alonso2019laminar}] \label{loose}
Let $\LL$ be a totally disconnected very full lamination system. Then $\LL$ is loose if and only if the following conditions are satisfied :
\begin{enumerate}
\item for each $p\in S^1$, at most finitely many leaves of $\LL$ have $p$ as an endpoint. 
\item There are no isolated leaves.
\end{enumerate}
\end{lem}

A group acting on the circle with two loose invariant laminations
with certain conditions is called a pseudo-fibered triple. It was
observed in the first author's PhD thesis
\cite{baikthesis} that each nontrivial element in
the pseudo-fibered triple has at most finitely many fixed
points under the assumption that the fixedpoint set is
countable, hence countability of the fixedpoint sets is an underlying
assumption in \cite{alonso2019laminar}. This section should serve as
an appendix to \cite{alonso2019laminar} in which we prove that
additional assumption that the fixedpoint sets are countable is not
necessary.


In this section, we consider a pseudo-fibered triple which is a triple $(\LL_1,\LL_2, G)$ in which $G$ is a finitely generated subgroup of $\Homeop(S^1)$, each nontrivial element of $G$ has at most countably many fixed points in $S^1$ and $\LL_i$ are $G$-invariant very full  loose  lamination systems  with $E(\LL_1)\cap E(\LL_2)=\phi$. Indeed, without the fixed point condition of $G$, we can induce the original definition, that is, each nontrivial element of $G$ has finitely many fixed points. Let's begin with a weaker version of the definition of a pseudo-fibered triple.

\begin{defn}
Let $G$ be a finitely generated subgroup of $\Homeop(S^1)$, and $\LL_1$ and $\LL_2$ be two $G$-invariant lamination system. Then a triple $(\LL_1, \LL_2, G)$ is \textsf{pseudo-fibered} if $\LL_1$ and $\LL_2$ are  very full loose lamination systems with $E(\LL_1)\cap E(\LL_2)=\phi$.   
\end{defn}
The disjoint endpoints condition of two lamination systems enforces totally disconnectedness on lamination systems. 

\begin{prop}[\cite{alonso2019laminar}]\label{tot}
Let $(\LL_1, \LL_2, G)$ be a pseudo-fibered triple. Then $\LL_1$ and $\LL_2$ are totally disconnected. 
\end{prop}
\begin{proof}
It comes from Corollary $\ref{ed}$ and Lemma $\ref{tot dis}$.
\end{proof}
The following proposition says that  there is no sticky leaf on two lamination systems. 
\begin{prop}\label{same type endpoints}
Let $(\LL_1, \LL_2, G)$ be a pseudo-fibered triple, and $\GG$ be a leaf in $\LL_1$. For each $g\in G$, it is either  $v(\GG)\subseteq Fix_g$ or $v(\GG)\subseteq S^1-Fix_g$.
\end{prop}
\begin{proof}
Fix $g$ in $G$. If $Fix_g=\phi$ or $Fix_g=S^1$, then it is obvious. Assume that $Fix_g\neq \phi$ and $Fix_g \neq S^1$. Denote $\GG=\ell((u,v)_{S^1})$. If $u\in Fix_g$ and $v\in S^1-Fix_g$, then for each $n\in \ZZ$, $g^n(\ell((u,v)_{S^1}))=\ell((g^n(u),g^n(v))_{S^1})=\ell((u,g^n(v) )_{S^1})$. Since $g(v)\neq v$, it is either $(u, v)_{S^1}\subsetneq g((u,v)_{S^1})=(u,g(v))_{S^1}$ or $(u, g(v))_{S^1}=g((u,v)_{S^1})\subsetneq (u,v)_{S^1}$. Therefore, there are infinitely many leaves $\{g^{n}(\GG)| \ n\in \ZZ \}$ in which $u$ is an endpoint. However, by Proposition \ref{tot}, $\LL_1$ and $\LL_2$ are totally disconnected and so we can apply Lemma \ref{loose} to $\LL_1$. It is a contradiction. If $v\in Fix_g$ and $u\in S^1-Fix_g$, we can make a same argument with $\GG=\ell((v,u)_{S^1})$. Thus, we are done.
\end{proof}
With this proposition, we analyze the complement of the fixed points set of a non-trivial element of $G$. First, we recall the following lemma.

\begin{lem}[\cite{BaikFuchsian}]\label{3fix}
Let $g$ be a non-trivial orientation preserving homeomorphism on $S^1$ with $3\leq | Fix_g|$. Then any very full lamination system $\LL$ which is $\langle g \rangle$-invariant has a leaf $\ell$ in $\LL$ such that $v(\ell)\subseteq Fix_g$. Moreover, for any connected component $I$  of $S^1-Fix_g$ with $I=(a,b)_{S^1}$, at least one of $a$ and $b$ is an endpoint of a leaf of $\LL$. 
\end{lem}

We can interpret this lemma in a pseudo-fibered triple as the following proposition.

\begin{prop}\label{end point}
Let $(\LL_1, \LL_2, G)$ be a pseudo-fibered triple and $g$ a nontrivial element of $G$ with $3\leq |Fix_g|$. For any connected component $(u,v)_{S^1}$ of $S^1-Fix_g$, $u\in E(\LL_i)$ and $v\in E(\LL_j)$ with $i\neq j \in \{1,2\}$.
\end{prop}
\begin{proof}
By Lemma \ref{3fix} and the condition $E(\LL_1)\cap E(\LL_2)=\phi$, it is obvious.  
\end{proof}

\begin{prop}\label{disjoint intervals}
Let $(\LL_1, \LL_2, G)$ be a pseudo-fibered triple and $g$ a nontrivial element of $G$. Suppose that there are two distinct connected components $I_1$ and $I_2$ of $S^1-Fix_g$ such that $\overline{I_1}\cap \overline{I_2}=\phi$. Then, for each $i\in \{1,2\}$, there is no leaf $\ell$ of $\LL_i$ such that $|v(\ell)\cap I_1 |=|v(\ell)\cap I_2 |=1$.  
\end{prop}
\begin{proof}
Denote $I_1=(u_1,v_1)_{S^1}$ and $I_2=(u_2,v_2)_{S^1}$. Since $\overline{I_1}\cap \overline{I_2}=\phi$, $|\{u_1,v_1,u_2,v_2\}|=4$ and since $\{u_1,v_1,u_2,v_2\}\subseteq Fix_g$, $4 \leq |Fix_g|$. Then we can apply Proposition \ref{end point} to $(u_1,v_1)_{S^1}$ and so $u_1\in E(\LL_i)$ and $v_1\in E(\LL_j)$ with $i\neq j\in \{1,2\}$. Assume that there is an element $I$ in $\LL_1 \cup \LL_2$ such that $[v_1,u_2]_{S^1}\subseteq I\subset \overline{I}\subset (u_1,v_2)_{S^1}$. Since $I\in C_{v_1}^{(u_1,v_2)_{S^1}}$ and $C_{v_1}^{(u_1,v_2)_{S^1}}$ is preserved by $g$ and linearly ordered by the set inclusion by Proposition \ref{a}, so it is either $I\subseteq g(I)$ or $g(I)\subseteq I$. Since $\partial I \subseteq S^1-Fix_g$, it is either $\overline{I}\subseteq g(I)$ or $g(\overline{I})\subseteq I$. So one of two sequences $\{g^n(\ell(I))\}_{n=1}^{\infty}$ and $\{g^{-n}(\ell(I))\}_{n=1}^{\infty}$ converges to $(v_1,u_2)_{S^1}$ and the other converges to $(u_1,v_2)_{S^1}$. It implies that $\{u_1, v_1\}\subseteq E(\LL_i)$ for some $i\in \{1,2\}$. But it is a contradiction since $E(\LL_1)\cap E(\LL_2)=\phi$. Therefore there is no such $I$. So we are done.
\end{proof}

\begin{prop}\label{epi}
Let $(\LL_1, \LL_2, G)$ be a pseudo-fibered triple and $g$ a nontrivial element of $G$ with $3\leq |Fix_g|$. For any connected component $I$ of $S^1-Fix_g$, each point of $\partial I$ is isolated in $Fix_g$. 
\end{prop}
\begin{proof}
Denote $I=(u,v)_{S^1}$. Then by Proposition \ref{end point}, $u\in E(\LL_i)$ and $v\in E(\LL_j)$ with $i\neq j \in \{1,2\}$. Let's say that $u\in E(\LL_1)$ and $v\in E(\LL_2)$. Since $E(\LL_1)\cap E(\LL_2)=\phi$, $u\notin E(\LL_2)$ and so by Lemma \ref{er}, $u$ has a rainbow $\{I_n\}_{n=1}^{\infty}$ in $\LL_2$. Since $3\leq |Fix_g|$, there is a fixed point $w$ in $Fix_g-\{u,v\}$. Note that $w\in (v,u)_{S^1}$. Since $\displaystyle \bigcap_{n=1}^{\infty}I_n =\{u\}$, there is $I_N$ in $\{I_n\}_{n=1}^{\infty}$ such that $\{v,w\}\subseteq I_N^*$. 

So, we can assume that $u\in I_n$ and $\{v,w\} \subseteq I_n^*$ for all $n\in \NN$ and denote $I_n=(u_n,v_n)_{S^1}$. Then,  we want to show that for all $n\in \NN$, $v_n\in (u,v)_{S^1}$ and $u_n\in (w,u)_{S^1}$. Since $\phi(u_n,u,v_n)=1$ and $\phi(v_n,v,u_n)=\phi(u_n,v_n,v)=1$,  by the cocycle condition on the four points $(u_n,u,v_n,v)$, 
$$\phi(u,v_n,v)-\phi(u_n,v_n,v)+\phi(u_n,u,v)-\phi(u_n, u,v_n)=\phi(u,v_n,v)-1+\phi(u_n,u,v)-1=0$$ and so the only possible case is $\phi(u,v_n,v)=\phi(u_n,u,v)=1$. On the other hand, since $\phi(u_n, u,v_n)=1$ and $\phi(v_n,w,u_n)=\phi(u_n,v_n,w)=1$, by the cocycle conditions on the four points $(u_n,u,v_n,w)$,
$$\phi(u,v_n,w)-\phi(u_n,v_n,w)+\phi(u_n,u,w)-\phi(u_n,u,v_n)=\phi(u,v_n,w)-1+\phi(u_n,u,w)-1=0$$
and so the only possible case is $\phi(u,v_n,w)=\phi(u_n,u,w)=1$. Therefore, for all $n\in \NN$, $v_n\in (u,v)_{S^1}$ and $u_n\in (w,u)_{S^1}$ since $\phi(u,v_n,v)=1$ and $\phi(w,u_n,u)=\phi(u_n,u,w)=1.$

Fix $n$ in $\NN$. Since $v_n\in(u,v)_{S^1} \subseteq S^1-Fix_g$ , by Proposition \ref{same type endpoints}, $u_n\in S^1-Fix_g$ and so there is a unique connected component $J$ of $S^1-Fix_g$ which contains $u_n$. Since $u_n\in (w,u)_{S^1}$, so $J\subseteq (w,u)_{S^1}$ and by Proposition \ref{disjoint intervals}, $\overline{J}\cap [u,v]_{S^1}\neq \phi$. Since $v\neq w$, $\overline{J}\cap [u,v]_{S^1}\subseteq [w,u]_{S^1}\cap [u,v]_{S^1}=\{u\}$ and so $\overline{J}\cap [u,v]_{S^1}=\{u\}$. Therefore, $u$ is isolated in $Fix_g$. Likewise, $v$ is isolated in $Fix_g$.
\end{proof}

From now on, we prove lemmas which will be used in the proof of the main theorem.

\begin{lem} \label{sl}
Let $(\LL_1, \LL_2, G)$ be a pseudo-fibered triple and $g$ a nontrivial element of $G$ with $4\leq |Fix_g|$. Suppose that there is an isolated fixed point $p$ of $g$. Then there is an element $I$ in $\LL_1 \cup \LL_2$ such that $p\in v(\ell(I))$ and $|I\cap Fix_g|=1$.
\end{lem}
\begin{proof}
Since $p$ is an isolated fixed point and $4\leq |Fix_g|$, there is the connected component $(p,q)_{S^1}$ of $S^1-Fix_g$ which is a nondegenerate open interval. By Proposition \ref{epi}, $q$ is also an isolated fixed point. So there is the connected component $(q,r)_{S^1}$ of $S^1-Fix_g$. Since $4\leq |Fix_g|$, $r\neq p$. By Proposition \ref{end point}, $\{p, r\} \subseteq E(\LL_i)$ and $q\in E(\LL_j)$ with $i\neq j \in \{1,2\}$. Say that  $\{p, r\} \subseteq E(\LL_1)$ and $q\in E(\LL_2)$. Since $E(\LL_1)$ and $E(\LL_2)$ are disjoint, there is a rainbow $\{I_n\}_{n=1}^{\infty}$ at $q$ in $\LL_1$ by Theorem \ref{er}. Since $\displaystyle \bigcap_{n=1}^{\infty} I_n=\{q\}$, there is $I_N$ in $\{I_n\}_{n=1}^{\infty}$ such that $\{p,r\}\subseteq I_N^*$. We can get that $q\in I_N \subset \overline{I_N} \subseteq (p,r)_{S^1}$.   Since $C_q^{(p,r)_{S^1}}$ on $\LL_1$ is linearly ordered and preserved by $g$, $g(I_N)\subseteq I_N$ or $I_N\subseteq g(I_N)$. Note that $\partial I_N$ is contained in $S^1-Fix_g$. So, $g(\overline{I_N})\subset I_N$ or $\overline{I_N}\subset g(I_N).$ Then, one of two sequence $\{g^k(\ell(I_N))\}_{k=1}^{\infty}$ and $\{g^{-k}(\ell(I_N))\}_{k=1}^{\infty}$ converges to $(p,r)_{S^1}$ on $\LL_1$. Therefore, $(p,r)_{S^1}\in \LL_1$.  
\end{proof}

\begin{lem}\label{sli}
Let $(\LL_1, \LL_2, G)$ be a pseudo-fibered triple and $g$ a nontrivial element of $G$ with $5 \leq |Fix_g|$. Suppose that there is an isolated fixed point $p$ of $g$. If   $I$  is an element in $\LL_1 \cup \LL_2$ such that $p\in v(\ell(I))$ and $|I\cap Fix_g|=1$. Then $I^*$ is isolated.  
\end{lem}
\begin{proof}
Without loss of generality, we can assume that there is such an element $I$ in $\LL_1$.
By Proposition \ref{same type endpoints}, $v(\ell(I))\subset Fix_g$ since $p\in Fix_g$. Say $I=(u,v)_{S^1}$ and denote  the fixed point in $I$ by $q$. By Proposition \ref{epi}, there are two connected components $(x,u)_{S^1}$ and $(v,y)_{S^1}$ of $S^1-Fix_g$. Since $5\leq |Fix_g|$, $x\neq y$.

Suppose that there is a $I^*$-side sequence $\{\ell_n\}_{n=1}^{\infty}$ on $\LL_1$. There is a sequence $\{I_n\}_{n=1}^{\infty}$ on $\LL_1$ such that $\ell_n=\ell(I_n)$ for all $n\in \NN$  and $I^*\subseteq \liminf I_n\subseteq \limsup I_n \subseteq \overline{I^*}$. Since $I^*\subseteq \liminf I_n $, there is $N$ such that $\displaystyle \{x,y\} \subseteq \bigcap_{n=N}^{\infty}I_n$ and since $\limsup I_n\subseteq \overline{I^*}$, there is $N'$ such that  $\displaystyle q \notin \bigcup_{n=N'}^{\infty}I_n$. 

Fix $n$ with $n>\max \{N,N'\}$. Then by the choice of $n$, $\{x,y\} \subseteq I_n  $ and $q \notin I_n$. From now on, we show $I_n \subsetneq I^*$. If $I_n^*\subseteq I^*$, then $q \in I\subseteq I_n$ and it is a contradiction since $q\notin I_n$. If $I^*\subseteq I_n^*$, then $\{x,y\}\subseteq I_n \subseteq I$ and it is a contradiction since $\{x,y\} \subseteq I^*$. Hence, it is $I_n \subseteq I^*$ or $I^*\subseteq I_n$. Then, since $\ell(I)\neq \ell(I_n)$, it is $I_n \subsetneq I^*$ or $I^*\subsetneq I_n$. Therefore, since $\ell(I_n)$ is lies on $I^*$, $I_n \subsetneq I^*$ is the only possible case. 

Say $I_n=(a,b)_{S^1}$. By the assumption, $\{x,y\}\subseteq I_n$. So, we want to show that $[y,x]_{S^1}\subset (a,b)_{S^1}$. Choose $z\in (y,x)_{S^1}$. 
Since $z\in (y,x)_{S^1}$ and $q\in (x,y)_{S^1}$, we get $\phi(y,z,x)=1$ and $\phi(x,q,y)=1$, respectively. So,
it implies that $\phi(x,q,z)=1$ and $\phi(y,z,q)=1$. Since $\{x,y\}\subseteq I_n$ and $q\in I\subseteq I_n^*$, we get $\phi(a,x,b)=1$ and $\phi(b,q,a)=1$, respectively. Then, it implies $\phi(b,q,x)=1$. Likewise, since $\{x,y\}\subseteq I_n$ and $q\in I_n^*$, we get $\phi(a,y,b)=1$ and $\phi(b,q,a)=1$, respectively. Then, it implies $\phi(a,y,q)=1$. Therefore, since $\phi(x,q,z)=\phi(b,q,x)=1$, $\phi(q,z,b)=1$ and since $\phi(y,z,q)=\phi(a,y,q)=1$, $\phi(q,a,z)=1$. Finally, by applying the cocycle condition to four points $(a,z,q,b)$,
$$\phi(z,q,b)-\phi(a,q,b)+\phi(a,z,b)-\phi(a,z,q)=0$$ 
and so 
$$(-1)-(-1)+\phi(a,z,b)-1=0 .$$
Hence, $\phi(a,z,b)=1$ and so we can conclude $(y,x)_{S^1}\subset (a,b)_{S^1}.$ Thus, since $\{x,y\}\subseteq I_n$, 
$[y,x]_{S^1}\subset (a,b)_{S^1}.$

We have shown that $[y,x]_{S^1}\subset (a,b)_{S^1}\subsetneq (v,u)_{S^1}.$ Then $a\in [v,y)_{S^1}$ and $b\in (x,u]_{S^1}$. By Proposition \ref{same type endpoints}, $\{a,b\}=\{u,v\}$ or $a\in (v,y)_{S^1}$ and $b\in (x,u)_{S^1}$. If  $\{a,b\}=\{u,v\}$, it is a contradiction since $(a,b)_{S^1}\subsetneq (v,u)_{S^1}.$ 
Therefore, $a\in (v,y)_{S^1}$ and $b\in (x,u)_{S^1}$. However, by Proposition \ref{disjoint intervals}, it is a contradiction. Thus, $I^*$ is isolated. 
\end{proof} 
     
\begin{lem}\label{finite fixed point}
Let $(\LL_1, \LL_2, G)$ be a pseudo-fibered triple and $g$ a nontrivial element of $G$ with $Fix_g\neq \phi$. Suppose that there is a non-leaf gap $\GG$ of $\LL_i$ for some $i\in \{1,2\}$. If there is an isolated fixed point in $v(\GG)$, then $v(\GG) \subseteq Fix_g$ and for all $I\in \GG$, $|I\cap Fix_g|=1$.
\end{lem}
\begin{proof}
Without loss of generality, we can assume that $\GG$ is a gap on $\LL_1$. Denote the isolated fixed point by $p$. By Proposition \ref{same type endpoints} we can derive $v(\GG)\subseteq Fix_g$. Since $\GG$ is a non-leaf gap, $3\leq |Fix_g|$. Then, for each $J\in \GG $, $1\leq |J\cap Fix_g|$. If not, there is an element $J$ in $\GG$ such that $J\cap Fix_g=\phi$. It implies that $J$ is a connected component of $S^1-Fix_g$. But it is a contradiction by Proposition \ref{end point}. So, we also conclude  $6\leq |Fix_g|$. 

Let us denote $p=x_1$ and $\GG=\{(x_1,x_2)_{S^1},(x_2,x_3)_{S^1}, \cdots, (x_{n-1},x_n)_{S^1}, (x_n,x_1)_{S^1}\}.$ We use $\ZZ_n$ as the index set. First, assume that $x_i$ is an isolated fixed point for some $i\in \ZZ_n$.   Since $x_i$ is isolated, there is the connected component $(x_i,x_i')_{S^1}$. Since $(x_i, x_i')_{S^1}\subseteq (x_i,x_{i+1})_{S^1}$ and $1\leq |(x_i,x_{i+1})_{S^1}\cap  Fix_g|$, so $x_i\neq x_i'$ and $(x_i, x_i')_{S^1}\subsetneq (x_i,x_{i+1})_{S^1}$. By Proposition \ref{epi}, $x_i'$ is also an isolated fixed point and so there is the connected component $(x_i',x_i'')$ of $S^1-Fix_g$. Then there are two cases. One is $(x_i,x_i'')_{S^1}= (x_i,x_{i+1})_{S^1}$ and the other is $(x_i,x_i'')_{S^1}\subsetneq  (x_i,x_{i+1})_{S^1}$. Assume that $(x_i,x_i'')_{S^1}\subsetneq (x_i,x_{i+1})_{S^1}$. By Proposition \ref{end point}, $\{x_i,x_i''\}\subseteq E(\LL_1)$ and $x_i'\in E(\LL_2)$. Then by Theorem \ref{er}, there is a rainbow $\{I_n\}_{n=1}^{\infty}$ at $x_i'$ on $\LL_1$ since $E(\LL_1)\cap E(\LL_2)=\phi$. Since $\displaystyle \bigcap_{n=1}^{\infty}I_n=\{x_i'\}$, there is $I_N$ in $\{I_n\}_{n=1}^{\infty}$ such that $\{x_i,x_i''\}\subseteq I_N^*$. Then $I_N \in C_{x_i'}^{(x_i, x_i'')_{S^1}}$ with $\overline{I_N}\subset (x_i,x_i'')_{S^1}.$ Since $ C_{x_i'}^{(x_i, x_i'')_{S^1}}$ is preserved by $g$ and $v(\ell(I_N))\subseteq S^1-Fix_g$, $g(\overline{I_N})\subset I_N$ or $\overline{I_N}\subset g(I_N)$. So, one of two sequences $\{\ell(g^k(I_N))\}_{k=1}^{\infty}$ and $\{\ell(g^{-k}(I_N))\}_{k=1}^{\infty}$ converges to $(x_i,x_i'')_{S^1}$ on $\LL_1$. Therefore, $(x_i,x_i'')\in \LL_1$. 
By Lemma \ref{sli}, $(x_i, x_i'')_{S^1}^*$ is isolated. Therefore, by Lemma \ref{eog}, there is  a non-leaf gap $\GG'$ of $\LL_1$ which contains $(x_i,x_i'')_{S^1}$, However, by the definition of looseness, it is a contradiction since $\GG'\neq\GG$ and $x_i\in v(\GG)\cap v(\GG')$. Thus, $(x_i,x_i'')_{S^1}= (x_i,x_{i+1})_{S^1}$  is the possible case. So, $(x_i, x_{i+1})_{S^1}$ contains only one fixed point and $x_{i+1}$ is an isolated fixed point by Proposition \ref{epi}. Therefore, since $x_1$ is an isolated fixed point, we are done.

\end{proof}

Let's prove the main theorem. 

\begin{thm}
Let $(\LL_1,\LL_2,G)$ be a pseudo-fibered triple and $g$ a nontrivial element of $G$. Then $|Fix_g|<\infty$. 
\end{thm}
\begin{proof}
It is enough to show the case $5\leq |Fix_g|$. Assume $5\leq |Fix_g|$. Since $g$ is nontrivial, $S^1-Fix_g$ is nonempty and so there is a connected component $I$ of $S^1-Fix_g$ which is nondegenerate open interval. By Proposition \ref{epi}, for each $p\in \partial I$, $p$ is an isolated fixed point. Choose $p\in \partial I$. By Lemma \ref{sl}, there is an element $J$ in $\LL_1\cup \LL_2$ such that $p\in v(\ell(J))$ and $|J\cap Fix_g|=1$. Without loss generality, say $J\in \LL_1$. Then, by Lemma \ref{sli}, $J^*$ is isolated and so by Lemma \ref{eog}, there is a non-leaf gap $\GG$ such that $J\in \GG$. Therefore, by Lemma \ref{finite fixed point}, $v(\GG)\subseteq Fix_g$ and for all $K\in \GG$, $|K\cap Fix_g|=1$. Thus, it implies $|Fix_g|< \infty$.
\end{proof}

\section{Future directions}

We conclude the paper by suggesting some future directions. As we saw
in Corollary \ref{cor: a free group in a tight pair}, if $(\LL, G)$ is
a tight pair, then $G$ contains a nonabelian free subgroup as long as
it does not admit a global fixed point. Indeed, one can show that a
sticky pair has no global fixed point, hence the group of the sticky
pair necessarily contains a nonabelian free subgroup (the proof will
be contained in an upcoming paper of the authors). 
Although it seems more difficult to determine if a slippery pair has no
global fixed point, we propose the following conjecture.

\begin{conj} Suppose $(\LL, G)$ is
a tight pair, then $G$ admits no global fixed point (therefore, it
contains a nonabelian free subgroup). 
\end{conj}

 Another direction is to study further properties of pseudo-fibered
 triples. In \cite{alonso2019laminar}, the following conjectured based
 on observations in \cite{baikthesis, BaikFuchsian} was proposed.

 \begin{conj}[\cite{alonso2019laminar}]
   \label{conj:promotionconj}
   Let $(G, \LL_1, \LL_2)$ be a pseudo-fibered triple. Suppose $G$ is
   finitely generated, torsion-free, and freely indecomposable. Then one of the following
   three possibilities holds:
   \begin{itemize}
   \item[1.] $G$ is virtually abelian.
   \item[2.] $G$ is topologically conjugated into $\PR$.
     \item[3.] $G$ is isomorphic to a closed hyperbolic 3-manifold
       group. 
   \end{itemize}
 \end{conj}

By Theorem \ref{thm:baikmainthm} (or its simplified version), the
second possibility of Conjecture \ref{conj:promotionconj} holds if
there exists the third invariant lamination which is compatible with
other $\LL_1, \LL_2$. The following theorem is a combination of two
main theorems of \cite{alonso2019laminar} on pseudo-fibered triples.

\begin{thm}[\cite{alonso2019laminar}]
\label{thm:ABSmainthm}
  Suppose $G$ is a group as in
  Conjecture \ref{conj:promotionconj}.
  \begin{itemize}
    \item[1.] $G$ satisfies a type of Tits alternative. Namely, each
      subgroup of $G$ either contains a nonabelian free subgroup or is
      virtually abelian.
     \item[2.] If $G$ purely consists of hyperbolic elements, then $G$
       acts on $S^2$ as a convergence group. 
    \end{itemize} 
\end{thm} 

The 2-sphere appears in the second part of Theorem
\ref{thm:ABSmainthm} is obtained as a quotient of the circle on which
the group $G$ acts on. The quotient map is the map collapsing
laminations $\LL_1$ and $\LL_2$ which is analogous to the famous
Cannon-Thurston map constructed in their seminal paper
\cite{CannonThurston}. The study of the induced action on $S^2$ in our work was largely influenced by Fenley's work \cite{fenley2012ideal}.

One strategy to achieve the third possibility of Conjecture
\ref{conj:promotionconj} is first strengthening the second part of Theorem
\ref{thm:ABSmainthm}. Namely, one may try to show that if $G$ contains
both hyperbolic and non-hyperbolic elements, then $G$ acts on $S^2$ as
a uniform convergence group. Then by a theorem of Bowditch
\cite{Bowditch98}, $G$ is word-hyperbolic and $S^2$ is equivariantly
homeomorphic to its boundary. Hence, if one can prove the Cannon's
conjecture in this setting, one ends up with the third possibility of Conjecture
\ref{conj:promotionconj}. Perhaps as an intermediate step, one may try
the following conjecture.

\begin{conj} Suppose $G$ is a group as in
  Conjecture \ref{conj:promotionconj}, and assume $G$ is not virtually
  abelian. Then $G$ is word-hyperbolic. 
\end{conj} 

\bibliographystyle{abbrv}
\bibliography{biblio}
\end{document}